\documentclass[11pt,reqno]{amsart}
\usepackage{amsaddr}
\usepackage{hyperref}
\usepackage{graphicx}
\usepackage{amssymb}
\usepackage{amsfonts}
\usepackage[]{amsmath}
\usepackage[]{epsfig}
\usepackage[]{pstricks}
\newpsobject{malla}{psgrid}{subgriddiv=1,griddots=10,gridlabels=6pt}
\usepackage[]{float}
\usepackage{setspace}
\newpsobject{malla}{psgrid}{subgriddiv=1,griddots=10,gridlabels=6pt}
\newtheorem{theorem}{Theorem}[section]
\newtheorem{lemma}{Lemma}[section]

\newtheorem{remark}{Remark}[section]
\newtheorem{proposition}{Proposition}[section]

\numberwithin{equation}{section}
\newcommand{\R}{\mathbb R}

\newcommand{\T}{{\mathbb T}}

\newcommand{\Z}{{\mathbb Z}}
\newcommand{\ft}{{\mathcal{F}}}

\newcommand{\Sch}{{\mathcal{S}}}
\newcommand{\supp}{{\mbox{supp}}}

\newcommand{\px}{\partial_x}
\newcommand{\pt}{\partial_t}

\def\norm#1{\|#1\|}
\def\bra#1{\langle#1\rangle}
\def\wt#1{\widetilde{#1}}
\def\wh#1{\widehat{#1}}
\def\set#1{\{#1\}}

\begin{document}

	\pagenumbering{arabic}	
\title[IBVP for the fifth-order KdV-type equations]{Local well-posedness of the fifth-order KdV-type equations on the half-line}

\author[M. Cavalcante]{M\'arcio Cavalcante}

\address{\emph{Instituto de Matem\'{a}tica, Universidade Federal de Alagoas\\ Macei\'o AL-Brazil}}
\email{marcio.melo@im.ufal.br}

\author[C. Kwak]{Chulkwang Kwak}

\address{\emph{Facultad de Matem\'{a}ticas, Pontificia Universidade Catolica de Chile\\ Santiago-Chile}}
\email{chkwak@mat.uc.cl}
\thanks{C. Kwak is supported by FONDECYT de Postdoctorado 2017 Proyecto No. 3170067.}

\begin{abstract}
This paper is a continuation of authors' previous work \cite{CK2018-1}. We extend the argument \cite{CK2018-1} to fifth-order KdV-type equations with different nonlinearities, in specific, where the scaling argument does not hold. We establish the $X^{s,b}$ nonlinear estimates for $b < \frac12$, which is almost optimal compared to the standard $X^{s,b}$ nonlinear estimates for $b > \frac12$ \cite{CGL2010, JH2009}. As an immediate conclusion, we prove the local well-posedness of the initial-boundary value problem (IBVP) for fifth-order KdV-type equations on the right half-line and the left half-line. 

%following the ideas developed by Homer and Cavalcante.
\end{abstract}
\maketitle
%%%%%%%%%%%%%%%%%
%%%%%%%%%%%%%%%%%%%%%%%%%%%%%%%%%%%%%%%%%%%%%%%%%%%%%%%%%%%%%%%%%%%%%%%%%%%
%\tableofcontents

\tableofcontents

\section{Introduction}\label{sec:intro}
This paper is a continuation of authors' previous work \cite{CK2018-1}. In \cite{CK2018-1}, the authors studied the \emph{Duhamel boundary forcing} operator associated to the fifth-order linear operator, and established the local well-posedness of Kawahara equation posed on the right/left half-line. In this paper, we extend the previous study to the fifth-order KdV-type equations whose nonlinearities are different, in particular, do not satisfy the scaling symmetry. The lack of the scaling invariance cause an additional analysis on time trace estimates with a cutoff function supported on $|t| \le T$, and to provide such an analysis is one of aims of this work. Consider the following fifth-order KdV-type equation:
\begin{equation}\label{eq:5kdv}
\pt u - \px^5 u + F(u)=0,
\end{equation}
where $u(t,x)$ is real-valued function and $F(u)$ is a nonlinearity. We, here, take $F(u) = (1-\px^2)^{\frac12}\px (u^2)$ or $F(u) = \px(u^3)$. 

When $F(u) = (1-\px^2)^{\frac12}\px (u^2)$, the equation \eqref{eq:5kdv} was introduced by Tina, Gui and Liu \cite{TGL2008} to understand the role of dispersive and nonlinear convection effects in the fifth-order $K(m,n,p)$ equations of the form
\[\pt u + \beta_1\px(u^m) + \beta_2\px^3(u^n) + \beta_3\px^5(u^p) = 0\]
when $(m,n,p) = (2,2,1)$, in particular, $\beta_1=1$ and $\beta_2= \beta_3 = -1$.

When $F(u) = \px(u^3)$, the equation \eqref{eq:5kdv} is well-known as the modified Kawahara equation, which was proposed first by Kawahara \cite{Kawahara1972}. The modified Kawahara equation arises in the theory of shallow water waves, the theory of magneto-acoustic waves in plasmas and propagation of nonlinear water-waves in the long-wavelength region as in the case of KdV equations. This equation is also regarded as a singular perturbation of KdV equation. We refer to \cite{AS1985, HS1988, GJ1995} and references therein for more background informations.

\subsection{Main analysis} 
The principal contribution in the paper is to establish the nonlinear estimates for both nonlinearities $F(u) = (1-\px^2)^{\frac12}\px (u^2)$ and $F(u) = \px (u^3)$, in particular, 
\begin{equation}\label{Main_Non}
\norm{F(u)}_{X^{s,-b} \cap Y^{s,-b}} \lesssim \norm{u}_{X^{s,b} \cap D^{\alpha}}^k, \qquad k=2,3,
\end{equation}
for a certain regularity $s \in \R$, $0 < b < \frac12 < \alpha < 1-b$. The functions spaces used in \eqref{Main_Non} are the standard $X^{s,b}$ space\footnote{It is called Bourgain's space or dispersive Sobolev space.} equipped with the norm
\begin{equation}\label{Xsb}
\norm{f}_{X^{s,b}} = \norm{\bra{\xi}^s\bra{\tau-\xi^5}^b \wt{f}}_{L_{\tau,\xi}^2},
\end{equation}
where $\wt{f}$ is the space time Fourier coefficient (also denoted by $\ft(f)$) and $\bra{\cdot} = (1+|\cdot|^2)^{\frac12}$, initially introduced in its current form by Bourgain \cite{Bourgain1993}, and its various modifications, see below for a short explanation of spaces and Section \ref{sec:sol space} for precise definitions of spaces.

The $X^{s,b}$ space is known to be an appropriate device to detect dispersive phenomena in the Fourier analysis. In other words, solutions to \eqref{eq:5kdv} have the \emph{dispersive smoothing effect} which means that the (space-time) Fourier coefficients decay far away from the characteristic curve $\tau = \xi^5$, where $\xi$ and $\tau$ are the Fourier variables corresponding to $x$ and $t$, respectively. This property is naturally reflected in \eqref{Xsb} as a weight $\bra{\tau-\xi^5}^b$. The choice of the exponent $b < \frac12$ in \eqref{Xsb} is imposed in the study on IBVP due to the presence of the Duhamel boundary forcing operator, which reflects what role boundary conditions play in the solutions (see Section \ref{sec:Duhamel boundary forcing operator} for more details), while the standard $X^{s,b}$ space with $b > \frac12$ works well in the study on the initial value problem (IVP). 

The trade-off of choosing $b < \frac12$ causes the lack of $\tau$-integrability in \eqref{Main_Non}, when all functions are localized in the frequency support $|\xi| \le 1$. In order to resolve this problem, an additional \emph{low frequency} localized space $D^{\alpha}$, $\alpha > \frac12$, is needed, i.e., \eqref{Xsb} under the restriction $|\xi| \le 1$, precisely,
\[\norm{f}_{D^{\alpha}} = \norm{\bra{\tau}^b \chi_{\set{|\xi| \le 1}}(\xi)\wt{f}}_{L_{\tau,\xi}^2}.\]
On the other hand, time trace estimates of the Duhamel parts (Lemma \ref{duhamel}) in $X^{s,b}$-type spaces hold true for only positive regularities (see Remark \ref{rem:Ysb}), thus an additional introduction of the (time-adapted) Bourgain space $Y^{s,b}$ as an intermediate norm in the iteration process, which is defined similarly as the standard $X^{s,b}$ space but with a weight in terms of $\tau$ instead of $\xi$ in the sense of $\pt \sim \px^5$ in \eqref{eq:5kdv} (replacing $\bra{\xi}^s$ by $\bra{\tau}^{\frac{s}{5}}$ in \eqref{Xsb}), precisely
\[\norm{f}_{Y^{s,b}} = \norm{\bra{\tau}^{\frac{s}{5}}\bra{\tau-\xi^5}^b \wt{f}}_{L_{\tau,\xi}^2},\]
is necessary to cover the negative regularities. 

The followings are the main results established in the paper.
\begin{theorem}\label{thm:nonlinear1}   $\;$

(a) For $-5/4 < s $, there exists $b = b(s) < 1/2$ such that for all $\alpha > 1/2$, we have
\begin{equation}\label{eq:nonlinear1}
\norm{(1-\px^2)^{\frac12}\px (uv)}_{X^{s,-b}} \lesssim \norm{u}_{X^{s,b} \cap D^{\alpha}}\norm{v}_{X^{s,b} \cap D^{\alpha}}.
\end{equation}

(b) For $-5/4 < s \le 0$, there exists $b = b(s) < 1/2$ such that for all $\alpha > 1/2$, we have
	\begin{equation}\label{eq:nonlinear1.1}
	\norm{(1-\px^2)^{\frac12}\px (uv)}_{Y^{s,-b}} \lesssim \norm{u}_{X^{s,b} \cap D^{\alpha}}\norm{v}_{X^{s,b} \cap D^{\alpha}}.
	\end{equation}

The implicit constants in \eqref{eq:nonlinear1} and \eqref{eq:nonlinear1.1} depend only on $s$, $b$, ans $\alpha$.	
\end{theorem}

\begin{theorem}\label{thm:nonlinear2}  $\;$

(a) For $-1/4 \le s$, there exists $b = b(s) < 1/2$ such that for all $\alpha > 1/2$, we have
\begin{equation}\label{eq:nonlinear2}
\norm{\px(uvw)}_{X^{s,-b} } \lesssim \norm{u}_{X^{s,b} \cap D^{\alpha}}\norm{v}_{X^{s,b} \cap D^{\alpha}}\norm{w}_{X^{s,b} \cap D^{\alpha}}.\end{equation}
	
(b) For $-1/4 \le s \le 0$, there exists $b = b(s) < 1/2$ such that for all $\alpha > 1/2$, we have
	\begin{equation}\label{eq:nonlinear2.1}
	\norm{\px(uvw)}_{Y^{s,-b} } \lesssim \norm{u}_{X^{s,b} \cap D^{\alpha}}\norm{v}_{X^{s,b} \cap D^{\alpha}}\norm{w}_{X^{s,b} \cap D^{\alpha}}.
	\end{equation}
	The implicit constants in \eqref{eq:nonlinear2} and \eqref{eq:nonlinear2.1} depend only on $s$, $b$, ans $\alpha$.
\end{theorem}

\begin{remark}
Theorems \ref{thm:nonlinear1} and \ref{thm:nonlinear2} are almost sharp compared with \cite{CGL2010} and \cite{JH2009}, respectively, in the low regularity sense (in particular, negative regularity). We also refer to \cite{YL2012} for the weak ill-posedness result for the modified Kawahara equation in $H^s(\R)$, $s < -\frac14$.
\end{remark}

\begin{remark}
Both \eqref{eq:nonlinear1.1} and \eqref{eq:nonlinear2.1} can be obtain in some positive regularity regime similarly as Proposition 5.2 in \cite{CK2018-1}, but we, here, explore the nonlinear estimates for both $F(u) = (1-\partial_x^2)^{\frac12}\partial_x(u^2)$ and $F(u) = \partial_x(u^3)$ in $Y^{s,b}$ only in the negative regularity regime, since the intermediate norm $Y^{s,b}$ occurs in the Picard iteration mechanism only in the negative regularity regime (see the proof of Lemma \ref{duhamel} (b)).
\end{remark}

The proof of Theorems \ref{thm:nonlinear1} (a) and \ref{thm:nonlinear2} (a) are based on the harmonic analysis technique, in particular the Tao’s $[k; Z]$-multiplier norm method \cite{Tao2001}, which has become now standard to prove the multilinear estimates. Precisely, let $P_k$ be the Littlewood-Paley projection operator in terms of the frequency on a support $|\xi| \sim 2^k$, $k \in \Z$ (will be precisely defined in Section \ref{sec:pre}) and $f_k = P_k f$. Let further decompose $f_k$ into $\wt{f}_{k,j} = \eta_j(\tau - \xi^5)\wt{f}_k$ pieces, where $\eta_j(\zeta)$ is a smooth bump function supported in $|\zeta| \sim 2^j$. Then, the Littlewood-Paley decomposition (not only in terms of the frequency $|\xi|$, but also with respect to the modulation $|\tau-\xi^5|$) allows us to separate the left-hand side of \eqref{eq:nonlinear1}--\eqref{eq:nonlinear2.1} into each frequency part, for instance in \eqref{eq:nonlinear1}
\[\norm{(1-\px^2)^{\frac12}\px (uv)}_{X^{s,-b}} \lesssim \sum_{k,k_1,k_2 \in \Z}\max(1,2^{(1+s)k})2^k \norm{P_k(u_{k_1}v_{k_2})}_{X^{0,-b}}\]
and
\[\norm{P_k(u_{k_1}v_{k_2})}_{X^{0,-b}} \sim \sum_{j,j_1,j_2 \ge 0} 2^{-bj}\norm{\eta_j(\tau-\xi^5)\ft(P_k(u_{k_1,j_1}v_{k_2,j_2}))}_{L_{\tau,\xi}^2}.\]
The quadratic nonlinearity consists of \emph{high $\times$ low $\Rightarrow$ high}, \emph{high $\times$ high $\Rightarrow$ high} and \emph{high $\times$ high $\Rightarrow$ low} cases\footnote{The \emph{low $\times$ low $\Rightarrow$ low} interaction case can be dealt with by the trivial bound, hence we do not comment it on here.} (in terms of the relation among frequencies), and the cubic nonlinearity consists of more cases with various frequency relations, see the proof of Proposition \ref{prop:tri1}. Thus, the task is reduced to prove multilinear estimates of each piece in $L^2$. Such $L^2$-block estimates have already been provided by Chen, Li, Miao and Wu \cite{CLMW2009} for the bilinear case (Chen and Guo \cite{CG2011} corrected the \emph{high $\times$ high $\Rightarrow$ high} case), and by the second author \cite{Kwak2015} for the trilinear case (in \cite{Kwak2015}, the $L^2$-block estimates for periodic functions in the spatial variable are given, but the proof for non-periodic functions is analogous). Performing $L^2$ estimates and gathering all pieces, one reaches the right-hand side of \eqref{eq:nonlinear1}--\eqref{eq:nonlinear2.1}. 

A direct proof of trilinear estimates \eqref{eq:nonlinear2} is given, while bilinear estimates and $TT^{\ast}$ argument are used to prove the trilinear estimates in \cite{CGL2010}. Moreover, the Strichartz estimate (with derivative gains) for the linear operator group $\{e^{t\px^5}\}$ \cite{CT2005} is needed to deal with \emph{high $\times$ high $\times$ high $\Rightarrow$ high} interaction component, since the trilinear $L^2$-block estimates (Lemma \ref{lem:tri-L2}) are given for only \emph{high $\times$ low $\times$ low $\Rightarrow$ high} case. 

The proof of Theorems \ref{thm:nonlinear1} (b) and \ref{thm:nonlinear2} (b) are based on the proof of Lemma 5.10 (b) in \cite{Holmerkdv}, but more careful examination of frequency relations is needed (also for the proof of Theorems \ref{thm:nonlinear1} (a) and \ref{thm:nonlinear2} (a)).

\subsection{IBVP problems setting} The IBVP of \eqref{eq:5kdv}, here, is studied as an application of Theorems \ref{thm:nonlinear1} and \ref{thm:nonlinear2}. The precise IBVP of \eqref{eq:5kdv} on the right/left half-lines is set as follows:
\begin{equation}\label{kawahararight}
\begin{cases}
\pt u - \px^5 u + F(u)=0, & (t,x)\in (0,T)  \times(0,\infty),\\
u(0,x)=u_0(x),                                   & x\in(0,\infty),\\
u(t,0)=f(t),\ u_x(t,0)=g(t)& t\in(0,T)
\end{cases}
\end{equation}
and
\begin{equation}\label{kawaharaleft}
\begin{cases}
	\pt u - \px^5 u + F(u)=0, & (t,x)\in  (0,T) \times (-\infty,0),\\
	u(0,x)=u_0(x),                                   & x\in(-\infty,0),\\
	u(t,0)=f(t),\ 	u_x(t,0)=g(t),\ u_{xx}(t,0)=h(t)& t\in(0,T).
\end{cases}
\end{equation}
The number of boundary conditions in \eqref{kawahararight} and \eqref{kawaharaleft} is inspired from the uniqueness issue arising in a direct calculation of $L^2$ integral identities for linear equations:
\begin{equation}\label{unico1}
\begin{split}
	\int_0^{\infty}u^2(T,x)dx=&\int_0^{\infty}u^2(0,x)dx -  \int_0^T(\partial_x^2u)^2(t,0) dt + 2\int_0^T\partial_x^3u(t,0)\partial_xu(t,0)dt\\
&	-2\int_0^T\partial_x^4u(t,0)u(t,0)dt
	\end{split}
\end{equation}
and
\begin{equation}\label{unico2}
\begin{split}
	\int_{-\infty}^{0}u^2(T,x)dx=&\int_{-\infty}^{0}u^2(0,x)dx + \int_0^T(\partial_x^2u)^2(t,0) dt - 2\int_0^T\partial_x^3u(t,0)\partial_xu(t,0)dt\\
&	+2\int_0^T\partial_x^4u(t,0)u(t,0)dt.
	\end{split}
\end{equation}
We refer to \cite{Holmerkdv, CK2018-1} for more expositions.

The \emph{local smoothing effect} \cite{KPV1991} 
\[\|\partial_x^je^{t\partial_x^5}\phi\|_{L_x^{\infty}\dot{H}^{\frac{s+2-j}{5}}(\mathbb{R}_t)}\leq c\|\phi\|_{\dot{H}^{s}(\mathbb{R})},\ \text{for}\ j=0,1,2,\]
stipulates the appropriate spaces for the initial and boundary data, thus the initial and boundary data for \eqref{kawahararight} and \eqref{kawaharaleft} satisfy
\begin{equation}\label{regularidade}
u_0\in H^s(\mathbb{R}^+),\; f(t)\in H^{\frac{s+2}{5}}(\mathbb{R}^+)\; \text{and}\;  g(t)\in H^{\frac{s+1}{5}}(\mathbb{R}^+)
\end{equation}
and
\begin{equation}\label{regularidadeleft}
u_0\in H^s(\mathbb{R}^-), \; f(t)\in H^{\frac{s+2}{5}}(\mathbb{R}^+),\; g(t)\in H^{\frac{s+1}{5}}(\mathbb{R}^+)\; \text{and}\;  h(t)\in H^{\frac{s}{5}}(\mathbb{R}^+),
\end{equation}
respectively. On the other hand, the compatibility conditions in high regularities, for instance, $\frac12 < s < \frac32$ or $\frac32 < s < \frac52$..., are required to be considered as follows, for instance,
\[u_0(0) = f(0), \quad \mbox{if } \; \frac12 < s < \frac32, \quad u_0(0) = f(0), \partial_x u_0(0) = g(0) \quad \mbox{if } \; \frac32 < s\]
for \eqref{kawahararight}. However, our local well-posedness results for both \eqref{kawahararight} and \eqref{kawaharaleft} are valid only in the regularity region $s<\frac12$ (see theorems \ref{theorem1}--\ref{theorem22} below), and hence the compatibility conditions for high regularities are negligible. See \cite{CK2018-1} for the comparison.

Theorems \ref{thm:nonlinear1} and \ref{thm:nonlinear2} in addition to the standard argument used in \cite{CK2018-1} immediately imply the local well-posedness of the IBVP for \eqref{eq:5kdv} on the right half-line. We state theorems separately for the sake of reader's convenience.
\begin{theorem}\label{theorem1}
Let $s \in (-\frac54,\frac12)$. For given initial-boundary data $(u_0,f,g)$ satisfying \eqref{regularidade}, there exist a positive time $T>0$ depending on $\|u_0\|_{H^s(\mathbb{R}^+)}$, $\|f\|_{H^{\frac{s+2}{5}}(\mathbb{R}^+)}$ and $\|g\|_{H^{\frac{s+1}{5}}(\mathbb{R}^+)}$, and a solution $u(t,x) \in C((0 , T);H^s(\R^+))$ to \eqref{kawahararight}-\eqref{regularidade} with the nonlinearity $F(u) = (1-\px^2)^{\frac12}\px (u^2)$ satisfying
\[u \in C\bigl(\mathbb{R}^+;\; H^{\frac{s+2}{5}}(0,T)\bigr) \cap X^{s,b}((0,T) \times \R^+) \cap D^{\alpha}((0,T) \times \R^+) \; \mbox{ and } \; \partial_xu\in C\bigl(\R^+;\; H^{\frac{s+1}{5}}(0,T)\bigr)\]
for some $b(s) < \frac12$ and $\alpha(s) > \frac12$. Moreover, the map $(u_0,f,g)\longmapsto u$ is analytic from $H^s(\mathbb{R}^+)\times H^{\frac{s+2}{5}}(\mathbb{R}^+)\times H^{\frac{s+1}{5}}(\mathbb{R}^+)$ to  $C\big((0,T);\,H^s(\mathbb{R}^+)\big)$.
\end{theorem}

\begin{theorem}\label{theorem2}
Let $s \in [-\frac14,\frac12)$. For given initial-boundary data $(u_0,f,g)$ satisfying \eqref{regularidade}, there exist a positive time $T>0$ depending on $\|u_0\|_{H^s(\mathbb{R}^+)}$, $\|f\|_{H^{\frac{s+2}{5}}(\mathbb{R}^+)}$ and $\|g\|_{H^{\frac{s+1}{5}}(\mathbb{R}^+)}$, and a solution $u(t,x) \in C((0 , T);H^s(\R^+))$ to \eqref{kawahararight}-\eqref{regularidade} with the nonlinearity $F(u) = \px (u^3)$ satisfying
\[u \in C\bigl(\mathbb{R}^+;\; H^{\frac{s+2}{5}}(0,T)\bigr) \cap X^{s,b}((0,T) \times \R^+) \cap D^{\alpha}((0,T) \times \R^+) \; \mbox{ and } \; \partial_xu\in C\bigl(\R^+;\; H^{\frac{s+1}{5}}(0,T)\bigr)\]
for some $b(s) < \frac12$ and $\alpha(s) > \frac12$. Moreover, the map $(u_0,f,g)\longmapsto u$ is analytic from $H^s(\mathbb{R}^+)\times H^{\frac{s+2}{5}}(\mathbb{R}^+)\times H^{\frac{s+1}{5}}(\mathbb{R}^+)$ to  $C\big((0,T);\,H^s(\mathbb{R}^+)\big)$.
\end{theorem}

Moreover, we have the local well-posedness of the IBVP for \eqref{eq:5kdv} on the left half-line.
\begin{theorem}\label{theorem12}
Let $s \in (-\frac54,\frac12)$. For given initial-boundary data $(u_0,f,g,h)$ satisfying \eqref{regularidadeleft}, there exist a positive time $T$ depending on $\|u_0\|_{H^s(\mathbb{R}^-)}$, $\|f\|_{H^{\frac{s+2}{5}}(\mathbb{R}^+)}$, $\|g\|_{H^{\frac{s+1}{5}}(\mathbb{R}^+)}$ and $\|h\|_{H^{\frac{s}{5}}(\mathbb{R}^+)}$, and a solution $u(t,x) \in C((0 , T);H^s(\R^-))$ to \eqref{kawaharaleft}-\eqref{regularidadeleft} with the nonlinearity $F(u) = (1-\px^2)^{\frac12}\px (u^2)$ satisfying
\[\begin{aligned}
&u \in C\bigl(\mathbb{R}^-;\; H^{\frac{s+2}{5}}(0,T)\bigr) \cap X^{s,b}((0,T) \times \R^-) \cap D^{\alpha}((0,T) \times \R^-), \\
&\partial_xu\in C\bigl(\R^-;\; H^{\frac{s+1}{5}}(0,T)\bigr) \; \mbox{ and } \; \partial_x^2u\in C\bigl(\R^-;\; H^{\frac{s}{5}}(0,T)\bigr)
\end{aligned}\]
for some $b(s) < \frac12$ and $\alpha(s) > \frac12$. Moreover, the map $(u_0,f,g,h)\longmapsto u$ is analytic from $H^s(\mathbb{R}^-)\times H^{\frac{s+2}{5}}(\mathbb{R}^+)\times H^{\frac{s+1}{5}}(\mathbb{R}^+) \times H^{\frac{s}{5}}(\mathbb{R}^+)$ to  $C\big((0,T);\,H^s(\mathbb{R}^-)\big)$.
\end{theorem}

\begin{theorem}\label{theorem22}
Let $s \in [-\frac14,\frac12)$. For given initial-boundary data $(u_0,f,g,h)$ satisfying \eqref{regularidadeleft}, there exist a positive time $T$ depending on $\|u_0\|_{H^s(\mathbb{R}^-)}$, $\|f\|_{H^{\frac{s+2}{5}}(\mathbb{R}^+)}$, $\|g\|_{H^{\frac{s+1}{5}}(\mathbb{R}^+)}$ and $\|h\|_{H^{\frac{s}{5}}(\mathbb{R}^+)}$, and a solution $u(t,x) \in C((0 , T);H^s(\R^-))$ to \eqref{kawaharaleft}-\eqref{regularidadeleft} with the nonlinearity $F(u) = \px (u^3)$ satisfying
\[\begin{aligned}
&u \in C\bigl(\mathbb{R}^-;\; H^{\frac{s+2}{5}}(0,T)\bigr) \cap X^{s,b}((0,T) \times \R^-) \cap D^{\alpha}((0,T) \times \R^-), \\
&\partial_xu\in C\bigl(\R^-;\; H^{\frac{s+1}{5}}(0,T)\bigr) \; \mbox{ and } \; \partial_x^2u\in C\bigl(\R^-;\; H^{\frac{s}{5}}(0,T)\bigr)
\end{aligned}\]
for some $b(s) < \frac12$ and $\alpha(s) > \frac12$. Moreover, the map $(u_0,f,g,h)\longmapsto u$ is analytic from $H^s(\mathbb{R}^-)\times H^{\frac{s+2}{5}}(\mathbb{R}^+)\times H^{\frac{s+1}{5}}(\mathbb{R}^+) \times H^{\frac{s}{5}}(\mathbb{R}^+)$ to  $C\big((0,T);\,H^s(\mathbb{R}^-)\big)$.
\end{theorem}

\begin{remark}
The proof of Theorems \ref{theorem1}--\ref{theorem22} in Section \ref{sec:thm proof} claims that extension of solutions are unique under an auxiliary condition: the solution $u(t,x)$ as in \eqref{eq:solution map} is unique in $Z_{\ell}^{s,b,\alpha}$, or it just guarantees the existence of a weak solution (a formulation of the integral equation). However, an analogous argument in \cite[Section 4]{BSZ2006}\footnote{The argument introduced in \cite{BSZ2006} works well to this problem, even if the argument is concerned with the KdV equation.}, in particular Proposition 4.13, Corollary 4.14 and Proposition 4.15, shows that weak solutions obtained in Theorems \ref{theorem1}--\ref{theorem22} are mild solutions and thus they are unique. The precise definitions and uniqueness of mild solutions are given in \cite{BSZ2006}, and we do not pursue it here. 
\end{remark}

The proof of Theorems \ref{theorem1}--\ref{theorem22} relies on the argument introduced in Colliander-Kenig work \cite{CK}, which was further developed by Holmer \cite{Holmerkdv}, i.e., the Fourier restriction norm method for "a suitable extension of solutions" ensures the local well-posedness. Precisely, the Duhamel boundary forcing operator, introduced in \cite{CK}, corresponding to fifth-order KdV operator enables us to successfully construct solutions in $\R_x$ (or extend solutions to \eqref{eq:5kdv} posed on the (right/left) half-line to solutions defined in whole line $\R$, see Section \ref{sec:Duhamel boundary forcing operator}). In other words, the IBVP of \eqref{eq:5kdv} is converted to the IVP of \eqref{eq:5kdv} (integral equation formula).  This work has been done in our previous work \cite{CK2018-1}. After this procedure, we follows the standard iteration method in addition to the energy and nonlinear estimates (will be established in Sections \ref{sec:energy} and \ref{sec:nonlinear}, respectively) to show the local well-posedness of IVP of \eqref{eq:5kdv}. The new ingredients here are the multilinear estimates for $F(u) = (1-\px^2)^{\frac12}\px (u^2)$ and $F(u) = \px(u^3)$ presented in Theorems \ref{thm:nonlinear1} and \ref{thm:nonlinear2}.

When $F(u) = (1-\px^2)^{\frac12}\px(u^2)$, the equation \eqref{eq:5kdv} does not admit the scaling symmetry due to the presence of the nonlocal operator $(1-\px^2)^{\frac12}$ in the nonlinearity. When $F(u) = \px(u^3)$, on the other hand, it is known that the modified Kawahara equation admits the scaling symmetry: if $u$ is a solution to \eqref{eq:5kdv}, $u_{\lambda}$ defined by
\[u_{\lambda}(t,x) := \lambda^2u(\lambda^5t,\lambda x), \qquad \lambda > 0\]
is a solution to \eqref{eq:5kdv} as well. A straightforward calculation gives
\begin{equation}\label{eq:scaling}
\begin{aligned}
\norm{u_{0,\lambda}}_{H^s} &+ \norm{f_{\lambda}}_{H^{\frac{s+2}{5}}} + \norm{g_{\lambda}}_{H^{\frac{s+1}{5}}} + \norm{h_{\lambda}}_{H^{\frac{s}{5}}} \\
=&\lambda^{\frac32}\bra{\lambda}^s\norm{u_0}_{H^s} + \lambda^{-\frac12}\bra{\lambda}^{s+2}\norm{f}_{H^{\frac{s+2}{5}}} + \lambda^{\frac12}\bra{\lambda}^{s+1}\norm{g}_{H^{\frac{s+1}{5}}} + \lambda^{\frac32}\bra{\lambda}^{s}\norm{f}_{H^{\frac{s}{5}}},
\end{aligned}
\end{equation}
which reveals that the $\lambda$-scaled initial and boundary data cannot be small in some sense at the same time due to $\lambda^{-\frac12}\bra{\lambda}^{s+2}\norm{f}_{H^{\frac{s+2}{5}}}$ term in the right-hand side of \eqref{eq:scaling} for $0 < \lambda \ll 1$. Thus, the scaling-rescaling argument no longer applies to the IBVP of \eqref{eq:5kdv} with nonlinearity both  $F(u) = (1-\px^2)^{\frac12}\px(u^2)$ and $F(u) = \px(u^3)$.

To study the IBVP of \eqref{eq:5kdv} for arbitrary initial and boundary data, the energy estimate in a short time interval $[0,T]$, $T \ll 1$, is needed, in particular, (derivatives) time trace estimates. However, the time localized cut-off function is no longer free in the time trace norm (Lemmas \ref{grupo} (b),  \ref{duhamel} (b) and \ref{edbf} (b)), thus, the regularity threshold, in particular the upper-bound of regularity ($s<\frac12$) in Theorems \ref{theorem1}--\ref{theorem22}, is restricted by the time trace norm estimates. See Section \ref{sec:energy} for the details.  

\subsection{Review on the well-posedness results.} 
Both the IVP and the IBVP of the fifth-order KdV-type equations have been extensively studied. When $F(u) = (1-\px^2)^{\frac12}\px(u^2)$, the local well-posedness of \eqref{eq:5kdv} was first established by Tina, Gui and Liu \cite{TGL2008} in $H^s(\R)$, $s \ge -11/16$ by using the Fourier restriction norm method \cite{Bourgain1993}. In addition to the technique \emph{Tao's $[K;Z]$ multiplier norm method} \cite{Tao2001}, Chen and Liu \cite{CL2010} improved the local well-posedness in $H^s(\R)$, $s > -5/4$ and they also showed the ill-posedness, in the sense of the lack of continuity of the flow map, for $s < -5/4$. At the endpoint regularity $H^{-5/4}(\R)$, Chen, Guo and Liu \cite{CGL2010} proved the local well-posedness by using Besov-type function spaces. This is the optimal result until now as far as authors' know.

When $F(u) = \px(u^3)$, the Cauchy problem for \eqref{eq:5kdv} was studied by Jia and Huo \cite{JH2009} and Chen, Li, Miao and Wu \cite{CLMW2009}, independently. They established the local well-posedness in $H^s(\R)$, $s \ge -1/4$, by using the Fourier restriction norm method. In \cite{CLMW2009}, the authors used bilinear $L^2$-block estimates and $TT*$ argument to prove the trilinear estimate, while a direct calculation of trilinear integral operator was performed in\cite{JH2009}. The global well-posedness of \eqref{eq:5kdv} in $H^s(\R)$, $s>-3/22$ was shown by Yan, Li and Yang \cite{YLY2011} via the \emph{I-method} \cite{CKSTT2003}.

The IBVP of the modified Kawahara equation posed on the right half-line in the high regularity Sobolev space $H^s(\R^+)$ ($\frac14 \le s < 2$) has been studied by Tao and Lu \cite{TL2007}. On the other hand, the IBVP of \eqref{eq:5kdv} on the half-lines (both right and left), where the nonlinearity is given by both $F(u) = (1-\px^2)^{\frac12}\px(u^2)$ and $F(u) = \px(u^3)$, in the low regularity setting (in particular, negative regularities) is first considered here as far as we know. We end this section with referring to \cite{San2003, Lar1, Fam2009, Lar2, CK2018-1} and references therein for the IBVP results of the fifth-order KdV-type equations posed on the half-line.

\subsection{Organization of the paper} 
The rest of paper is organized as follows: In Section \ref{sec:pre}, we mainly construct the solution space and observe several basic properties for the IBVP of \eqref{eq:5kdv}. In Section \ref{sec:nonlinear}, we give the proofs of Theorems \ref{thm:nonlinear1} and \ref{thm:nonlinear2}. In Section \ref{sec:Duhamel boundary forcing operator}, we briefly introduce the Duhamel boundary forcing operator for the fifth-order equations. In Section \ref{sec:energy}, we establish energy estimates, in particular time trace estimate, with a short time cut-off function. In Sections \ref{sec:thm proof}, we prove Theorems \ref{theorem1} -- \ref{theorem22}.

\subsection*{Acknowledgments}
Authors appreciate anonymous referee(s) for a careful reading and helpful suggestions.

\section{Preliminaries}\label{sec:pre}
Let $\R^+ = (0,\infty)$. For positive real numbers $x,y \in \R^+$, we mean $x \lesssim y$ by $x \le Cy$ for some $C>0$. Also, $x \sim y$ means $x \lesssim y$ and $y\lesssim x$. Similarly, $\lesssim_a$ and $\sim_a$ can be defined, where the implicit constants depend on $a$. 

For a cut-off function $\psi$ given by
\begin{equation}\label{eq:cutoff0}
\psi \in C_0^{\infty}(\mathbb{R}) \quad \mbox{such that} \quad 0 \le \psi \le1, \quad  \psi \equiv 1 \; \mbox{ on } \; [-1,1], \quad \psi \equiv 0, \; |t| \ge 2,
\end{equation}
we fix the time localized function 
%\begin{equation}\label{eq:cutoff}
\[\psi_T(t) = \psi (t/T), \qquad 0 < T < 1.\]
%\end{equation}

\subsection{Riemann-Liouville fractional integral}
A brief summary of the Riemann-Liouville fractional integral operator is, here, given, see \cite{CK, Holmerkdv} for more details. Let $t_+$ be a function defined by
\[t_+ = t \quad \mbox{if} \quad t > 0, \qquad t_+ = 0  \quad \mbox{if} \quad t \le 0,\]
and $t_-$ can be defined by $t_- = (-t)_+$. Let $\alpha $ be a complex number. For $\mbox{Re } \alpha > 0$, the tempered distribution $\frac{t_+^{\alpha-1}}{\Gamma(\alpha)}$ is defined as a locally integrable function by
\begin{equation*}
	\left \langle \frac{t_+^{\alpha-1}}{\Gamma(\alpha)},\ f \right \rangle=\frac{1}{\Gamma(\alpha)}\int_0^{\infty} t^{\alpha-1}f(t) \; dt.
\end{equation*}
It is straightforward to obtain
\begin{equation}\label{eq:derivative}
	\frac{t_+^{\alpha-1}}{\Gamma(\alpha)}=\partial_t^k\left( \frac{t_+^{\alpha+k-1}}{\Gamma(\alpha+k)}\right),
\end{equation}
for all $k\in\mathbb{N}$. The expression \eqref{eq:derivative} facilitates to extend the definition of $\frac{t_+^{\alpha-1}}{\Gamma(\alpha)}$ to all $\alpha \in \mathbb{C}$ in the sense of distributions. The Fourier transform of $\frac{t_+^{\alpha-1}}{\Gamma(\alpha)}$ is given by
\begin{equation}\label{transformada}
	\left(\frac{t_+^{\alpha-1}}{\Gamma(\alpha)}\right)^{\widehat{}}(\tau)=e^{-\frac{1}{2}\pi i \alpha}(\tau-i0)^{-\alpha},
\end{equation}
where $(\tau-i0)^{-\alpha}$ is the distributional limit. When $\alpha \notin \Z$, \eqref{transformada} can be rewritten by 
\begin{equation}\label{transformada1}
	\left(\frac{t_+^{\alpha-1}}{\Gamma(\alpha)}\right)^{\widehat{}}(\tau)=e^{-\frac12 \alpha \pi i}|\tau|^{-\alpha}\chi_{(0,\infty)}+e^{\frac12 \alpha \pi i}|\tau|^{-\alpha}\chi_{(-\infty,0)}.
\end{equation}
Together with \eqref{transformada} and \eqref{transformada1}, we see
%\begin{equation}\label{transformada2}
\[(\tau-i0)^{-\alpha} = |\tau|^{-\alpha}\chi_{(0,\infty)}+e^{\alpha \pi i}|\tau|^{-\alpha}\chi_{(-\infty,0)}.\]
%\end{equation}
For $f\in C_0^{\infty}(\mathbb{R}^+)$, we define the \emph{Riemann-Liouville fractional integral} operator by
\begin{equation*}
	\mathcal{I}_{\alpha}f=\frac{t_+^{\alpha-1}}{\Gamma(\alpha)}*f,
\end{equation*}
in particular, 
%\begin{equation}\label{eq:IO}
\[	\mathcal{I}_{\alpha}f(t)=\frac{1}{\Gamma(\alpha)}\int_0^t(t-s)^{\alpha-1}f(s) \; ds,\]
%\end{equation}
for $\mbox{Re }\alpha>0$. The well-know properties are $\mathcal{I}_0f=f$, $\mathcal{I}_1f(t)=\int_0^tf(s) \; ds$, $\mathcal{I}_{-1}f=f'$ and $\mathcal{I}_{\alpha}\mathcal{I}_{\beta}=\mathcal{I}_{\alpha+\beta}$.

We end this subsection with introducing some lemmas associated to the \emph{Riemann-Liouville fractional integral} operator $\mathcal{I}_{\alpha}f$ without proofs.
\begin{lemma}[Lemma 2.1 in \cite{Holmerkdv}]
	If $f\in C_0^{\infty}(\mathbb{R}^+)$, then $\mathcal{I}_{\alpha}f\in C_0^{\infty}(\mathbb{R}^+)$, for all $\alpha \in \mathbb{C}$.
\end{lemma}
\begin{lemma}[Lemma 5.3 in \cite{Holmerkdv}]\label{lio}
	If $0\leq \mathrm{Re} \ \alpha <\infty$ and $s\in \mathbb{R}$, then $\|\mathcal{I}_{-\alpha}h\|_{H_0^s(\mathbb{R}^+)}\leq c \|h\|_{H_0^{s+\alpha}(\mathbb{R}^+)}$, where $c=c(\alpha)$.
\end{lemma}
\begin{lemma}[Lemma 5.4 in \cite{Holmerkdv}]
	If $0\leq \mathrm{Re}\ \alpha <\infty$, $s\in \mathbb{R}$ and $\mu\in C_0^{\infty}(\mathbb{R})$, then
	$\|\mu\mathcal{I}_{\alpha}h\|_{H_0^s(\mathbb{R}^+)}\leq c \|h\|_{H_0^{s-\alpha}(\mathbb{R}^+)},$ where $c=c(\mu, \alpha)$.
\end{lemma}

\subsection{Oscillatory integral}
Let 
\begin{equation}\label{eq:oscil}
B^{(n)}(x)=\frac{1}{2\pi}\int_{\R}(i\xi)^n e^{ix\xi}e^{i \xi^5} \; d\xi
\end{equation} 
for $n = 0, 1, \cdots$. A direct calculation (with the change of variable $\eta = \xi^5$, the change of contour in complex analysis and a property of the gamma function $\Gamma(z)\Gamma(1-z) = \frac{\pi}{\sin(z\pi)}$) gives\footnote{Non-singularity of $B^{(n)}(x)$ at $x=0$ is needed for the continuity property of $\partial_x^k \mathcal L^0f$, $k=0,1,2,3$, see Lemma \ref{continuity}.}
\[B(0) = \frac{\cos\left(\frac{\pi}{10}\right)}{5\sin\left(\frac{\pi}{5}\right)\Gamma(4/5)}, \qquad B'(0) = -\frac{\cos\left(\frac{3\pi}{10}\right)}{5\sin\left(\frac{2\pi}{5}\right)\Gamma(3/5)},\] 
\[B''(0) = -\frac{\cos\left(\frac{3\pi}{10}\right)}{5\sin\left(\frac{2\pi}{5}\right)\Gamma(2/5)} \quad \mbox{and} \quad  B^{(3)}(0) = \frac{\cos\left(\frac{\pi}{10}\right)}{5\sin\left(\frac{\pi}{5}\right)\Gamma(1/5)}.\]
Moreover, we have 
\[\int_0^{\infty} B(y) \; dy = \frac{1}{2\pi} \left( - \frac{\pi}{5} + \pi \right) = \frac25.\]
We refer to \cite{CK2018-1} for the details.

We finish this subsection with introducing some lemmas associated to $B(x)$ without proofs.
\begin{lemma}[Decay of oscillatory integral $B(x)$, \cite{SS2003, CK2018-1}]\label{lem:decay}
	Suppose $x>0$. Then as $x\rightarrow \infty$,
	\begin{itemize}
		\item [(i)] $B(x)\lesssim \langle x \rangle^{-N}$ for all $N>0$.
		\item [(ii)] $B(-x)\lesssim \langle x \rangle^{-3/8}$.
	\end{itemize}
\end{lemma}

\begin{lemma}[Mellin transform of $B(x)$] \label{mellin}$\;$\\
 \begin{itemize}
		\item[(i)]For Re $\lambda>0$ we have
		\begin{equation}\label{mellin2}
		\int_0^{\infty}x^{\lambda-1}B(x)dx=\frac{\Gamma(\lambda)\Gamma(\frac15-\frac{\lambda}{5})}{5 \pi}\cos \left(\frac{(1+4\lambda)\pi}{10}\right).
		\end{equation}
\item[(ii)] For $0<\mathrm{Re}\ \lambda <\frac{3}{8}$ we have
	%\begin{equation}\label{mellin1}
\[\int_0^{\infty}x^{\lambda-1}B(-x)dx=\frac{\Gamma(\lambda)\Gamma(\frac15-\frac{\lambda}{5})}{5 \pi}\cos \left(\frac{(1-6\lambda)\pi}{10}\right).\]
%	\end{equation}
		\end{itemize}
\end{lemma}
We remark in \eqref{mellin2} that $\Gamma(\frac15-\frac{\lambda}{5})$ has poles at $\lambda = 1+ 5n$, $n=0,1,2, \cdots$, but $\cos \left(\frac{(1+4\lambda)\pi}{10}\right) = 0$ at the same values of $\lambda$. Moreover, the range of $\mbox{Re } \lambda$ relies on the decay rates of $B(x)$ and $B(-x)$ in Lemma \ref{lem:decay}. 

\subsection{Sobolev spaces on the half-line and solution spaces}\label{sec:sol space}
Let $s \ge 0$. We say $f \in H^s(\mathbb{R}^+)$ if there exists $F \in H^s(\R)$ such that $f(x)=F(x)$ for $x>0$, in this case we set $\|f\|_{H^s(\mathbb{R}^+)}=\inf_{F}\|F\|_{H^{s}(\mathbb{R})}$. For $s \in \R$, we say $f \in H_0^s(\mathbb{R}^+)$ if there exists $F \in H^s(\R)$ such that $F$ is the extension of $f$ on $\R$ and $F(x) = 0$ for $x<0$. In this case, we set $\norm{f}_{H^s_0(\R^+)} = \norm{F}_{H^s(\R)}$.
For $s<0$, we define $H^s(\mathbb{R}^+)$ as the dual space of $H_0^{-s}(\mathbb{R}^+)$.

We also set $C_0^{\infty}(\mathbb{R}^+)=\{f\in C^{\infty}(\mathbb{R});\, \supp f \subset [0,\infty)\}$, and define $C_{0,c}^{\infty}(\mathbb{R}^+)$ as the subset of $C_0^{\infty}(\mathbb{R}^+)$, whose members have a compact support on $(0,\infty)$. We remark that $C_{0,c}^{\infty}(\mathbb{R}^+)$ is dense in $H_0^s(\mathbb{R}^+)$ for all $s\in \mathbb{R}$.

We state elementary properties of the Sobolev space on the half-line. We refer to \cite{JK1995, CK, CK2018-1} for the proofs.

\begin{lemma}[Lemma 2.1 in \cite{CK2018-1}]\label{sobolevh0}
	For $-\frac{1}{2}<s<\frac{1}{2}$ and $f\in H^s(\mathbb{R})$, we have
%	\begin{equation}\label{eq:0}
	\[\|\chi_{(0,\infty)}f\|_{H^s(\mathbb{R})}\leq c \|f\|_{H^s(\mathbb{R})}.\]
	%\end{equation}
\end{lemma}

\begin{lemma}[Lemma 2.2 in \cite{CK2018-1}]\label{sobolev0}
	If $0\leq s<\frac{1}{2}$, then $\|\psi f\|_{H^s(\mathbb{R})}\leq c \|f\|_{\dot{H}^{s}(\mathbb{R})}$ and $\|\psi f\|_{\dot{H}^{-s}(\mathbb{R})}\leq c \|f\|_{H^{-s}(\mathbb{R})}$ , where the constant $c$ depends only on $s$ and $\psi$. 
\end{lemma}
Remark that Lemma \ref{sobolev0} is equivalent that $\|f\|_{H^s(\mathbb{R})} \sim \|f\|_{\dot{H}^{s}(\mathbb{R})}$ for $-\frac12 < s < \frac12$ where $f \in H^s$ with $\supp f \subset [0,1]$.

\begin{lemma}[Proposition 2.4 in \cite{CK}]\label{alta}
	If $\frac{1}{2}<s<\frac{3}{2}$ the following statements are valid:
	\begin{enumerate}
\item [(a)] $H_0^s(\R^+)=\big\{f\in H^s(\R^+);f(0)=0\big\},$\medskip
\item [(b)] If  $f\in H^s(\R^+)$ with $f(0)=0$, then $\|\chi_{(0,\infty)}f\|_{H_0^s(\R^+)}\leq c \|f\|_{H^s(\R^+)}$.
	\end{enumerate}
\end{lemma}

\begin{lemma}[Proposition 2.5. in \cite{CK}]\label{cut}
	Let $f\in  H_0^s(\mathbb{R}^+)$. For the cut-off function $\psi$ defined in \eqref{eq:cutoff0}, we have $ \|\psi f\|_{H_0^s(\mathbb{R}^+)}\leq c \|f\|_{H_0^s(\mathbb{R}^+)}$ for $-\infty<s<\infty$.
\end{lemma}

Let $f \in \Sch (\R^2) $. We define the Fourier transform of $f$ with respect to both spatial and time variables by
\[\wt{f}(\tau , \xi)=\int _{\R ^2} e^{-ix\xi}e^{-it\tau}f(t,x) \;dxdt,\]
and denote by $\wt{f}$ or $\ft (f)$. We use $\ft_x$ and $\ft_t$ (or $\wh{\;}$ without distinction of variables) to denote the Fourier transform with respect to space and time variable respectively. 

For $s,b\in \mathbb{R}$, the classical Bourgain space $X^{s,b}$ \cite{Bourgain1993} associated to \eqref{eq:5kdv} is defined as the completion of $\Sch'(\mathbb{R}^2)$ under the norm
\[\norm{f}_{X^{s,b}}^2 = \int_{\R^2} \bra{\xi}^{2s}\bra{\tau - \xi^5}^{2b}|\wt{f}(\tau,\xi)|^2 \; d\xi d\tau, \]
where $\bra{\cdot} = (1+|\cdot|^2)^{1/2}$. 

As already mentioned in Section \ref{sec:intro}, the modifications of $X^{s,b}$ spaces are needed for our analysis due to the Duhamel boundary forcing operator and the time trace estimates. The modulation exponent $b$ of the standard $X^{s,b}$ space is forced to be taken in the range $(0,\frac12)$ from the $X^{s,b}$ estimation of the Duhamel boundary forcing terms (see Lemma \ref{edbf} (c)). On the other hand, very low frequency interactions in the nonlinear estimates compel the exponent $b$ to be bigger than $1/2$. To balance these inter-contradiction conditions, we define the low frequency localized $X^{s,b}$-type space $D^{\alpha}$ as the completion of $\Sch'(\mathbb{R}^2)$ under the norm
\[\norm{f}_{D^{\alpha}}^2 = \int_{\R^2} \bra{\tau}^{2\alpha}\mathbf{1}_{\set{\xi : |\xi| \le 1}}(\xi)|\wt{f}(\tau,\xi)|^2 \; d\xi d\tau,\]
where $\mathbf{1}_A$ is the characteristic functions on a set $A$. 

Besides, the time trace estimate of the Duhamel parts 
\[\|\psi_T(t) \partial_x^2\mathcal{D}w(x,t)\|_{C(\mathbb{R}_x;H^{\frac{s}{5}}(\mathbb{R}_t))}\lesssim T^{\theta}\|w\|_{X^{s,-b}}\]
holds only for the positive regularity. To meet the negative regularity in the nonlinear estimate (see Lemma \ref{duhamel} (b)), it is necessary to define the (time-adapted) Bourgain space $Y^{s,b}$ associated to \eqref{eq:5kdv} as the completion of $S'(\mathbb{R}^2)$ under the norm
\[\norm{f}_{Y^{s,b}}^2 = \int_{\R^2} \bra{\tau}^{\frac{2s}{5}}\bra{\tau - \xi^5}^{2b}|\wt{f}(\tau,\xi)|^2 \; d\xi d\tau.\]

We make a Littlewood-Paley decomposition. Let $\Z_+ = \Z \cap [0,\infty)$. For $k \in \Z_+$, we set
\[I_0 = \set{\xi \in \R : |\xi| \le 2} \hspace{1em} \mbox{ and } \hspace{1em} I_k = \set{\xi \in \R : 2^{k-1} \le |\xi| \le 2^{k+1}}, \hspace{1em} k \ge 1.\]
Let $\eta_0: \R \to [0,1]$ denote a smooth bump function supported in $ [-2,2]$ and equal to $1$ in $[-1,1]$. For $k \in \Z_+ $, we define 
%\begin{equation}\label{eq:cut-off1}
\[\chi_0(\xi) = \eta_0(\xi), \hspace{1em} \mbox{and} \hspace{1em} \chi_k(\xi) = \eta_0(\xi/2^k) - \eta_0(\xi/2^{k-1}), \hspace{1em} k \ge 1,\]
%\end{equation}
on the support $I_k$. Let $P_k$ denote the $L^2$ operators defined by $\widehat{P_kv}(\xi)=\chi_k(\xi)\wh{v}(\xi)$. For the modulation decomposition, we use the multiplier $\eta_j$, but the same as $\eta_j(\tau-\xi^5) = \chi_j(\tau-\xi^5)$. For $k,j \in \Z_+$, let
\[D_{k,j}=\{(\tau,\xi) \in \R^2 : \tau - \xi^5 \in I_j, \xi \in I_k \}, \hspace{2em} D_{k,\le j}=\cup_{l\le j}D_{k,l}.\]
The Littlewood-Paley theory allows that
\begin{equation}\label{eq:dyadic X}
\norm{f}_{X^{s,b}}^2 \sim \sum_{k\ge0}\sum_{j\ge0}2^{2sk}2^{2bj}\norm{\eta_j(\tau-\xi^5)\chi_k(\xi)\wt{f}(\tau,\xi)}_{L^2}^2
\end{equation}
and
\[\norm{f}_{D^{\alpha}}^2 \sim \norm{P_0f}_{X^{0,\alpha}}^2.\]

We define the solution space denoted by $Z_{\ell}^{s,b,\alpha}$ under the norm\footnote{$Y^{s,b}$ norm plays a role of the intermediate norm in the Picard iteration argument (see Lemma \ref{duhamel} (b) and Section \ref{sec:nonlinear}).}:
\begin{equation}\label{eq:solution space}
\norm{f}_{Z_{\ell}^{s,b,\alpha}(\R^2) }= \sup_{t \in \R} \norm{f(t,\cdot)}_{H^s} + \sum_{j=0}^{\ell}\sup_{x \in \R} \norm{\px^jf(\cdot,x)}_{H^{\frac{s+2-j}{5}}} + \norm{f}_{X^{s,b} \cap D^{\alpha}},
\end{equation}
for $\ell = 1,2$. From the boundary conditions in \eqref{regularidade} and \eqref{regularidadeleft}, one can see that $Z_{1}^{s,b,\alpha}$ ($Z_{2}^{s,b,\alpha}$) space is for the right half-line (left half-line) problem (see Section \ref{sec:thm proof}). The standard spatial and time localization of $Z_{\ell}^{s,b,\alpha}(\R^2)$ is 
\[Z_{\ell}^{s,b,\alpha}((0 , T)\times \R^+) = Z_{\ell}^{s,b,\alpha} \Big|_{(0 , T)\times \R^+}\]
equipped with the norm
\[\norm{f}_{Z_{\ell}^{s,b,\alpha}((0 , T)\times \R^+) } = \inf\limits_{g \in Z_{\ell}^{s,b,\alpha}} \set{\norm{g}_{Z_{\ell}^{s,b,\alpha}} : g(t,x) = f(t,x) \; \mbox{ on } \; (0,T) \times \R^+}. \]

\section{Nonlinear estimates}\label{sec:nonlinear}
In this section, we are going to establish nonlinear estimates, in particular, the control of $\norm{F(u)}_{X^{s,-b}}$ and $\norm{F(u)}_{Y^{s,-b}}$.

\subsection{$L^2$-block estimates}
Let $a_1,a_2,a_3 \in \R$. The quantities $a_{max} \ge a_{med} \ge a_{min}$ can be conveniently defined to be the maximum, median and minimum values of $a_1,a_2,a_3$ respectively. Similarly, for $b_1,b_2,b_3,b_4 \in \R$, the quantities $b_{max} \ge b_{sub} \ge b_{thd} \ge b_{min}$ are defined to be the maximum, sub-maximum, third-maximum  and minimum values of $b_1,b_2,b_3,b_4$ respectively.

For $\xi_1,\xi_2 \in \R$, let denote the (quadratic) resonance function by
\begin{equation}\label{eq:resonant function}
\begin{aligned}
H = H(\xi_1,\xi_2) &= (\xi_1 + \xi_2)^5 - \xi_1^5 -\xi_2^5 \\
&=\frac52\xi_1\xi_2(\xi_1+\xi_2)(\xi_1^2+\xi_2^2+(\xi_1+\xi_2)^2),
\end{aligned}
\end{equation}
which plays an crucial role in the bilinear $X^{s,b}$-type estimates. 

Let $f,g,h \in L^2(\R^2)$ be compactly supported functions. We define a quantity by
\[J_2(f,g,h) = \int_{\R^4} f(\zeta_1,\xi_1)g(\zeta_2,\xi_2)h(\zeta_1+\zeta_2+H(\xi_1,\xi_2), \xi_1+\xi_2) \; d\xi_1d\xi_2d\zeta_1\zeta_2. \]
The change of variables in the integration yields
\[J_2(f,g,h)=J_2(g^*,h,f)=J_2(h,f^*,g),\]
where 
\begin{equation}\label{f*}
f^*(\zeta,\xi)=f(-\zeta,-\xi).
\end{equation}

From the identities
\begin{equation}\label{eq:symmetry1}
\xi_1 + \xi_2 = \xi_3
\end{equation}
and
\begin{equation}\label{eq:symmetry2}
(\tau_1 - \xi_1^5)+(\tau_2 - \xi_2^5) = (\tau_3 - \xi_3^5) + H(\xi_1,\xi_2)
\end{equation}
on the support of $J_2(f^{\sharp},g^{\sharp},h^{\sharp})$, where $f^{\sharp}(\tau,\xi) = f(\tau-\xi^5,\xi)$ with the property $\norm{f}_{L^2} = \norm{f^{\sharp}}_{L^2}$, we see that $J(f^{\sharp},g^{\sharp},h^{\sharp})$ vanishes unless
\begin{equation}\label{eq:support property}
\begin{array}{c}
2^{k_{max}} \sim 2^{k_{med}} \gtrsim 1\\
2^{j_{max}} \sim \max(2^{j_{med}}, |H|).
\end{array}
\end{equation}

We give the bilinear $L^2$-block estimates for the quadratic nonlinearity $F(u) = (1-\px^2)^{\frac12}\px (u^2)$. See \cite{CLMW2009, CG2011} for the proof. 
\begin{lemma}\label{lem:block estimate} 
	Let $k_i \in \Z,j_i\in \Z_+,i=1,2,3$. Let $f_{k_i,j_i} \in L^2(\R\times\R) $ be nonnegative functions supported in $[2^{k_i-1},2^{k_i+1}]\times I_{j_i}$.
	
	(a) For any $k_1,k_2,k_3 \in \Z$ with $|k_{max}-k_{min}| \le 5$ and $j_1,j_2,j_3 \in \Z_+$, then we have
	%\begin{eqnarray}\label{eq:block estimate1}
	\[J_2(f_{k_1,j_1},f_{k_2,j_2},f_{k_3,j_3}) \lesssim 2^{j_{min}/2}2^{j_{med}/4}2^{- \frac34 k_{max}}\prod_{i=1}^3 \|f_{k_i,j_i}\|_{L^2}.\]
%	\end{eqnarray}
	
	(b) If $2^{k_{min}} \ll 2^{k_{med}} \sim 2^{k_{max}}$, then for all $i=1,2,3$ we have
	%\begin{eqnarray}\label{eq:block estimate2}
	\[J_2(f_{k_1,j_1},f_{k_2,j_2},f_{k_3,j_3}) \lesssim 2^{(j_1+j_2+j_3)/2}2^{-3k_{max}/2}2^{-(k_i+j_i)/2}\prod_{i=1}^3 \|f_{k_i,j_i}\|_{L^2}.\]
%	\end{eqnarray}
	
	(c) For any $k_1,k_2,k_3 \in \Z$ and $j_1,j_2,j_3 \in \Z_+$, then we have
	%\begin{eqnarray}\label{eq:block estimate3}
	\[J_2(f_{k_1,j_1},f_{k_2,j_2},f_{k_3,j_3}) \lesssim 2^{j_{min}/2}2^{k_{min}/2}\prod_{i=1}^3 \|f_{k_i,j_i}\|_{L^2}.\]
%	\end{eqnarray}
\end{lemma}

Similarly, for $\xi_1,\xi_2,\xi_3 \in \R$, let
\begin{equation}\label{eq:tri-resonant function}
\begin{aligned}
G(\xi_1,\xi_2,\xi_3) &= (\xi_1+\xi_2+\xi_3)^5 - \xi_1^5 - \xi_2^5 - \xi_3^5\\
&= \frac52(\xi_1+\xi_2)(\xi_2 + \xi_3)(\xi_3 + \xi_1)(\xi_1^2 + \xi_2^2 +\xi_3^2+(\xi_1 + \xi_2 + \xi_3)^2)
\end{aligned}
\end{equation}
be the (cubic) resonance function, which plays an important role in the trilinear $X^{s,b}$-type estimates. 

For compactly supported functions $f_i \in L^2(\R \times \R)$, $i=1,2,3,4$, we define 
\[J_3(f_1,f_2,f_3,f_4) = \int_{*}f_1(\zeta_1,\xi_1)f_2(\zeta_2,\xi_2)f_3(\zeta_3,\xi_3)f_4(\zeta_1 + \zeta_2 + \zeta_3  + G(\xi_1,\xi_2,\xi_3), \xi_1 + \xi_2 + \xi_3) ,\]
where the $\int_* = \int_{\R^6} \cdot \; d\xi_1d\xi_2 d\xi_3 d\zeta_1 d \zeta_2 d\zeta_3$. From the identities
\begin{equation}\label{eq:symmetry3}
\xi_1+\xi_2+\xi_3 = \xi_4
\end{equation}
and
\begin{equation}\label{eq:symmetry4}
(\tau_1 - \xi_1^5) + (\tau_2 - \xi_2^5) + (\tau_3 - \xi_3^5) = (\tau_4 - \xi_4^5) + G(\xi_1,\xi_2,\xi_3)
\end{equation}
on the support of $J_3(f_1^{\sharp},f_2^{\sharp},f_3^{\sharp},f_4^{\sharp})$, we see that $J_3(f_1^{\sharp},f_2^{\sharp},f_3^{\sharp},f_4^{\sharp})$ vanishes unless
\begin{equation}\label{eq:tri-support property}
\begin{array}{c}
2^{k_{max}} \sim 2^{k_{sub}}\\
2^{j_{max}} \sim \max(2^{j_{sub}}, |G|),
\end{array}
\end{equation}
where $|\xi_i| \sim 2^{k_i}$ and $|\zeta_i| \sim 2^{j_i}$, $i=1,2,3,4$. A direct calculation shows 
\begin{equation}\label{eq:symmetry}
|J(f_1,f_2,f_3,f_4)|=|J(f_2,f_1,f_3,f_4)|=|J(f_3,f_2,f_1,f_4)|=|J(f_1^{\ast},f_2^{\ast},f_4,f_3)|,
\end{equation}
for $f^{\ast}$ as in \eqref{f*}.

The following lemma provides the trilinear $L^2$-block estimates for the cubic nonlinearity $F(u) = \px (u^3)$.
\begin{lemma}\label{lem:tri-L2}
Let $k_i, j_i\in \Z_+$, $i=1,2,3,4$. Let $f_{k_i,j_i} \in L^2(\R \times \R) $ be nonnegative functions supported in $I_{j_i} \times I_{k_i}$.

(a) For any $k_i,j_i \in \Z_+$, $i=1,2,3,4$, we
have
\begin{eqnarray}\label{eq:tri-block estimate-a1}
J_3(f_{k_1,j_1},f_{k_2,j_2},f_{k_3,j_3},f_{k_4,j_4}) \lesssim 2^{(j_{min}+j_{thd})/2}2^{(k_{min}+k_{thd})/2}\prod_{i=1}^4 \|f_{k_i,j_i}\|_{L^2}.
\end{eqnarray}

(b) Let $k_{thd} \le k_{max}-10$.

(b-1) If $(k_i,j_i) = (k_{thd},j_{max})$ for $i=1,2,3,4$, we have
%\begin{eqnarray}\label{eq:tri-block estimate-b1.2}
 \[J_3(f_{k_1,j_1},f_{k_2,j_2},f_{k_3,j_3},f_{k_4,j_4}) \lesssim 2^{(j_1+j_2+j_3+j_4)/2}2^{-2k_{max}}2^{k_{thd}/2}2^{-j_{max}/2}\prod_{i=1}^4 \|f_{k_i,j_i}\|_{L^2}.\]
%\end{eqnarray}

(b-2) If $(k_i,j_i) \neq (k_{thd},j_{max})$ for $i=1,2,3,4$, we have
%\begin{eqnarray}\label{eq:tri-block estimate-b1.4}
\[ J_3(f_{k_1,j_1},f_{k_2,j_2},f_{k_3,j_3},f_{k_4,j_4}) \lesssim 2^{(j_1+j_2+j_3+j_4)/2}2^{-2k_{max}}2^{k_{min}/2}2^{-j_{max}/2}\prod_{i=1}^4 \|f_{k_i,j_i}\|_{L^2}.\]
%\end{eqnarray}
\end{lemma} 

We refer to \cite{KP2015, Kwak2015} for the proof of Lemma \ref{lem:tri-L2}. In \cite{Kwak2015}, the second author established (cubic) $L^2$-block estimates for functions $f_{k_i,j_i} \in L^2(\R \times \Z)$, but the proof, here, is almost identical and easier, see \cite{KP2015}.

\subsection{$F(u) = (1-\px^2)^{\frac12}\px (u^2)$ case}\label{sec:nonlinear1}
We first prove Theorem \ref{thm:nonlinear1}.
\begin{proposition}\label{prop:bi1}
	For $-5/4 < s $, there exists $b = b(s) < 1/2$ such that for all $\alpha > 1/2$, we have
	\begin{equation}\label{eq:bilinear1}
	\norm{(1-\px^2)^{\frac12}\px(uv)}_{X^{s,-b}} \le c\norm{u}_{X^{s,b} \cap D^{\alpha}}\norm{v}_{X^{s,b} \cap D^{\alpha}}.
	\end{equation}
\end{proposition}

\begin{proof}
	Let
	\begin{equation}\label{eq:restriction1}
	\wt{f}_1(\tau_1,\xi_1) = \beta_1(\tau_1,\xi_1)\wt{u}(\tau_1,\xi_1) \hspace{1em} \mbox{and} \hspace{1em} \wt{f}_2(\tau_2,\xi_2) = \beta_2(\tau_2,\xi_2)\wt{v}(\tau_2,\xi_2),
	\end{equation}
	where
	\begin{equation}\label{eq:restriction2}
	\beta_i(\tau_i,\xi_i) = \bra{\tau_i-\xi_i^5}^b + \mathbf{1}_{|\xi_i|\le1}(\xi_i)\bra{\tau_i}^{\alpha}, \hspace{1em} i=1,2.
	\end{equation}
	Note that $f_1, f_2 \in L^2 \Leftrightarrow u, v \in X^{s,b} \cap D^{\alpha}$ and
	\begin{equation}\label{eq:weight}
	\begin{aligned}
	\frac{1}{\beta_i(\tau_i,\xi_i)} \lesssim \begin{cases}\begin{array}{ll}\bra{\tau_i-\xi_i^5}^{-b}, & \hspace{1em} \mbox{when} \hspace{1em} |\xi_i| > 1, \\ \bra{\tau_i}^{-\alpha}, & \hspace{1em} \mbox{when} \hspace{1em} |\xi_i| \le 1.  \end{array}\end{cases}
	\end{aligned}
	\end{equation}
	
By the duality argument, \eqref{eq:bilinear1} is equivalent to 
	\begin{equation}\label{eq:bilinear1-1}
	\iint\limits_{\substack{\xi_1+\xi_2=\xi \\ \tau_1+\tau_2=\tau}} \frac{|\xi|\bra{\xi}^{s+1}\wt{f}_1(\tau_1,\xi_1)\wt{f}_2(\tau_2,\xi_2)\wt{f}_3(\tau,\xi)}{\bra{\xi_1}^s\bra{\xi_2}^s\bra{\tau-\xi^5}^b\beta_1(\tau_1,\xi_1)\beta_2(\tau_2,\xi_2)} \lesssim \norm{f_1}_{L^2}\norm{f_2}_{L^2}\norm{f_3}_{L^2}.
	\end{equation}
	For $k_i,j_i \in \Z_+$, we make the Littlewood-Paley decomposition of $f_i$, $i=1,2,3$, into $f_{k_i,j_i}$, $i=1,2,3$, by $f_{k_i,j_i}(\tau,\xi) = \eta_{j_i}(\tau-\xi^5)\chi_{k_i}(\xi)\wt{f}_i(\tau,\xi)$. We divide the frequency regions of integration 
	\begin{equation}\label{eq:bilinear1-2}
	\iint\limits_{\substack{\xi_1+\xi_2=\xi \\ \tau_1+\tau_2=\tau}} \frac{|\xi|\bra{\xi}^{s+1}\wt{f}_1(\tau_1,\xi_1)\wt{f}_2(\tau_2,\xi_2)\wt{f}_3(\tau,\xi)}{\bra{\xi_1}^s\bra{\xi_2}^s\bra{\tau-\xi^5}^b\beta_1(\tau_1,\xi_1)\beta_2(\tau_2,\xi_2)}
	\end{equation}
	into several regions associated to the relation of frequencies to prove \eqref{eq:bilinear1-1}.
	
	\textbf{Case I.} \emph{high $\times$ high $\Rightarrow$ high} ($k_3 \ge 10, |k_3-k_1|,|k_3-k_2|\le 5$). From \eqref{eq:support property} and \eqref{eq:resonant function}, $j_{max} \ge 5k_3 -5$ holds in this case. The change of variables yields that \eqref{eq:bilinear1-2} is bounded by
	\[\sum_{\substack{k_3 \ge 10 \\|k_3-k_1| \le 5 \\ |k_3-k_2|\le 5}}\sum_{j_1,j_2,j_3 \ge 0} 2^{(2-s)k_3}2^{-b(j_1+j_2+j_3)}J_2(f_{k_1,j_1}^{\sharp},f_{k_2,j_2}^{\sharp},f_{k_3,j_3}^{\sharp}).\]
By applying Lemma \ref{lem:block estimate} (a) to $J_2(f_{k_1,j_1}^{\sharp},f_{k_2,j_2}^{\sharp},f_{k_3,j_3}^{\sharp})$, and using the Cauchy-Schwarz inequality and \eqref{eq:dyadic X}, it suffices to show
	\begin{equation}\label{eq:bilinear1-3}
	\sum_{\substack{k_3 \ge 10 \\|k_3-k_1| \le 5 \\ |k_3-k_2|\le 5}}\sum_{j_1,j_2,j_3 \ge 0} 2^{2(2-s)k_3}2^{-2b(j_1+j_2+j_3)}2^{j_{min}}2^{j_{med}/2}2^{- \frac32 k_{max}} \lesssim 1.
	\end{equation}
	Without loss of generality, we may assume that $j_1 \le j_2 \le j_3$.  Given $-5/4 < s$, we can choose $\max(3/8 , 3/20-s/5) < b < 1/2$. A computation of the summation over $0 \le j_1 \le j_2 \le j_3$ with $5k_3 - 5 \le j_3$ and $k_j$, $j=1,2,3$ gives
	\[\mbox{LHS of } \eqref{eq:bilinear1-3} \lesssim  \sum_{k_3 \ge 10} 2^{(1/2 -2s)k_3}2^{-10bk_3} \lesssim 1,\]
	which completes the proof of \eqref{eq:bilinear1-3}.

	\textbf{Case II} \emph{high $\times$ low $\Rightarrow$ high} ($k_3 \ge 10, 0 \le k_1 \le k_3-5, |k_3-k_2| \le 5$).\footnote{We may assume that $\xi_1$ the low frequency without loss of generality due to the symmetry.} We further divide the case into two cases: $k_1 = 0$ and $k_1 \ge 1$.
	
	\textbf{Case II-a} $k_1 = 0$. Without loss of generality, we may assume that $j_2 \le j_3$. By \eqref{eq:weight} and Lemma \ref{lem:block estimate} (b), \eqref{eq:bilinear1-2} on this case is dominated by
\begin{equation}\label{eq:bilinear1-4}
\sum_{\substack{k_3 \ge 10 \\ |k_3-k_2|\le 5}}\sum_{j_1,j_2,j_3 \ge 0} 2^{2k_3}2^{-\alpha j_1-bj_2-bj_3}2^{(j_1+j_2+j_3)/2}2^{-3k_{max}/2}2^{-(k_3+j_3)/2}\prod_{\ell =1}^{3} \norm{f_{k_\ell,j_\ell}}_{L^2}.
\end{equation}
Note that
\[\sum_{\substack{0 \le j_1 \\0 \le j_2 \le j_3}} 2^{(1-2\alpha)j_1}2^{(1-2b)j_2}2^{-2bj_3} \lesssim \sum_{0 \le j_1, j_3} 2^{(1-2\alpha)j_1}2^{(1-4b)j_3} \lesssim 1,\]
whenever we choose $1/4 < b <1/2$ and for all $\alpha > 1/2$. We use the Cauchy-Schwarz inequality to obtain  
\[\eqref{eq:bilinear1-4} \lesssim \norm{f_1}_{L^2}\sum_{k_3 \ge 10} \norm{f_{k_3}}_{L^2}\sum_{|k_3-k'|\le 5}\norm{f_{k'}}_{L^2} \lesssim \norm{f_1}_{L^2}\norm{f_2}_{L^2}\norm{f_3}_{L^2}.\]

\begin{remark}
The proof of Proposition \ref{prop:bi1} is indeed analogous to the proof of Proposition 5.1 in \cite{CK2018-1}. However, the \emph{high-low} interaction component with very low frequency ($|\xi_1| \le 1$) (\textbf{Case II-a} above) of $(1-\px^2)^{\frac12}\px (u^2)$ is slightly worse than the same one of $\px(u^2)$ in some sense, since the \emph{high-low} bilinear local smoothing effect exactly cancels two derivatives in high frequency (two derivative gains). As a consequence of this observation, the argument used in the proof of Proposition 5.1 in \cite{CK2018-1} causes a logarithmic divergence in $k_3$-summation, and thus more delicate computation, here, is required as above compared with \textbf{Case II-a} in the proof of Proposition 5.1 in \cite{CK2018-1}.
\end{remark}	
	
	\textbf{Case II-b} $k_1 \ge 1$. In this case, we have from \eqref{eq:support property} and \eqref{eq:resonant function} that $j_{max} \ge 4k_3 + k_1 - 5$. By \eqref{eq:weight} and Lemma \ref{lem:block estimate} (b), it suffices, similarly as \textbf{Case I}, to show
	\begin{equation}\label{eq:bilinear1-5}
	\sum_{\substack{k_3 \ge 10 \\ 1 \le k_1 \le k_3 - 5 \\ |k_3-k_2|\le 5}}\sum_{j_1,j_2,j_3 \ge 0} 2^{4k_3}2^{-2sk_1}2^{-2b(j_1+j_2+j_3)}2^{(j_1+j_2+j_3)}2^{-3k_{max}}2^{-(k_i+j_i)} \lesssim 1.
	\end{equation}
Without loss of generality, we may assume that $j_2 \le j_3$. 
	
	If $j_1 \neq j_{max}$,  given $-5/2 < s$, by choosing $\max((5-s)/15 , 1/3) < b < 1/2$. We perform the summation over $0 \le j_1, j_2 \le j_3$ with the fact $4k_3 + k_1 - 5 \le j_3$ after choosing $(k_i,j_i) = (k_3,j_3)$ to obtain
	\[\mbox{LHS of } \eqref{eq:bilinear1-5} \lesssim  \sum_{k_3 \ge 10  }\sum_{1 \le k_1 \le k_3 - 5} 2^{(8-24b)k_3}2^{(2-6b-2s)k_1} \lesssim 1.\]
	
	If $j_1 = j_{max}$,  given $-5/2 < s$, we can choose $\max((5-s)/15 , 7/24) < b < 1/2$. We perform the summation over $0 \le j_1, j_2 \le j_3$ with the fact $4k_3 + k_1 - 5 \le j_3$ after choosing $(k_i,j_i) = (k_1,j_1)$ to obtain
	\[\mbox{LHS of } \eqref{eq:bilinear1-5} \lesssim  \sum_{k_3 \ge 10  }\sum_{1 \le k_1 \le k_3 - 5} 2^{(9-24b)k_3}2^{(1-6b-2s)k_1} \lesssim 1.\] 
Thus, given $-5/2 < s$, we choose $\max((5-s)/15 , 7/24) < b < 1/2$ such that \eqref{eq:bilinear1-5} holds.  

\begin{remark}\label{rem:lowhigh}
In view of Lemma \ref{lem:block estimate}, the case when the low frequency mode has the largest modulation ($(k_i,j_i) = (k_{min},j_{max})$) is the worst among other cases.
\end{remark}
	
	\textbf{Case III.} \emph{high $\times$ high $\Rightarrow$ low} ($k_2 \ge 10, |k_1-k_2| \le 5, 0 \le k_3 \le k_2 - 5$). We, similarly, further divide the case into two cases: $k_3 = 0$ and $k_3 \ge 1$.
	
	\textbf{Case III-a} $k_3 = 0$. In this case, we know $\bra{\xi_3} \sim 1$. We further decompose the low frequency component $f_3=\sum_{l\leq 0}f_{3,l}$ with $f_{3,l}=\ft^{-1}1_{|\xi|\sim 2^l}\ft f_3$. Then, from \eqref{eq:weight}, \eqref{eq:bilinear1-2} is bounded by
	\begin{equation}\label{eq:hhl}
	\sum_{\substack{k_2 \ge 10 \\ |k_1-k_2| \le 5}}\sum_{l\le0}\sum_{j_1,j_2,j_3 \ge 0} 2^{-2sk_2}2^{l}2^{-b(j_1+j_2+j_3)}J(f_{k_1,j_1}^{\sharp},f_{k_2,j_2}^{\sharp},f_{l,j_3}^{\sharp}),
	\end{equation}
	where $f_{l,j_3}(\tau,\xi) = \eta_{j_3}(\tau-\xi^5)\wt{f}_{3,l}(\tau,\xi)$. Without loss of generality, we may assume that $j_1 \le j_2$.  
	
	From Remark \ref{rem:lowhigh}, the worst case occurs when $j_3=j_{max}$. Lemma \ref{lem:block estimate} (b) in $J(f_{k_1,j_1}^{\sharp},f_{k_2,j_2}^{\sharp},f_{l,j_3}^{\sharp})$ and the Cauchy-Schwarz inequality in terms of $k_2,l,j_i's$ yield
	%\begin{equation}\label{eq:bilinear1-6}
	\[\eqref{eq:hhl} \lesssim \sum_{\substack{k_2 \ge 10 \\ |k_1-k_2| \le 5}}\sum_{l\le 0}\sum_{0 \le j_1 \le j_2 \le j_3} 2^{-4sk_2}2^{2l}2^{-2b(j_1+j_2+j_3)}2^{(j_1+j_2+j_3)}2^{-3k_{max}}2^{-(l+j_3)}.\]
%	\end{equation}
	Given $-7/4 < s$, we can choose $\max((5-4s)/24 , 1/3) < b < 1/2$.  Since $j_{max} \ge 4k_2 + l - 5$ and $1/3 <b < 1/2$, we have
	\[\sum_{\substack{0 \le j_1 \le j_2 \le j_3 \\4k_2 + l - 5 \le j_3}} 2^{(1-2b)j_1}2^{(1-2b)j_2}2^{-2bj_3} \lesssim 2^{(2-6b)(4k_2+l)},\]
 which implies
	\[\eqref{eq:hhl} \lesssim \sum_{k_2 \ge 10}\sum_{l\le 0}2^{(5-4s-24b)k_2}2^{(3-6b)l} \lesssim 1.\]
	
%	If $j_3 \neq j_{max}$ ($j_2 = j_{max}$), similarly as before, we have
%	%\begin{equation}\label{eq:bilinear1-7}
%	\[\eqref{eq:hhl} \lesssim \sum_{\substack{k_2 \ge 10 \\ |k_1-k_2| \le 5}}\sum_{l\le 0}\sum_{0 \le j_1, j_3 \le j_2} 2^{-4sk_2}2^{2l}2^{-2b(j_1+j_2+j_3)}2^{(j_1+j_2+j_3)}2^{-3k_{max}}2^{-(k_2+j_2)}.\]
%	%\end{equation}
%	Given $s > -2$, we can choose $\max((1-s)/6 , 1/3) < b < 1/2$.  Since $j_{max} \ge 4k_2 + l - 5$ and $1/3 <b < 1/2$, we have
%	\[\sum_{\substack{0 \le j_1,j_3 \le j_2 \\4k_2 + l - 5 \le j_2}} 2^{(1-2b)j_1}2^{-2bj_2}2^{(1-2b)j_3} \lesssim 2^{(2-6b)(4k_2+l)},\]
%	which implies
%	\[\eqref{eq:hhl} \lesssim \sum_{k_2 \ge 10}\sum_{l\le 0}2^{(4-4s-24b)k_2}2^{(4-6b)l} \lesssim 1.\]
		
	\textbf{Case III-b} $k_3 \ge 1$. From \eqref{eq:support property} and \eqref{eq:resonant function}, we know $j_{max} \ge 4k_2 + k_3 - 5$ in this case. Similarly, it suffices from \eqref{eq:weight} and Lemma \ref{lem:block estimate} (b) to show
	\begin{equation}\label{eq:bilinear1-8}
	\sum_{\substack{k_2 \ge 10 \\ |k_1-k_2| \le 5 \\ 1 \le k_3 \le k_2 -5}}\sum_{j_1,j_2,j_3 \ge 0} 2^{2(2+s)k_3}2^{-4sk_2}2^{-2b(j_1+j_2+j_3)}2^{(j_1+j_2+j_3)}2^{-3k_{max}}2^{-(k_i+j_i)} \lesssim 1.
	\end{equation}
Without loss of generality, we may assume that $j_1 \le j_2$. 
	
%	If $j_3 \neq j_{max}$ ($j_2 = j_{max}$),  given $ -2 < s$, we can choose $\max((1-s)/6 , 1+s/3) < b < 1/2$\footnote{When $s \ge -3/2$, the $1+s/3 < b$  implies $(5-s)/15 < b$, which guarantees $10 -2s -30b < 0$.}.  We perform the summation over $0 \le j_1, j_3 \le j_2$ in addition to $4k_2 + k_3 - 5 \le j_2$ after choosing $(k_i,j_i) = (k_2,j_2)$ ot have
%	\[\mbox{LHS of } \eqref{eq:bilinear1-8} \lesssim  \sum_{k_2 \ge 10 }\sum_{1 \le k_3 \le k_2 -5} 2^{(4-4s-24b)k_2}2^{(6+2s-6b)k_3}\lesssim 1.\]
	
	Similarly, it suffice to consider the case when $j_3 = j_{max}$. For given $-7/4 < s$, we can choose $ \max((5-4s)/24, (5+2s)/6) < b < 1/2$\footnote{Similarly, when $s \ge -5/4$, the $(5+2s)/6 < b$  implies $(5-s)/15 < b$, which guarantees $10 -2s -30b < 0$.}.  We perform the summation over $0 \le j_1 \le j_2 \le j_3$ in addition to $4k_2 + k_3 - 5 \le j_3$ after choosing $(k_i,j_i) = (k_3,j_3)$ to obtain
	\[\mbox{LHS of } \eqref{eq:bilinear1-8} \lesssim  \sum_{\substack{k_2 \ge 10 \\ |k_1-k_2| \le 5 \\ 1 \le k_3 \le k_2 -5}} 2^{(5-4s-24b)k_2}2^{(5+2s-6b)k_3} \lesssim 1,\] 
	which completes the proof of \eqref{eq:bilinear1-8}.
	
The \emph{low $\times$low $\Rightarrow$low} interaction component can be directly controlled by the Cauchy-Schwarz inequality, since  the low frequency localized space $D^{\alpha}$ with $\alpha >1/2$ allows the $L^2$ integrability with respect to $\tau$-variables.
	
	Therefore, the proof of \eqref{eq:bilinear1} is completed.
\end{proof}

\begin{remark}
In view of the proof of Proposition 5.1 in \cite{CK2018-1}, one can see that the regularity threshold appears in the \emph{high $\times$ high $\Rightarrow$ low} interaction component, which is the well-known worst component of quadratic nonlinearity (for semi-linear "dispersive" equations), while the regularity threshold $-5/4$, here, occurs in the \emph{high $\times$ high $\Rightarrow$ high} interaction case. It is because the \emph{high $\times$ high $\Rightarrow$ low} interaction component of $(1-\px^2)^{\frac12}\px (u^2)$ is no longer different from $\px(u^2)$ (roughly, $(1-\px^2)^{\frac12} \sim 1$).
%The regularity threshold $-5/4$ occurs in the \emph{high $\times$ high $\Rightarrow$ high} interaction case, while the regularity threshold $-7/4$ occurs from \emph{high $\times$ high $\Rightarrow$ low} interaction case in Proposition 5.1 in \cite{CK2018-1}, where the nonlinearity $F(u) = \px(u^2)$ is considered, since the factor $(1-\px^2)^{1/2}$ in low frequency does not influence on the high frequency regularity in the \emph{high $\times$ high $\Rightarrow$ low} interaction component. 
\end{remark}

\begin{proposition}\label{prop:bi2}
	For $-5/4 < s \le 0$, there exists $b = b(s) < 1/2$ such that for all $\alpha > 1/2$, we have
	\begin{equation}\label{eq:bilinear2}
	\norm{(1-\px^2)^{\frac12}\px(uv)}_{Y^{s,-b}} \le c\norm{u}_{X^{s,b} \cap D^{\alpha}}\norm{v}_{X^{s,b} \cap D^{\alpha}}.
	\end{equation}
\end{proposition}

We state the elementary integral estimates without proof.
\begin{lemma}[Lemmas 5.12, 5.13 in \cite{Holmerkdv}]
Let $\alpha, \beta \in \R$. 

(a) If $\frac14 < b < \frac12$, then
\begin{equation}\label{eq:integral1}
\int_{-\infty}^{\infty} \frac{dx}{\bra{x-\alpha}^{2b}\bra{x-\beta}^{2b}} \le \frac{c}{\bra{\alpha-\beta}^{4b-1}}.
\end{equation}

(b) If $b < \frac12$, then
\begin{equation}\label{eq:integral2}
\int_{|x|\le\beta} \frac{dx}{\bra{x}^{4b-1}|\alpha-x|^{1/2}} \le \frac{c(1+\beta)^{2-4b}}{\bra{\alpha}^{1/2}}.
\end{equation}

(c) Moreover, if $\alpha \in \R$ and $\frac13 < b < \frac12$, we have
\begin{equation}\label{eq:integral3}
\int_{-\infty}^{\infty} \frac{dx}{\bra{x}^{2b}\bra{x - \alpha}^{4b-1}} \le \frac{c}{\bra{\alpha}^{6b-2}}.
\end{equation}
\end{lemma}

The proof of \eqref{eq:integral3} is almost identical to the proof of \eqref{eq:integral1} and \eqref{eq:integral2}, hence we omit the detail. 

\begin{proof}[Proof of Proposition \ref{prop:bi2}]
We may assume that $|\tau| \le \frac{1}{32}|\xi|^5$ for $-5/4 < s\le0$, otherwise, it follows \eqref{eq:bilinear1} in the proof of Proposition \ref{prop:bi1} due to $\bra{\tau}^{\frac{s}{5}} \lesssim \bra{\xi}^s$. A direct calculation gives 
	\[\frac{31}{32}|\xi|^5 \le \frac{31}{32}|\xi|^5 - (|\tau| - \frac{1}{32}|\xi|^5) = |\xi|^5 - |\tau| \le |\tau-\xi^5| \le |\tau|+ |\xi|^5 \le \frac{33}{32}|\xi|^5,\]
	which implies
	\begin{equation}\label{eq:same size}
	|\tau-\xi^5| \sim |\xi|^5
	\end{equation} 
	under the assumption $|\tau| \le \frac{1}{32}|\xi|^5$. Moreover, we have
	\begin{equation}\label{eq:same size1}
	|\tau-\frac{1}{16}\xi^5| \sim |\xi|^5.
	\end{equation} 
 We use the same notation $f_i$ defined as in \eqref{eq:restriction1} under \eqref{eq:restriction2} and \eqref{eq:weight}. Then, \eqref{eq:bilinear2} is equivalent to
	\begin{equation}\label{eq:bilinear2-1}
	\iint\limits_{\substack{\xi_1+\xi_2=\xi \\ \tau_1+\tau_2=\tau}} \frac{|\xi|\bra{\xi}\bra{\tau}^{\frac{s}{5}}\wt{f}_1(\tau_1,\xi_1)\wt{f}_2(\tau_2,\xi_2)\wt{f}_3(\tau,\xi)}{\bra{\xi_1}^s\bra{\xi_2}^s\bra{\xi}^{5b}\beta_1(\tau_1,\xi_1)\beta_2(\tau_2,\xi_2)} \lesssim \norm{f_1}_{L^2}\norm{f_2}_{L^2}\norm{f_3}_{L^2}.
	\end{equation}
	We may assume from the symmetry that $|\xi_1| \le |\xi_2|$ without loss of generality.
	
	\textbf{Case I} $|\xi_2|< 1$. From the identity \eqref{eq:symmetry1}, we know $|\xi| < 1$ in this case, which implies $|\tau| \lesssim 1$. Then, the left-hand side of \eqref{eq:bilinear2-1} is equivalent to
	\[\iint\limits_{\substack{\xi_1+\xi_2=\xi \\ \tau_1+\tau_2=\tau \\ |\xi_1|,|\xi_2|,|\xi| < 1}} \frac{\wt{f}_1(\tau_1,\xi_1)\wt{f}_2(\tau_2,\xi_2)\wt{f}_3(\tau,\xi)}{\bra{\tau_1}^{\alpha}\bra{\tau_2}^{\alpha}}.\]
The Cauchy-Schwarz inequality yields \eqref{eq:bilinear2-1} thanks to $\alpha > 1/2$.
	
	\textbf{Case II} $|\xi_2| \ge 1$. We further split the region of $\xi_1$ into two regions.
	
	\textbf{Case II-1} $|\xi_1| < 1$. From the identity \eqref{eq:symmetry1}, we know $|\xi| \ge 1$. Moreover, $\bra{\tau}^{s/5} \lesssim 1$ in the negative regularity regime. Then, the left-hand side of \eqref{eq:bilinear2-1} is bounded by
	\begin{equation}\label{eq:bilinear2-2}
	\iint\limits_{\ast} \frac{|\xi|^2\wt{f}_1(\tau_1,\xi_1)\wt{f}_2(\tau_2,\xi_2)\wt{f}_3(\tau,\xi)}{\bra{\xi_2}^s\bra{\xi}^{5b}\bra{\tau_1}^{\alpha}\bra{\tau_2-\xi_2^5}^{b}}
	\end{equation}
	where
	\[\ast = \{(\tau_1,\tau_2,\tau,\xi_1,\xi_2,\xi) \in \R^6: \xi_1+\xi_2=\xi,\;  \tau_1+\tau_2=\tau,\;  |\xi_1| < 1,\; |\xi_2|, |\xi| \ge 1\}.\]
	From \eqref{eq:same size} and \eqref{eq:resonant function} under the assumption $|\xi_1| < 1 \le |\xi|$, we know
	\begin{equation}\label{eq:big condition}
	|\tau - \xi^5| \sim |\xi|^5 \gg |\xi_1||\xi|^4 \sim |H|.
	\end{equation}
From \eqref{eq:support property}, we, thus, divide this case into the following two cases:
	\[|\tau - \xi^5| \sim |\tau_1 - \xi_1^5| \hspace{1em} \mbox{or} \hspace{1em} |\tau - \xi^5| \sim |\tau_2 - \xi_2^5| \gg |\tau_1 - \xi_1^5|.\]
	
	For the first case, we denote the region of $\xi_1$ in the integral by $A = \set{\xi_1 : |\xi_1| \le |\xi|^{-2}} \cup \set{\xi_1 : |\xi|^{-2} < |\xi_1| \le 1} =: A_1 \cup A_2$.
	
	On $A_1$, for given $-3/2 < s \le 0$ we can choose $b=b(s)$ satisfying $\frac{1-s}{5} < b < \frac12$. Since $\bra{\tau_2 - \xi_2^5}^{-b} \lesssim 1$, the Cauchy-Schwarz inequality with respect to $\xi_1,\xi_2, \tau_1,\tau_2$ yields
	\[\eqref{eq:bilinear2-2} \lesssim |\xi|^{2-s-5b-1}\norm{f_1}_{L^2}\norm{f_2}_{L^2}\norm{f_3}_{L^2} \lesssim \norm{f_1}_{L^2}\norm{f_2}_{L^2}\norm{f_3}_{L^2}.\]
	
	On $A_2$, from \eqref{eq:symmetry2} and \eqref{eq:big condition}, we always have 
	\[|\tau_2- \xi_2^5| = |\tau-\xi^5 - (\tau_1 - \xi_1^5) + H| \gtrsim |H| \sim |\xi_1||\xi|^4 \gtrsim |\xi|^2,\]
	which guarantees $\bra{\tau_2 - \xi_2^5}^{-b} \lesssim |\xi|^{-2b}$. For given $-3/2 < s \le 0$, we can choose $b=b(s)$ satisfying $\frac{2-s}{7} < b < \frac12$. Then, the Cauchy-Schwarz inequality with respect to $\xi_1,\xi_2, \tau_1,\tau_2$ yields 
	\[\eqref{eq:bilinear2-2} \lesssim |\xi|^{2-s-7b}\norm{f_1}_{L^2}\norm{f_2}_{L^2}\norm{f_3}_{L^2} \lesssim \norm{f_1}_{L^2}\norm{f_2}_{L^2}\norm{f_3}_{L^2}.\]
	
	For the second case ($|\tau - \xi^5| \sim |\tau_2 - \xi_2^5|\gg |\tau_1 - \xi_1^5|$), we know
	\[|\tau_2- \xi_2^5| \sim |\tau - \xi^5| \sim |\xi|^5.\]
For given $-3 < s \le 0$ we can choose $b=b(s)$ satisfying $\frac{2-s}{10} < b < \frac12$. Then, the Cauchy-Schwarz inequality with respect to $\xi_1,\xi_2, \tau_1,\tau_2$ yields
	\[\eqref{eq:bilinear2-2} \lesssim |\xi|^{2-s-10b}\norm{f_1}_{L^2}\norm{f_2}_{L^2}\norm{f_3}_{L^2} \lesssim \norm{f_1}_{L^2}\norm{f_2}_{L^2}\norm{f_3}_{L^2}.\]
	
	\textbf{Case II-2} $1 \le |\xi_1| \le |\xi_2|$. We may further assume that $|\tau_1 - \xi_1^2| \le |\tau_2 - \xi_2^5|$ due to the symmetry. 
	
	\textbf{Case II-2.a} $|\tau_2 - \xi_2^5| \le 100000 |\tau -\xi^5|$. In this case, it suffices to show from the Cauchy-Schwarz inequality that
	\begin{equation}\label{eq:bilinear2-3}
	\sup_{\substack{\xi,\tau \in \R\\ |\tau| \le \frac{1}{32}|\xi|^5}}\Big(\iint\limits_{\substack{\xi_1+\xi_2=\xi \\ \tau_1+\tau_2=\tau}} \frac{|\xi|^2\bra{\xi}^2\bra{\tau}^{\frac{2s}{5}}}{\bra{\xi_1}^{2s}\bra{\xi_2}^{2s}\bra{\xi}^{10b}\bra{\tau_1-\xi_1^5}^{2b}\bra{\tau_2-\xi_2^5}^{2b}} \; d\xi_1\;d\tau_1\Big)^{1/2} \le c.
	\end{equation}
	Under the assumption, we only consider the case when $|\xi| \ge 1$. Otherwise, \eqref{eq:same size} implies $|\tau_1-\xi_1^5| \lesssim 1$, and hence we have \eqref{eq:bilinear2-1} similarly as \textbf{Case II-1} for $-2 \le s \le 0$. Indeed, from the identity \eqref{eq:symmetry2} under this condition, we know $|H| \lesssim 1$. Since 
%	\begin{equation}\label{eq:H low}
	\[|H| = \frac52|\xi_1||\xi_2||\xi|(\xi_1^2+\xi_2^2+\xi^2) \ge 5|\xi_1|^2|\xi_2|^2|\xi|,\]
	%\end{equation}
	we have $|\xi_1|^{-s}|\xi_2|^{-s} \lesssim |\xi|^{\frac{s}{2}}$, and hence $|\xi|^{1+s/2} \le 1$ for $-2 \le s \le 0$. The Cauchy-Schwarz inequality with respect to $\xi_1, \xi, \tau_1, \tau$ guarantees \eqref{eq:bilinear2-1}.
	
	We now consider \eqref{eq:bilinear2-3} on the case when $|\xi| \ge 1$. We use \eqref{eq:integral1} in addition to \eqref{eq:symmetry2} so that the left-hand side of \eqref{eq:bilinear2-3} is bounded by
	\begin{equation}\label{eq:bilinear2-4}
	\frac{|\xi|^2\bra{\tau}^{\frac{s}{5}}}{\bra{\xi}^{5b}}\left(\int_{\xi_1+\xi_2=\xi} \frac{d\xi_1}{\bra{\xi_1}^{2s}\bra{\xi_2}^{2s}\bra{\tau-\xi^5 + H}^{4b-1}}\right)^{1/2}.
	\end{equation}
	The support property ($|\tau-\xi^5| \gtrsim |H|$) and \eqref{eq:same size} implies
	\[|\xi_1|^{-2s}|\xi_2|^{-2s} \lesssim |\xi|^{-4s},\]
	and hence \eqref{eq:bilinear2-4} can be controlled by
	\begin{equation}\label{eq:bilinear2-5}
	\frac{|\xi|^{2-2s}\bra{\tau}^{\frac{s}{5}}}{\bra{\xi}^{5b}}\left(\int_{\R} \frac{d\xi_1}{\bra{\tau-\xi^5 + H}^{4b-1}}\right)^{1/2}.
	\end{equation}
	Let $\mu = \tau - \xi^5 + H$. Note that $|\mu| \le 2|\tau-\xi^5|$ in this case. Then, by the direct calculation, we know
	\[\mu - (\tau - \frac{1}{16}\xi^5) = -\frac{5}{16}\xi(\xi-2\xi_1)^2(2\xi^2 + (\xi-2\xi_1)^2)\]
	and
	\[d\mu = \frac52\xi(\xi^2 + (\xi-2\xi)^2)(\xi-2\xi_1) \; d\xi_1.\]
	Since
	\[|\xi|^{\frac32}|\mu-(\tau-\frac{1}{16}\xi^5)|^{\frac12} \le |\xi||\xi-2\xi_1||2\xi^2 +(\xi-2\xi_1)^2| \le 2|\xi||\xi-2\xi_1||\xi^2 +(\xi-2\xi_1)^2|,\]
	we can reduce \eqref{eq:bilinear2-5} by
	\begin{equation}\label{eq:bilinear2-6}
	\frac{|\xi|^{2-2s}\bra{\tau}^{\frac{s}{5}}}{\bra{\xi}^{5b}|\xi|^{\frac34}}\left(\int_{|\mu| \lesssim |\tau - \xi^5|} \frac{d\mu}{\bra{\mu}^{4b-1}|\mu - (\tau-\frac{1}{16}\xi^5)|^{1/2}}\right)^{1/2}.
	\end{equation}
	By \eqref{eq:integral2}, \eqref{eq:bilinear2-6} is bounded by
	\[\frac{|\xi|^{2-2s}\bra{\tau}^{\frac{s}{5}}\bra{\tau-\xi^5}^{1-2b}}{\bra{\xi}^{5b}|\xi|^{\frac34}\bra{\tau-\frac{1}{16}\xi^5}^{1/4}}.\]
	For given $-5/4 < s \le 0$, we choose $b = b(s)$ satisfying $\frac{5-2s}{15} \le b < \frac12$. From \eqref{eq:same size} and \eqref{eq:same size1} with $|\xi| \ge 1$ and $s \le 0$, we obtain
	\[\frac{|\xi|^{2-2s}\bra{\tau}^{\frac{s}{5}}\bra{\tau-\xi^5}^{1-2b}}{\bra{\xi}^{5b}|\xi|^{\frac34}\bra{\tau-\frac{1}{16}\xi^5}^{1/4}} \lesssim |\xi|^{5-2s-15b} \lesssim 1.\]
	
	\textbf{Case II-2.b} $|\tau - \xi^5| \le \frac{1}{100000} |\tau_2 -\xi_2^5|$. In this case, it suffices to show from the Cauchy-Schwarz inequality that
	\begin{equation}\label{eq:bilinear2-7}
	\sup_{\xi_2,\tau_2 \in \R}\Big(\iint\limits_{\substack{\xi_1+\xi_2=\xi \\ \tau_1+\tau_2=\tau}} \frac{|\xi|^2\bra{\xi}^2\bra{\tau}^{\frac{2s}{5}}}{\bra{\xi_1}^{2s}\bra{\xi_2}^{2s}\bra{\xi}^{10b}\bra{\tau_1-\xi_1^5}^{2b}\bra{\tau_2-\xi_2^5}^{2b}} \; d\xi\;d\tau\Big)^{1/2} \le c.
	\end{equation}
	In this case we fix $-2 < s \le 0$. Since $-5/2 <-2 < s$, we can choose $b = b(s)$ satisfying $-s/5 \le b < \frac12$. From the fact that
	\[\bra{\tau}^{\frac{2s}{5}+2b} \lesssim \bra{\xi^5}^{\frac{2s}{5}+2b} \sim \bra{\xi}^{2s+10b},\]
	the left-hand side of \eqref{eq:bilinear2-7} is bounded by
	\begin{equation}\label{eq:bilinear2-8}
	\frac{1}{\bra{\tau_2-\xi_2^5}^{b}}\Big(\iint\limits_{\substack{\xi_1+\xi_2=\xi \\ \tau_1+\tau_2=\tau}} \frac{|\xi|^2\bra{\xi}^{2s+2}}{\bra{\xi_1}^{2s}\bra{\xi_2}^{2s}\bra{\tau}^{2b}\bra{\tau_1-\xi_1^5}^{2b}} \; d\xi\;d\tau\Big)^{1/2}.
	\end{equation}
	
	When $|H| \le \frac12|\tau_2-\xi_2^5|$, we can know the following facts:
	\[\begin{aligned}
	&\hspace{-6em}\diamond |\tau-\xi^5| \ll |\tau_2-\xi_2^5| \; \mbox{ and } \; |\tau| \le \frac{1}{32}|\xi|^5 \; \mbox{ imply } \; |\xi|^5 \ll |\tau_2 - \xi_2^5|.\\
	&\hspace{-6em}\diamond \bra{\tau_2 - \xi_2^5 - H + \xi^5} \sim \bra{\tau_2 - \xi_2^5}.\\
	&\hspace{-6em}\diamond \bra{\xi_1}^{-2s}\bra{\xi_2}^{-2s} \lesssim |\tau_2 - \xi_2^5|^{-s}|\xi|^s.
	\end{aligned}\]
We perform the integration in \eqref{eq:bilinear2-8} in terms of $\tau$ variable by using \eqref{eq:integral1}, then \eqref{eq:bilinear2-8} is bounded by
	%\begin{equation}\label{eq:bilinear2-9}
	\[\begin{aligned}
	&\frac{1}{\bra{\tau_2-\xi_2^5}^{b}}\left(\int_{\R} \frac{|\xi|^2\bra{\xi}^{2s+2}}{\bra{\xi_1}^{2s}\bra{\xi_2}^{2s}\bra{\tau_2 - \xi_2^5 - H + \xi^5}^{4b-1}} \; d\xi\right)^{1/2}\\
	&\sim \frac{1}{\bra{\tau_2-\xi_2^5}^{3b-1/2}}\left(\int_{|\xi| \le |\tau_2 - \xi_2^5|^{1/5}} \frac{|\xi|^2\bra{\xi}^{2s+2}}{\bra{\xi_1}^{2s}\bra{\xi_2}^{2s}} \; d\xi\right)^{1/2}\\
	&\lesssim \frac{|\tau_2-\xi_5^5|^{-s/2}}{\bra{\tau_2-\xi_2^5}^{3b-1/2}}\left(\int_{|\xi| \le |\tau_2 - \xi_2^5|^{1/5}} |\xi|^{2+s}\bra{\xi}^{2s+2} \; d\xi\right)^{1/2}.
	\end{aligned}\]
	%\end{equation}
	For given $-5/2 <-2< s \le 0$, we choose can $b = b(s)$ satisfying $\frac{5-s}{15} < b < \frac12$\footnote{The strict inequality $\frac{5-s}{15} < b$ covers the logarithmic divergence when $s = -\frac53$.}. Then, by performing integration in terms of $\xi$, we have
	\[\frac{|\tau_2-\xi_5^5|^{-s/2}}{\bra{\tau_2-\xi_2^5}^{3b-1/2}}\left(\int_{|\xi| \le |\tau_2 - \xi_2^5|^{1/5}} |\xi|^{2+s}\bra{\xi}^{2s+2} \; d\xi\right)^{1/2} \lesssim \bra{\tau_2-\xi_2^5}^{\frac{1}{10}(10-2s-30b)}\lesssim 1.\]
	
	For the other case ($|H| > \frac12|\tau_2-\xi_2^5|$), we can know the following facts:
	\begin{equation}\label{eq:facts}
	\begin{aligned}
	&\hspace{-17em}\diamond 10|\xi| \le |\xi_1| \sim |\xi_2|.\\
	&\hspace{-17em}\diamond |\xi - \xi_2| \sim |\xi_2|.\\
	&\hspace{-17em}\diamond |\xi| \sim \frac{|\tau_2 - \xi_2^5|}{|\xi_2|^4}.\\
	&\hspace{-17em}\diamond |\xi|^5 \ll |\tau_2 - \xi_2^5|.\\
	&\hspace{-17em}\diamond \bra{\xi_1}^{-2s}\bra{\xi_2}^{-2s} \lesssim |\tau_2 - \xi_2^5|^{-s}|\xi|^s.
	\end{aligned}
	\end{equation}
	To verify the first one in \eqref{eq:facts}\footnote{It is not difficult to verify the others.}, suppose that $|\xi_1|\le 10|\xi|$. From \eqref{eq:symmetry1}, we know $|\xi_2| \le 11|\xi|$. Then,
	\[\begin{aligned}
	|H| &= \frac{5}{2}|\xi_1||\xi_2||\xi|(|\xi_1|^2 + |\xi_2|^2 + |\xi|^2) \\
	&\le 30525|\xi|^5 \le \frac{976800}{31}|\tau-\xi^5| \le \frac13|\tau_2 - \xi_2^5|,
	\end{aligned}\]
	which contradicts to the assumption $|H| > \frac12|\tau_2-\xi_2^5|$.
	
	Now, under the conditions \eqref{eq:facts}, we control the following integral:
	\begin{equation}\label{eq:bilinear2-10}
	\iint\limits_{\substack{\xi_1+\xi_2=\xi \\ \tau_1+\tau_2=\tau}} \frac{|\xi|^2\bra{\xi}^{2s+2}\bra{\xi_1}^{-2s}\bra{\xi_2}^{-2s}}{\bra{\tau_2-\xi_2^5}^{2b}\bra{\tau}^{2b}\bra{\tau_1-\xi_1^5}^{2b}} \; d\xi\;d\tau.
	\end{equation}
	
	When $|\xi| \le 1$, \eqref{eq:facts} and \eqref{eq:integral1} yield
	\[\begin{aligned}
	\eqref{eq:bilinear2-10} &\lesssim \iint\limits_{\substack{\xi_1+\xi_2=\xi \\ \tau_1+\tau_2=\tau}} \frac{|\xi|^{2+s}\bra{\xi}^{2s+2}\bra{\tau_2-\xi_2^5}^{-s-2b}}{\bra{\tau}^{2b}\bra{\tau_1-\xi_1^5}^{2b}} \; d\tau\;d\xi\\
	&\lesssim \int_{|\xi |\le 1} \frac{|\xi|^{2+s}\bra{\xi}^{2s+2}\bra{\tau_2-\xi_2^5}^{-s-2b}}{\bra{\tau_2-\xi_2^5 - H + \xi^5}^{4b-1}} \; d\xi
	\end{aligned}\]
	Let $\mu = \tau_2 - \xi_2^5 - H + \xi^5$, then we have $d\mu = 5(\xi-\xi_2)^4 \; d\xi$. From the facts \eqref{eq:facts} with $|\xi| \le 1$, since $|\xi_2|^{-4} \lesssim |\tau_2 - \xi_2^5|^{-1}$, the change of variable enables us to get
	\[\eqref{eq:bilinear2-10} \lesssim \int_{|\mu| \le |\tau_2 - \xi_2^5|} \frac{\bra{\tau_2-\xi_2^5}^{-s-2b-1}}{\bra{\mu}^{4b-1}} \; d\mu\]
	for $-2 < s \le 0$. For given $-2 < s \le 0$, we can choose $b=b(s)$ satisfying $\frac{1-s}{6} \le b < \frac12$. Then, by performing the integration in terms of $\mu$, we have
	\[\eqref{eq:bilinear2-10} \lesssim \bra{\tau_2-\xi_2^5}^{-s-6b+1} \lesssim 1.\]
	
	Now, we focus on the case when $|\xi| > 1$. Similarly as before, \eqref{eq:bilinear2-10} can be reduced by 
	\begin{equation}\label{eq:bilinear2-11}
	\int_{|\xi| > 1} \frac{|\xi|^{4+3s}\bra{\tau_2-\xi_2^5}^{-s-2b}}{\bra{\tau_2-\xi_2^5 - H + \xi^5}^{4b-1}} \; d\xi.
	\end{equation}
		
	We use the change of variable $\mu = \tau_2 - \xi_2^5 - H + \xi^5$ with
	\[d\mu = 5(\xi-\xi_2)^4 \; d\xi.\]
	If $-2 < s \le -5/3$, since 
	\[|\xi_2|^{-4} \sim |\xi||\tau_2-\xi_2^5|^{-1} \hspace{1em}(\Rightarrow |\xi|^{5+3s} \lesssim 1),\]
	we have 
	\[\eqref{eq:bilinear2-11} \lesssim \int_{|\mu| < |\tau_2-\xi_2^5|} \frac{\bra{\tau_2-\xi_2^5}^{-s-2b-1}}{\bra{\mu}^{4b-1}} \; d\mu \lesssim \bra{\tau_2-\xi_2^5}^{-s-6b+1} \lesssim 1\]
	by choosing $b = b(s)$ satisfying $(1-s)/6 < b < 1/2$.
	
	Otherwise ($-5/3 < s \le 0$), we can choose $b=b(s)$ satisfying $\frac{5-s}{15} \le b < \frac12$. Then, from the fact $|\xi| \ll |\tau_2 - \xi_2^5|^{1/5}$, we obtain 
	\[\begin{aligned}
	\eqref{eq:bilinear2-11} &\lesssim \int_{|\mu| \le |\tau_2 - \xi_2^5|} \frac{\bra{\tau_2-\xi_2^5}^{\frac{5+3s}{5}-s-2b-1}}{\bra{\mu}^{4b-1}} \; d\mu\\
	&\lesssim \bra{\tau_2 - \xi_2^5}^{\frac{10-2s-30b}{5}} \lesssim 1.
	\end{aligned}\]
	
	Therefore, we complete the proof of Proposition \ref{prop:bi2}.
\end{proof}

\subsection{$F(u) = \px(u^3)$ case}\label{sec:nonlinear2}
We now prove Theorem \ref{thm:nonlinear2}. 
\begin{proposition}\label{prop:tri1}
	For $-1/4 \le s $, there exists $b = b(s) < 1/2$ such that for all $\alpha > 1/2$, we have
	\begin{equation}\label{eq:trilinear1}
	\norm{\px(uvw)}_{X^{s,-b}} \le c\norm{u}_{X^{s,b} \cap D^{\alpha}}\norm{v}_{X^{s,b} \cap D^{\alpha}}\norm{w}_{X^{s,b} \cap D^{\alpha}}.
	\end{equation}
\end{proposition}

Before proving Proposition \ref{prop:tri1}, we bring the Strichartz estimates for the fifth-order dispersive equations.
\begin{lemma}[Strichartz estimates for $e^{t\px^5}$ operator \cite{CT2005}]\label{eq:strichartz}
Assume that $-1 < \sigma \le \frac32$ and $0 \le \theta \le 1$. Then there exists $C>0$ depending on $\sigma$ and $\theta$ such that
\[\norm{D^{\frac{\sigma\theta}{2}e^{t\px^5}\varphi}}_{L^q_tL^p_x} \le C \norm{\varphi}_{L^2}\]
for $\varphi \in L^2$, where $p=\frac{2}{1-\theta}$ and $q = \frac{10}{\theta(\sigma +1)}$. In particular, we have 
\begin{equation}\label{eq:strichartz}
\norm{e^{t\px^5}P_k\varphi}_{L_{t,x}^6} \lesssim  2^{-k/2}\norm{P_k\varphi}_{L^2}, \quad k\ge 1.
\end{equation}
\end{lemma}

\begin{proof}[Proof of Proposition \ref{prop:tri1}.]
Similar mechanism as in the proof of Proposition \ref{prop:bi1} will be used. Let 
\begin{equation}\label{eq:uvw}
\begin{split}
\wt{f}_1(\tau_1,\xi_1) &= \beta_1(\tau_1,\xi_1)\wt{u}(\tau_1,\xi_1),  \qquad \wt{f}_2(\tau_2,\xi_2) = \beta_2(\tau_2,\xi_2)\wt{v}(\tau_2,\xi_2)\\
&\mbox{and} \quad \wt{f}_3(\tau_3,\xi_3) = \beta_3(\tau_3,\xi_3)\wt{w}(\tau_3,\xi_3),
\end{split}
\end{equation}
where
	\begin{equation}\label{eq:uvw2}
	\beta_i(\tau_i,\xi_i) = \bra{\tau_i-\xi_i^5}^b + \mathbf{1}_{|\xi_i|\le1}(\xi_i)\bra{\tau_i}^{\alpha}, \hspace{1em} i=1,2,3
	\end{equation}
	satisfying
	\begin{equation}\label{eq:weight1}
	\begin{aligned}
	\frac{1}{\beta_i(\tau_i,\xi_i)} \lesssim \begin{cases}\begin{array}{ll}\bra{\tau_i-\xi_i^5}^{-b}, & \hspace{1em} \mbox{when} \hspace{1em} |\xi_i| > 1, \\ \bra{\tau_i}^{-\alpha}, & \hspace{1em} \mbox{when} \hspace{1em} |\xi_i| \le 1.  \end{array}\end{cases}
	\end{aligned}
	\end{equation}
	Note that $f_1, f_2, f_3 \in L^2 \Leftrightarrow u, v, w \in X^{s,b} \cap D^{\alpha}$. By the duality argument, \eqref{eq:trilinear1} is equivalent to 
	\begin{equation}\label{eq:trilinear1-1}
	\iint\limits_{\substack{\xi_1+\xi_2+\xi_3=\xi \\ \tau_1+\tau_2+\tau_3=\tau}} \frac{|\xi|\bra{\xi}^{s}\wt{f}_1(\tau_1,\xi_1)\wt{f}_2(\tau_2,\xi_2)\wt{f}_3(\tau_3,\xi_3)\wt{f}_4(\tau,\xi)}{\bra{\xi_1}^s\bra{\xi_2}^s\bra{\xi_3}^s\bra{\tau-\xi^5}^b\beta_1(\tau_1,\xi_1)\beta_2(\tau_2,\xi_2)\beta_3(\tau_3,\xi_3)} \lesssim \prod_{i=1}^{4}\norm{f_i}_{L^2}.
	\end{equation}
	Let $k_i,j_i \in \Z_+$. We decompose $f_i$, $i=1,2,3,4$, into $f_{k_i,j_i}$, $i=1,2,3,4$, by $f_{k_i,j_i}(\tau,\xi) = \eta_{j_i}(\tau-\xi^5)\chi_{k_i}(\xi)\wt{f}_i(\tau,\xi)$. We divide the frequency regions of integration 
	\begin{equation}\label{eq:trilinear1-2}
\iint\limits_{\substack{\xi_1+\xi_2+\xi_3=\xi \\ \tau_1+\tau_2+\tau_3=\tau}} \frac{|\xi|\bra{\xi}^{s}\wt{f}_1(\tau_1,\xi_1)\wt{f}_2(\tau_2,\xi_2)\wt{f}_3(\tau_3,\xi_3)\wt{f}_4(\tau,\xi)}{\bra{\xi_1}^s\bra{\xi_2}^s\bra{\xi_3}^s\bra{\tau-\xi^5}^b\beta_1(\tau_1,\xi_1)\beta_2(\tau_2,\xi_2)\beta_3(\tau_3,\xi_3)}.	
\end{equation}
into several regions associated to the relation of frequencies to prove \eqref{eq:trilinear1-1}.
	
\textbf{Case I.}	high-high-high $\Rightarrow$ high ($k_4 \ge 10$ and $|k_1-k_4|, |k_2-k_4|, |k_3-k_4| \le 5$). Without loss of generality, we may assume $j_4 = j_{max}$. The change of variables yields in this case that \eqref{eq:trilinear1-2} is bounded by
	\begin{equation}\label{eq:trilinear1-I.1}
\sum_{\substack{k_4 \ge 10 \\|k_4-k_i| \le 5, \; i=1,2,3}}\sum_{j_1,j_2,j_3,j_4 \ge 0} 2^{(1-2s)k_4}2^{-b(j_1+j_2+j_3+j_4)}J_3(f_{k_1,j_1}^{\sharp},f_{k_2,j_2}^{\sharp},f_{k_3,j_3}^{\sharp},f_{k_4,j_4}^{\sharp}).
\end{equation}
On th other hand, since $f_{k_i,j_i}^{\sharp}(\tau, \xi) = f_{k_i,j_i}(\tau-\xi^5,\xi)$, we get
\[\begin{aligned}
\Big|& J_3(f_{k_1,j_1}^{\sharp},f_{k_2,j_2}^{\sharp},f_{k_3,j_3}^{\sharp},f_{k_4,j_4}^{\sharp}) \Big| = \left| \int(f_{k_1,j_1}^{\sharp}\ast f_{k_2,j_2}^{\sharp}\ast f_{k_3,j_3}^{\sharp})f_{k_4,j_4}^{\sharp} \right| \\
&\lesssim \norm{f_{k_1,j_1}^{\sharp}\ast f_{k_2,j_2}^{\sharp}\ast f_{k_3,j_3}^{\sharp}}_{L^2}\norm{f_{k_4,j_4}^{\sharp}}_{L^2} \lesssim \prod_{i=1}^{3}\norm{\ft^{-1}(f_{k_i,j_i}^{\sharp})}_{L^6}\norm{f_{k_4,j_4}}_{L^2}.
\end{aligned}\]
The Fourier inversion formula, Minkowski inequality, \eqref{eq:strichartz} and the Cauchy-Schwarz inequality yield
\[\begin{aligned}
\norm{\ft^{-1}(f_{k_i,j_i}^{\sharp})}_{L^6} &= \left\|\int e^{it\tau}e^{ix\xi}e^{it\xi^5}f_{k_i,j_i}(\tau,\xi) \;d\xi d\tau \right\|_{L^6}\\
&\lesssim \int\left\| \int e^{ix\xi}e^{it\xi^5}f_{k_i,j_i}(\tau,\xi) \;d\xi  \right\|_{L^6} \; d\tau\\
&\lesssim 2^{-k_i/2}2^{j_i/2}\norm{f_{k_i,j_i}}_{L^2}.
\end{aligned}\]
Using this, we estimate \eqref{eq:trilinear1-I.1} by
\begin{equation}\label{eq:trilinear1-I.2} 
\sum_{\substack{k_4 \ge 10 \\|k_4-k_i| \le 5,\; i=1,2,3}}\sum_{0\le j_1\le j_2 \le j_3 \le j_4} 2^{(1-2s)k_4}2^{-\frac32k_4}2^{(\frac12-b)(j_1+j_2+j_3+j_4)}2^{-j_4/2} \prod_{i=1}^4\norm{f_{k_i,j_i}}_{L^2}.
\end{equation}
The choice of of $ \frac38 < b < \frac12$ ensures the $\ell^2$-summability of $2^{(\frac12-b)(j_1+j_2+j_3+j_4)}2^{-j_4/2}$ over $0\le j_1\le j_2 \le j_3 \le j_4$. On the other hand, we see that the frequency summation includes only one infinite sum as
\[\sum_{\substack{k_4 \ge 10 \\|k_4-k_i| \le 5,\; i=1,2,3}} = \sum_{k_4 \ge 10}\sum_{k_4 - 5 \le k_3\le k_4 +5}\sum_{k_4 - 5 \le k_2\le k_4 +5}\sum_{k_4 - 5 \le k_1\le k_4 +5}.\]
We therefore have for $s \ge -\frac14$ that
\[\eqref{eq:trilinear1-I.2} \lesssim \norm{f_1}_{L^2}\norm{f_2}_{L^2}\norm{f_3}_{L^2}\norm{f_4}_{L^2}.\]

\textbf{Case II.} high-high-low $\Rightarrow$ high ($k_4 \ge 10$, $|k_2-k_4|, |k_3-k_4| \le 5$ and $k_1 \le k_4-10$)\footnote{We may assume that $\xi_1$ is the lowest frequency without loss of generality due to the symmetry.}. In this case, we know from \eqref{eq:tri-support property} and \eqref{eq:tri-resonant function} that $j_{max} \ge 5k_4$. We further divide the case into two cases: $k_1 = 0$ and $k_1 \ge 1$.

\textbf{Case II-a.} $k_1=0$. It suffices to consider
%\begin{equation}\label{eq:trilinear1-II.1}
\[\sum_{\substack{k_4 \ge 10 \\|k_4-k_i| \le 5, \; i=2,3 }}\sum_{j_1,j_2,j_3,j_4 \ge 0} 2^{(1-s)k_4}2^{-\alpha j_1}2^{-b(j_2+j_3+j_4)}J_3(f_{0,j_1}^{\sharp},f_{k_2,j_2}^{\sharp},f_{k_3,j_3}^{\sharp},f_{k_4,j_4}^{\sharp}).\]
%\end{equation}
By \eqref{eq:tri-block estimate-a1}, we can control $J_3(f_{0,j_1}^{\sharp},f_{k_2,j_2}^{\sharp},f_{k_3,j_3}^{\sharp},f_{k_4,j_4}^{\sharp})$, and hence it suffices to show
\begin{equation}\label{eq:trilinear1-II.2}
\sum_{\substack{k_4 \ge 10 \\|k_4-k_i| \le 5, \; i=2,3 }}\sum_{j_1,j_2,j_3,j_4 \ge 0} 2^{2(1-s)k_4}2^{(1-2\alpha) j_1}2^{(1-2b)(j_2+j_3+j_4)}2^{k_4}2^{-(j_{sub}+j_{max})} \lesssim 1.
\end{equation}
Without loss of generality, we may assume $j_2 \le j_3 \le j_4$. When $j_4 = j_{max}$, we know $j_3 \le j_{sub}$. For $s > -1$, by choosing $\max(\frac{3-2s}{10}, \frac14) < b < \frac12$, we have
\[\begin{aligned}
\mbox{LHS of }\eqref{eq:trilinear1-II.2} &\lesssim \sum_{\substack{k_4 \ge 10 \\|k_4-k_i| \le 5, \; i=2,3 }}\sum_{\substack{0 \le j_1 \\  0 \le j_2 \le j_3 \le j_4}} 2^{(3-2s)k_4}2^{(1-2\alpha) j_1}2^{(1-2b)j_2}2^{-2bj_3}2^{-2bj_4)} \\
&\lesssim \sum_{\substack{k_4 \ge 10 \\|k_4-k_i| \le 5, \; i=2,3 }}\sum_{\substack{0 \le j_1, j_3 \\ j_4 \ge 5k_4}} 2^{(3-2s)k_4}2^{(1-2\alpha) j_1}2^{(1-4b)j_3}2^{-2bj_4}\\
&\lesssim \sum_{\substack{k_4 \ge 10 \\|k_4-k_i| \le 5, \; i=2,3 }}2^{(3-2s-10b)k_4} \lesssim 1.
\end{aligned}\]
whenever $\alpha > \frac12$. When $j_4 \neq j_{max}$, we know $j_3 \le j_1$. Since $j_{max} \ge 5k_4$, we have
\[\begin{aligned}
\mbox{LHS of }\eqref{eq:trilinear1-II.2} &\lesssim \sum_{\substack{k_4 \ge 10 \\|k_4-k_i| \le 5, \; i=2,3 }}\sum_{\substack{0 \le j_1 \\  0 \le j_2 \le j_3 \le j_4}} 2^{(-2-2s)k_4}2^{(1-2\alpha) j_1}2^{(1-2b)(j_2+j_3)}2^{-2bj_4)} \\
&\lesssim \sum_{0 \le j_1, j_4} 2^{(1-2\alpha) j_1}2^{(2-6b)j_4} \lesssim 1.
\end{aligned}\]
whenever $s > -1$, $\alpha >\frac12$ and $\frac13 < b < \frac12$.

\textbf{Case II-b.} $k_1 \ge 1$. Without loss of generality, we may assume that $j_1 \le j_2 \le j_3 \le j_4$. Similarly as before, it suffices to show
\begin{equation}\label{eq:trilinear1-II.3}
\begin{aligned}
\sum_{\substack{k_4 \ge 10 \\|k_4-k_i| \le 5, \; i=2,3\\1 \le k_1 \le k_4-10 }}\sum_{j_1,j_2,j_3,j_4 \ge 0} 2^{(1-s)k_4}2^{-sk_1}2^{-b(j_1 +j_2+j_3+j_4)}J_3(f_{k_1,j_1}^{\sharp},f_{k_2,j_2}^{\sharp},&f_{k_3,j_3}^{\sharp},f_{k_4,j_4}^{\sharp}) \\
&\lesssim \prod_{i=1}^4\norm{f_{k_i,j_i}}_{L^2}.
\end{aligned}
\end{equation}
If $j_1 = j_{\max}$, we apply the argument used in \textbf{Case I} to $J_3(f_{k_1,j_1}^{\sharp},f_{k_2,j_2}^{\sharp},f_{k_3,j_3}^{\sharp},f_{k_4,j_4}^{\sharp})$ by changing the role of $f_{k_1,j_1}$ and $f_{k_4,j_4}$. It is possible thanks to \eqref{eq:symmetry}. Similarly as before, we have 
\[\mbox{LHS of }\eqref{eq:trilinear1-II.3} \lesssim \sum_{\substack{k_4 \ge 10 \\|k_4-k_i| \le 5, \; i=2,3 \\1 \le k_1 \le k_4-10 }}\sum_{\substack{j_1,j_2,j_3,j_4 \ge 0 \\ j_1 = j_{max}}} 2^{-(1/2+s)k_4}2^{-sk_1}2^{(\frac12-b)(j_2+j_3+j_4)}2^{-bj_1}\prod_{i=1}^{4}\norm{f_{k_i,j_i}}_{L^2}.\]
Since the frequency summation includes only two infinite sums (but one of them is for low frequency mode), for $s \ge -\frac14$, by choosing $\frac38 < b <\frac12$, we can have
\begin{equation}\label{eq:trilinear1-II.4}
\mbox{LHS of }\eqref{eq:trilinear1-II.3} \lesssim \norm{f_1}_{L^2}\norm{f_2}_{L^2}\norm{f_3}_{L^2}\norm{f_4}_{L^2}.
\end{equation}
If $j_1 \neq j_{max}$ (we assume $j_4 = j_{max}$), we can obtain
\begin{equation}\label{eq:trilinear1-II.5}
J_3(f_{k_1,j_1}^{\sharp},f_{k_2,j_2}^{\sharp},f_{k_3,j_3}^{\sharp},f_{k_4,j_4}^{\sharp}) \lesssim 2^{-\frac32k_4}2^{(j_1+j_2+j_3)/2}\prod_{i=1}^{4}\norm{f_{k_i,j_i}}_{L^2}.
\end{equation}
Then, similarly as the case when $j_1=j_{max}$, we have \eqref{eq:trilinear1-II.4}. Now it remains to show \eqref{eq:trilinear1-II.5}. It suffices to  show 
\begin{equation}\label{eq:j4}
\int_{\R^3} g_1(\xi_1)g_2(\xi_2)g_3(\xi_3)g_4(G(\xi_1,\xi_2,\xi_3),\xi_1+\xi_2+\xi_3)\;d\xi_1d\xi_2d\xi_3 \lesssim 2^{-\frac32k_4} \prod_{i=1}^{4}\norm{g_i}_{L^2}
\end{equation}
for $L^2$-functions $g_i : \R \to \R_{\ge 0}$ supported in $I_{k_i}$, $i=1,2,3$, and  $g_4 : \R^2 \to \R_{\ge 0}$ supported in $I_{j_4} \times I_{k_4}$, where $G$ is defined as in \eqref{eq:tri-resonant function}. Indeed, if \eqref{eq:j4} holds true, then
\[\begin{aligned}
&J_3(f_{k_1,j_1}^{\sharp},f_{k_2,j_2}^{\sharp},f_{k_3,j_3}^{\sharp},f_{k_4,j_4}^{\sharp}) \\
&= \iint_{\ast} f_{k_1,j_1}^{\sharp}(\tau_1, \xi_1)f_{k_2,j_2}^{\sharp}(\tau_2, \xi_2)f_{k_3,j_3}^{\sharp}(\tau_3, \xi_3)f_{k_4,j_4}^{\sharp}(\tau_1 + \tau_2 + \tau_3 + G(\xi_1, \xi_2, \xi_3), \xi_1 +\xi_2 + \xi_3) \\
&\lesssim 2^{-\frac32k_4}\norm{f_{k_4,j_4}}_{L^2}\int_{\R^3} \norm{f_{k_1,j_1}(\tau_1)}_{L_{\xi_1}^2}\norm{f_{k_2,j_2}(\tau_2)}_{L_{\xi_2}^2}\norm{f_{k_3,j_3}(\tau_3)}_{L_{\xi_3}^2}\; d\tau_1d\tau_2d\tau_3 \\
&\lesssim 2^{-\frac32k_4}2^{(j_1+j_2+j_3)/2}\prod_{i=1}^{4}\norm{f_{k_i,j_i}}_{L^2}.
\end{aligned}\]
The change of variables ($\xi_1' = \xi_1$, $\xi_2' = \xi_1+\xi_2$ and $\xi_3' = \xi_3$) gives
\[\mbox{LHS of }\eqref{eq:j4} = \int g_1(\xi_1)g_2(\xi_2-\xi_1)g_3(\xi_3)g_4(G(\xi_1,\xi_2-\xi_1,\xi_3),\xi_2+\xi_3)\;d\xi_1d\xi_2d\xi_3.\]
Note that $|\xi_i| \sim 2^{k_i}$, $i=1,2,3$, still holds. A direct calculation gives
\[|\partial_{\xi_1} G(\xi_1,\xi_2-\xi_1,\xi_3)| = |-5\xi_1^4 + 5 (\xi_2-\xi_1)^4| \sim 2^{4k_4},\]
and then the Cauchy-Schwarz inequality with respect to $\xi_1$ and $\xi_2$, and the change of variable ($\mu = G(\xi_1,\xi_2-\xi_1,\xi_3)$) ensure
\[\begin{aligned}
\mbox{LHS of }\eqref{eq:j4} &\lesssim  2^{-2k_4}\int g_3(\xi_3)\norm{g_1}_{L^2}\norm{g_2}_{L^2}\norm{g_4}_{L^2} \; d\xi_3\\
&\lesssim 2^{-2k_4}2^{k_3/2}\prod_{i=1}^{4}\norm{g_i}_{L^2},
\end{aligned}\]
which completes the proof of \eqref{eq:j4}. Thanks to \eqref{eq:symmetry}, our assumption $j_4 = j_{max}$ does not lose the generality. 

\textbf{Case III.} high-high-high $\Rightarrow$ low ($k_3 \ge 10$, $|k_1-k_3|, |k_2-k_3| \le 5$ and $k_4 \le k_3-10$). In this case, we also have $j_{max} \ge 5k_4$ similarly as \textbf{Case II}. It suffices to show
\begin{equation}\label{eq:trilinear1-III.1}
\begin{aligned}
\sum_{\substack{k_3 \ge 10 \\|k_3-k_i| \le 5, \; i=1,2\\0 \le k_4 \le k_3-10 }}\sum_{j_1,j_2,j_3,j_4 \ge 0} 2^{(1-s)k_4}2^{-3sk_3}2^{-b(j_1 +j_2+j_3+j_4)}J_3(f_{k_1,j_1}^{\sharp},&f_{k_2,j_2}^{\sharp},f_{k_3,j_3}^{\sharp},f_{k_4,j_4}^{\sharp}) \\
&\lesssim \prod_{i=1}^4\norm{f_{k_i,j_i}}_{L^2}.
\end{aligned}
\end{equation}
The exact same argument as in \textbf{Case II-b} (by replacing the role of $j_1$ and $j_4$) can be applied to the left-hand side of \eqref{eq:trilinear1-III.1} and hence, for $s \ge -1/4$, by choosing $\frac38 < b < \frac12$, we prove \eqref{eq:trilinear1-III.1}.

\textbf{Case IV.} high-low-low $\Rightarrow$ high ($k_4 \ge 10$, $|k_3-k_4| \le 5$ and $k_1, k_2 \le k_4 -10$)\footnote{Due to the symmetry, the assumption $|\xi_1|, |\xi_2| \ll |\xi_3|$ does not lose the generality.}. We further assume that $k_1 \le k_2$ without loss of generality.

\textbf{Case IV-a.} $k_2 = 0$. By Lemma \ref{lem:tri-L2} (b-2)\footnote{Since $\xi_1$ and $\xi_2$ are comparable, we can avoid the case when $(k_i,j_i) = (k_{thd},j_{max})$, $i=1,2$.} and the Cauchy-Schwarz inequality, it suffices to show
\begin{equation}\label{eq:trilinear1-IV.1}
\sum_{\substack{k_4 \ge 10 \\|k_4-k_3| \le 5}}\sum_{j_1,j_2,j_3,j_4 \ge 0} 2^{2k_4}2^{-4k_4}2^{(1-2\alpha) (j_1+j_2)}2^{(1-2b)(j_3+j_4)}2^{-j_{max}} \lesssim 1.
\end{equation}
Without loss of generality, we may assume $j_3 \le j_4$. Since $\alpha > \frac12$, by choosing $\frac14 < b < \frac12$, we can show \eqref{eq:trilinear1-IV.1} for any $s \in \R$.

\textbf{Case IV-b.} $k_2 \ge 1$ and $k_1 = 0$. It suffices to consider
\begin{equation}\label{eq:trilinear1-IV.2}
\sum_{\substack{k_4 \ge 10 \\|k_4-k_3| \le 5 \\ 1 \le k_2 \le k_4 -10}}\sum_{j_1,j_2,j_3,j_4 \ge 0} 2^{k_4}2^{-sk_2}2^{-\alpha j_1}2^{-b(j_2+j_3+j_4)}J_3(f_{0,j_1}^{\sharp},f_{k_2,j_2}^{\sharp},f_{k_3,j_3}^{\sharp},f_{k_4,j_4}^{\sharp}).
\end{equation}
If $j_2 = j_{max}$, we have 
\[J_3(f_{0,j_1}^{\sharp},f_{k_2,j_2}^{\sharp},f_{k_3,j_3}^{\sharp},f_{k_4,j_4}^{\sharp}) \lesssim 2^{-2k_4}2^{k_2/2}2^{(j_1+j_3+j_4)/2}\norm{f_{0,j_1}}_{L^2}\prod_{i=2}^4\norm{f_{k_i,j_i}}_{L^2},\]
thanks to Lemma \ref{lem:tri-L2} (b-1). Otherwise, we have
\[J_3(f_{0,j_1}^{\sharp},f_{k_2,j_2}^{\sharp},f_{k_3,j_3}^{\sharp},f_{k_4,j_4}^{\sharp}) \lesssim 2^{-2k_4}2^{(j_1+j_2+j_3+j_4)/2}2^{-j_{max}/2}\norm{f_{0,j_1}}_{L^2}\prod_{i=2}^4\norm{f_{k_i,j_i}}_{L^2},\]
thanks to Lemma \ref{lem:tri-L2} (b-2). In both cases, for $s \ge -\frac12$, by choosing $\frac13 < b < \frac12$, we have
\[\eqref{eq:trilinear1-IV.2} \lesssim \norm{f_1}_{L^2}\norm{f_2}_{L^2}\norm{f_3}_{L^2}\norm{f_4}_{L^2},\]
whenever $\alpha > \frac12$.

\textbf{Case IV-c.} $k_1 \ge 1$. It suffices to consider
\begin{equation}\label{eq:trilinear1-IV.3}
\sum_{\substack{k_4 \ge 10 \\|k_4-k_3| \le 5 \\ 1 \le k_1 \le k_2 \le k_4 -10}}\sum_{j_1,j_2,j_3,j_4 \ge 0} 2^{k_4}2^{-sk_1}2^{-sk_2}2^{-b(j_1+j_2+j_3+j_4)}J_3(f_{k_1,j_1}^{\sharp},f_{k_2,j_2}^{\sharp},f_{k_3,j_3}^{\sharp},f_{k_4,j_4}^{\sharp}).
\end{equation}
Since the worst bound of $J_3(f_{k_1,j_1}^{\sharp},f_{k_2,j_2}^{\sharp},f_{k_3,j_3}^{\sharp},f_{k_4,j_4}^{\sharp})$ is
\[2^{-2k_4}2^{k_2/2}2^{(j_1+j_2+j_3+j_4)/2} 2^{-j_{max}/2},\]
for $s \ge -\frac14$, by choosing $\frac38 < b < \frac12$, we have
\[\eqref{eq:trilinear1-IV.3} \lesssim \sup_{k_4 > 10} 2^{-(\frac12 +2s)k_4}\sum_{j_{max} \ge 0}2^{(\frac32-4b)j_{max}} \prod_{i=1}^4 \norm{f_i}_{L^2} \lesssim \prod_{i=1}^4 \norm{f_i}_{L^2}. \]

\textbf{Case V.} high-high-low $\Rightarrow$ low ($k_3 \ge 10$, $|k_2-k_3| \le 5$ and $k_1,k_4 \le k_3 -10$)\footnote{We may assume that $\xi_1$ is the lowest frequency among $\xi_1, \xi_2, \xi_3$, without loss of generality due to the symmetry.}. We further divide the case in two cases $k_1 = 0$ and $k_1 \ge 1$. 

\textbf{Case V-a.} $k_1=0$. It suffices to consider
\begin{equation}\label{eq:trilinear1-V.1}
\sum_{\substack{k_3 \ge 10 \\|k_2-k_3| \le 5 \\ 0 \le k_4 \le k_3 -10}}\sum_{j_1,j_2,j_3,j_4 \ge 0} 2^{(1+s)k_4}2^{-2sk_3}2^{-\alpha j_1}2^{-b(j_2+j_3+j_4)}J_3(f_{0,j_1}^{\sharp},f_{k_2,j_2}^{\sharp},f_{k_3,j_3}^{\sharp},f_{k_4,j_4}^{\sharp}).
\end{equation}
The worst case happens when $k_4 \ge 1$ and $j_4 = j_{max}$. By Lemma \ref{lem:tri-L2} (b-1), we have for $s \ge -\frac12$ that
\[\begin{aligned}
\eqref{eq:trilinear1-V.1} &\lesssim \sum_{\substack{k_3 \ge 10 \\|k_2-k_3| \le 5 \\ 0 \le k_4 \le k_3 -10}}\sum_{0 \le j_1,j_2,j_3 \le j_4}2^{(\frac32+s)k_4}2^{-(2+2s)k_3}2^{(\frac12-\alpha )j_1 + (\frac12-b)(j_2+j_3)-bj_4}\norm{f_{0,j_1}}_{L^2}\prod_{i=2}^4\norm{f_{k_i,j_i}}_{L^2}\\
&\lesssim  \prod_{i=1}^4 \norm{f_i}_{L^2},
\end{aligned}\]
by choosing $\frac13 < b <\frac12$ and $\alpha > \frac12$.
 
\textbf{Case V-b.} $k_1 \ge 1$. Similarly as before, the worst bound of $J_3(f_{0,j_1}^{\sharp},f_{k_2,j_2}^{\sharp},f_{k_3,j_3}^{\sharp},f_{k_4,j_4}^{\sharp})$ is 
\[2^{-2k_3}2^{k_{thd}/2}2^{(j_1+j_2+j_3+j_4)/2}2^{-j_{max}/2}\]
thanks to Lemma \ref{lem:tri-L2} (b-1). Hence, for $s \ge -\frac14$, by choosing $\frac38 < b < \frac12$, we can obtain
\[\begin{aligned}
\sum_{\substack{k_3 \ge 10 \\|k_2-k_3| \le 5 \\ 1 \le k_1, k_4 \le k_3 -10}}\sum_{j_1,j_2,j_3,j_4 \ge 0} 2^{(1+s)k_4}2^{-2sk_3}2^{-sk_1}2^{-b(j_1+j_2+j_3+j_4)}J_3(f_{k_1,j_1}^{\sharp},f_{k_2,j_2}^{\sharp},&f_{k_3,j_3}^{\sharp},f_{k_4,j_4}^{\sharp}) \\
&\lesssim \prod_{i=1}^4 \norm{f_i}_{L^2}.
\end{aligned}\]

The \emph{low $\times$low $\times$ low $\Rightarrow$low} interaction component can be directly controlled by the Cauchy-Schwarz inequality, since  the low frequency localized space $D^{\alpha}$ with $\alpha >1/2$ allows the $L^2$ integrability with respect to $\tau$-variables.

Collecting all, we therefore complete the proof of Proposition \ref{prop:tri1}.
\end{proof}

\begin{proposition}\label{prop:tri2}
	For $-1/4 \le s \le 0 $, there exists $b = b(s) < 1/2$ such that for all $\alpha > 1/2$, we have
	\begin{equation}\label{eq:trilinear2}
	\norm{\px(uvw)}_{Y^{s,-b}} \le c\norm{u}_{X^{s,b} \cap D^{\alpha}}\norm{v}_{X^{s,b} \cap D^{\alpha}}\norm{w}_{X^{s,b} \cap D^{\alpha}}.
	\end{equation}
\end{proposition}

\begin{proof}
Similarly as in the proof of Proposition \ref{prop:bi2}, it is enough to consider the case when $|\tau| \le \frac{1}{2}|\xi|^5$, which ensures \eqref{eq:same size} (we recall here)
\begin{equation}\label{eq:same size.tri}
	|\tau-\xi^5| \sim |\xi|^5. %	 \quad \mbox{and} \quad |\tau-\frac{1}{16}\xi^5| \sim |\xi|^5.
\end{equation} 
Let $f_i$, $i=1,2,3$, be $L^2$-functions defined in \eqref{eq:uvw} under \eqref{eq:uvw2} and \eqref{eq:weight1}. Then, \eqref{eq:trilinear2} is equivalent to
	\begin{equation}\label{eq:trilinear2-1}
	\iint\limits_{\substack{\xi_1+\xi_2+\xi_3=\xi \\ \tau_1+\tau_2 + \tau_3 =\tau}} \frac{|\xi|\bra{\tau}^{\frac{s}{5}}\wt{f}_1(\tau_1,\xi_1)\wt{f}_2(\tau_2,\xi_2)\wt{f}_3(\tau_3,\xi_3)\wt{f}_4(\tau,\xi)}{\bra{\xi_1}^s\bra{\xi_2}^s\bra{\xi_3}^s\bra{\xi}^{5b}\beta_1(\tau_1,\xi_1)\beta_2(\tau_2,\xi_2)\beta_2(\tau_3,\xi_3)} \lesssim \prod_{i=1}^4\norm{f_i}_{L^2}.
	\end{equation}
Due to the symmetry, we may assume $|\xi_1| \le |\xi_2| \le |\xi_3|$ without loss of generality.
	
\textbf{Case I} (high $\times$ high $\times$ high $\Rightarrow$ high). $|\xi_1| > 1$, $|\xi_1| \sim |\xi_3| \sim |\xi|$. From \eqref{eq:same size.tri} and \eqref{eq:tri-resonant function}, we know
\[|\tau-\xi^5| \sim |\xi|^5 \gg |\xi|^2|\xi_1+\xi_2||\xi_2+\xi_3||\xi_3+\xi_1| \sim |G|.\]
Taking the Cauchy-Schwarz inequality to the left-hand side of \eqref{eq:trilinear2-1} in addition to \eqref{eq:symmetry4}, it suffices to show
%\begin{equation}\label{eq:trilinear2-2}
\[\sup_{\substack{\xi,\tau \in \R\\ |\tau| \le \frac{1}{2}|\xi|^5}} |\xi|^{1-3s-5b}\Big(\iint\limits_{\substack{|\xi_1|, |\xi_2| \le |\xi| \\ \tau_1, \tau_2 \in \R}} \frac{d\xi_1d\xi_2\;d\tau_1d\tau_2}{\bra{\tau_1-\xi_1^5}^{2b}\bra{\tau_2-\xi_2^5}^{2b}\bra{\tau_1-\xi_1^5 + (\tau_2-\xi_2^5) - \Sigma_1}^{2b}} \Big)^{1/2} \le c,\]
%\end{equation}
where
\begin{equation}\label{eq:Sigma1}
\Sigma_1 = \tau-\xi^5 + G(\xi_1,\xi_2,\xi-\xi_1-\xi_2).
\end{equation}
Since 
\[\iint\limits_{ \tau_1, \tau_2 \in \R} \frac{d\tau_1d\tau_2}{\bra{\tau_1-\xi_1^5}^{2b}\bra{\tau_2-\xi_2^5}^{2b}\bra{\tau_1-\xi_1^5 + (\tau_2-\xi_2^5) - \Sigma_1}^{2b}} \lesssim \bra{\tau-\xi^5}^{2-6b}\sim|\xi|^{5(2-6b)}\]
thanks to \eqref{eq:integral1} and \eqref{eq:integral3} for $\frac13 < b < \frac12$, for $-1 < s \le 0$, the choice $\max(\frac13, \frac{7-3s}{20}) < b < \frac12$ ensures
\[\eqref{eq:trilinear2-1} \lesssim \sup_{\xi \in \R}|\xi|^{7-3s-20b}\prod_{i=1}^4\norm{f_i}_{L^2} \lesssim \prod_{i=1}^4\norm{f_i}_{L^2}.\]

\textbf{Case II} (low $\times$ high $\times$ high $\Rightarrow$ high). $|\xi_2| > 1$, $|\xi_1| \ll |\xi_2| \sim |\xi_3| \sim |\xi|$. When $|\xi_1| \le 1$, the left-hand side of \eqref{eq:trilinear2-1} is bounded by
	\begin{equation}\label{eq:trilinear2-3}
	\iint\limits_{\ast} \frac{|\xi|^{1-2s-5b}\wt{f}_1(\tau_1,\xi_1)\wt{f}_2(\tau_2,\xi_2)\wt{f}_3(\tau_3,\xi_3)\wt{f}_4(\tau,\xi)}{\bra{\tau_1}^{\alpha}\bra{\tau_2-\xi_2^5}^{b}\bra{\tau_3-\xi_3^5}^{b}}
	\end{equation}
	where
	\[\ast = \{(\tau_1,\tau_2, \tau_3,\tau,\xi_1,\xi_2,\xi_3,\xi) \in \R^8: \xi_1+\xi_2 + \xi_3=\xi,\;  \tau_1+\tau_2 + \tau_3=\tau,\;  |\xi_1| < 1 \; |\xi_2| \sim |\xi_3| \sim  |\xi| \ge 1\}.\]
From \eqref{eq:symmetry4} and \eqref{eq:integral1}, we have for fixed $\tau, \xi, \tau_1, \xi_1$ that
\[\begin{aligned}
\iint\limits_{|\xi_2| \sim |\xi|, \tau_2} \frac{\wt{f}_2(\tau_2,\xi_2)}{\bra{\tau_2-\xi_2^5}^{b}\bra{\tau_3-\xi_3^5}^{b}} &\lesssim \norm{f_2}_{L^2}\left(\iint\limits_{|\xi_2| \sim |\xi|, \tau_2} \frac{d\xi_2d\tau_2}{\bra{\tau_2-\xi_2^5}^{2b}\bra{\tau_2-\xi_2^5  + \Sigma_2}^{2b}}\right)^{\frac12} \\
&\lesssim \norm{f_2}_{L^2} \left(\int_{|\xi_2| \sim |\xi|}\bra{\Sigma_2}^{1-4b} \; d\xi_2 \right)^{\frac12} \\
&\lesssim |\xi|^{\frac12}\norm{f_2}_{L^2},
\end{aligned}\]
for $\frac14 < b < \frac12$ ($\bra{\Sigma_2}^{1-4b} \lesssim 1$), where
\[\Sigma_2 = (\tau_1-\xi_1^5) - (\tau-\xi^5) - G(\xi_1,\xi_2,\xi-\xi_1-\xi_2).\]
Hence, for $-\frac12 < s \le 0$, the choice $\max(\frac14, \frac{3-4s}{10}) < b < \frac12$ in addition to $\alpha > \frac12$ yields
\[\eqref{eq:trilinear2-3} \lesssim \norm{f_1}_{L^2}\norm{f_2}_{L^2}\norm{f_3}_{L^2}\norm{f_4}_{L^2}.\]
	
When $|\xi_1| > 1$, we know from \eqref{eq:same size.tri} and \eqref{eq:tri-resonant function} that
\[|G| \sim |\xi|^5 \sim |\tau-\xi^5|.\]
Taking the Cauchy-Schwarz inequality in addition to \eqref{eq:symmetry4}, it suffices to show
\begin{equation}\label{eq:trilinear2-4}
\sup_{\substack{\xi,\tau \in \R\\ |\tau| \le \frac{1}{2}|\xi|^5}} |\xi|^{1-3s-5b}\Big(\iint\limits_{\substack{|\xi_1|, |\xi_2| \le |\xi| \\ \tau_1, \tau_2 \in \R}} \frac{d\xi_1d\xi_2\;d\tau_1d\tau_2}{\bra{\tau_1-\xi_1^5}^{2b}\bra{\tau_2-\xi_2^5}^{2b}\bra{\tau_1-\xi_1^5 + (\tau_2-\xi_2^5) - \Sigma_1}^{2b}} \Big)^{1/2} \le c,
\end{equation}
where $\Sigma_1$ is defined in \eqref{eq:Sigma1}. Let $\mu = \Sigma_1$. Since 
\[|\partial_{\xi_1} \Sigma_1| = |-5\xi_1^4 + 5(\xi-\xi_1-\xi_2)^4| \sim |\xi|^4,\]
we have for $\frac13 < b < \frac12$ that
\begin{equation}\label{eq:trilinear2-5}
\begin{aligned}
\int_{|\xi_1| \le |\xi|} \frac{d\xi_1}{\bra{\tau - \xi^5 + G(\xi_1,\xi_2,\xi-\xi_1-\xi_2)}^{6b-2}} &\sim \int_{|\mu| \le |\tau - \xi^5|} \frac{ |\xi|^{-4} \; d\mu}{\bra{\mu}^{6b-2}}\\
&\lesssim |\xi|^{-4}|\tau-\xi^5|^{3-6b} \sim |\xi|^{11-30b},
\end{aligned}
\end{equation}
where $\tau, \xi_2, \xi$ are fixed. Hence, by \eqref{eq:integral1}, \eqref{eq:integral3}, \eqref{eq:trilinear2-5} and the Cauchy-Schwarz inequality, we have
\[\begin{aligned} 
\Big(\iint\limits_{\substack{|\xi_1| \le |\xi|, |\xi_2| \sim |\xi| \\ \tau_1, \tau_2 \in \R}}& \frac{d\xi_1d\xi_2\;d\tau_1d\tau_2}{\bra{\tau_1-\xi_1^5}^{2b}\bra{\tau_2-\xi_2^5}^{2b}\bra{\tau_1-\xi_1^5 + (\tau_2-\xi_2^5) - \Sigma_1}^{2b}} \Big)^{1/2}\\
&\lesssim \Big(\int_{|\xi_2| \sim |\xi|}\int_{|\xi_1| \le |\xi| } \frac{d\xi_1\;d\xi_2}{\bra{\tau - \xi^5 + G(\xi_1,\xi_2,\xi-\xi_1-\xi_2)}^{6b-2}} \Big)^{\frac12} \\
&\lesssim |\xi|^{6-15b},
\end{aligned}\]
for $\frac13 < b < \frac12$, which implies that for $-1 < s \le 0$, the choice $\max(\frac13, \frac{7-3s}{20}) < b < \frac12$ ensures \eqref{eq:trilinear2-4}.
	
\textbf{Case III} (high $\times$ high $\times$ high $\Rightarrow$ low). $|\xi_1| > 1$, $|\xi| \ll |\xi_1| \sim |\xi_3|$. When $|\xi| \le 1$, we know from \eqref{eq:same size.tri} and \eqref{eq:tri-resonant function} that
\[|\tau - \xi^5| \lesssim 1 \ll |\xi_3|^5 \sim |G|.\]
We assume that $|\tau_1 - \xi_1^5| \le |\tau_2 - \xi_2^5| \le |\tau_3 - \xi_3^5|$. Taking the Cauchy-Schwarz inequality to the left-hand side of \eqref{eq:trilinear2-1} in addition to \eqref{eq:symmetry4}, it suffices to show
\begin{equation}\label{eq:trilinear2-6}
\sup_{|\xi_3| > 1, \tau_3 \in \R} \frac{|\xi_3|^{-3s}}{\bra{\tau_3 - \xi_3^5}^b}\Big(\iint\limits_{\substack{|\xi| \le 1, |\xi_1| \sim  |\xi_3| \\ \tau, \tau_1 \in \R}} \frac{d\xi d\xi_1\;d\tau d\tau_1}{\bra{\tau_1-\xi_1^5}^{2b}\bra{\tau-\xi^5}^{2b}\bra{\tau-\xi^5 - (\tau_1-\xi_1^5) - \Sigma_3}^{2b}} \Big)^{1/2} \le c,
\end{equation}
due to $\bra{\tau-\xi^5} \sim 1$, where 
\begin{equation}\label{eq:Sigma3}
\Sigma_3 = \tau_3-\xi_3^5 - G(\xi_1,\xi - \xi_1 - \xi_3,\xi_3).
\end{equation}
Using \eqref{eq:integral1} and \eqref{eq:integral3} for $\frac13 < b < \frac12$, we have
\[\mbox{LHS of } \eqref{eq:trilinear2-6} \lesssim \sup_{|\xi_3| > 1, \tau_3 \in \R} \frac{|\xi_3|^{-3s}}{\bra{\tau_3 - \xi_3^5}^b}\Big(\int\limits_{|\xi| \le 1, |\xi_1| \sim  |\xi_3| } \frac{d\xi d\xi_1}{\bra{\Sigma_3}^{6b-2}} \Big)^{1/2}.\]
If $|\tau_3 - \xi_3^5| \gg |G| \sim |\xi_3|^5$, for $-\frac32 < s \le 0$, we choose $\max(\frac13, \frac{11}{40} - \frac{3s}{20}) < b < \frac12$ so that the Cauchy-Schwarz inequality yields
\[\mbox{LHS of } \eqref{eq:trilinear2-6} \lesssim \sup_{|\xi_3| > 1, \tau_3 \in \R} |\xi_3|^{-3s}\bra{\tau_3 - \xi_3^5}^{1-4b}|\xi_3|^{\frac12} \lesssim \sup_{|\xi_3| > 1} |\xi_3|^{\frac{11}{2}-3s-20b} \lesssim 1. \]
Otherwise ($|\tau_3 - \xi_3^5| \sim |G| \sim |\xi_3|^5$), let $\mu = \Sigma_3$. Since 
\[|\partial_{\xi} \Sigma_3| = |5\xi^4 -5(\xi-\xi_1-\xi_3)^4| \sim |\xi_3|^4,\]
we have for $\frac13 < b < \frac12$ that
\begin{equation}\label{eq:trilinear2-7}
\begin{aligned}
\int_{|\xi| \le 1, |\xi_1| \sim |\xi_3|} \frac{d\xi \; d\xi_1}{\bra{\tau_3-\xi_3^5 + G(\xi_1,\xi - \xi_1 - \xi_2,\xi_3)}^{6b-2}} &\sim \int_{|\xi_1| \sim |\xi_3|}\int_{|\mu| \le |\xi_3|^5} \frac{ |\xi_3|^{-4} \; d\mu}{\bra{\mu}^{6b-2}}\;d\xi_1\\
&\lesssim |\xi_3|^{-4}|\xi_3|^{5(3-6b)}|\xi_3| \sim |\xi_3|^{12-30b}.
\end{aligned}
\end{equation}
For $-\frac43 < s \le 0$, the choice $\max(\frac13, \frac{6-3s}{20}) < b < \frac12$ in addition to \eqref{eq:trilinear2-7} yields
\[\mbox{LHS of } \eqref{eq:trilinear2-6} \lesssim \sup_{|\xi_3| > 1} |\xi_3|^{-3s}|\xi_3|^{-5b}|\xi_3|^{6-15b} \lesssim 1. \]
We remark in the above argument that our assumption $|\tau_1 - \xi_1^5| \le |\tau_2 - \xi_2^5| \le |\tau_3 - \xi_3^5|$ does not lose the generality. 

When $|\xi| > 1$, a direct computation for $-\frac52 < -1 <s \le 0$ gives
\begin{equation}\label{eq:tau-xi}
\frac{|\xi|\bra{\tau}^{\frac{s}{5}}}{\bra{\xi}^{5b}} \lesssim \frac{|\xi|^{1+s}}{\bra{\tau}^b}.
\end{equation}
We assume $|\tau_1 - \xi_1^5| \le |\tau_2 - \xi_2^5| \le |\tau_3 - \xi_3^5|$. If $|\tau_3 - \xi_3^5| \gg |G|$, we know
\[\bra{-\xi^5 - (\tau_3 - \xi_3^5) + G)} \sim |\tau_3 - \xi_3^5| \quad \mbox{and} \quad |\xi|^2|\xi_3| \ll |G| \ll |\tau_3-\xi_3^5|.\]
Similarly, thanks to \eqref{eq:integral1} and \eqref{eq:integral3} for $\frac13 < b < \frac12$, it suffices to show 
\begin{equation}\label{eq:trilinear2-8}
 \sup_{|\xi_3| > 1, \tau_3 \in \R} \frac{|\xi_3|^{-3s}}{\bra{\tau_3 - \xi_3^5}^b}\Big(\int\limits_{\substack{1 < |\xi| \lesssim  |\tau_3-\xi_3^5|^{\frac12}|\xi_3|^{-\frac32} \\ |\xi_1| \sim  |\xi_3|}} \frac{|\xi|^{2+2s} d\xi d\xi_1}{\bra{- \xi^5 - \Sigma_3}^{6b-2}} \Big)^{1/2} \lesssim 1.
\end{equation}
For $-\frac{7}{18}< s \le 0$, the choice $\frac{7}{16} < b < \frac12$ ensures
\[\mbox{LHS of } \eqref{eq:trilinear2-8} \lesssim \sup_{|\xi_3| > 1, \tau_3 \in \R}|\xi_3|^{-3s + \frac12 -\frac32(\frac{3+2s}{2})}|\tau_3 - \xi_3^5|^{1-4b + \frac{3+2s}{4}} \lesssim 1.\]
Otherwise ($|\tau_3 - \xi_3^5| \sim |G| \sim |\xi_3|^5$), similarly as before, it suffices to show
\begin{equation}\label{eq:trilinear2-8}
 \sup_{|\xi_3| > 1} |\xi_3|^{-3s-5b}\Big(\int\limits_{\substack{1 < |\xi| \lesssim  |\xi_3| \\ |\xi_1| \sim  |\xi_3|}} \frac{|\xi|^{2+2s} d\xi d\xi_1}{\bra{- \xi^5 - \Sigma_3}^{6b-2}} \Big)^{1/2} \lesssim 1
\end{equation}
for $\frac13 < b < \frac12$. Let $\mu = -\xi^5 - \Sigma_3$. Since $|\mu| \lesssim |G| \sim |\xi_3|^5$ and 
\[|\partial_{\xi} (-\xi^5 - \Sigma_3)| = |5(\xi-\xi_1-\xi_3)^4| \sim |\xi_3|^4,\]
we have for $\frac13 < b < \frac12$ that
\begin{equation}\label{eq:trilinear2-9}
\begin{aligned}
\int_{\substack{1 < |\xi| \lesssim  |\xi_3| \\ |\xi_1| \sim  |\xi_3|}} \frac{d\xi \; d\xi_1}{\bra {- \xi^5 - \Sigma_3}^{6b-2}} &\sim \int_{|\xi_1| \sim |\xi_3|}\int_{|\mu| \le |\xi_3|^5} \frac{ |\xi_3|^{-4} \; d\mu}{\bra{\mu}^{6b-2}}\;d\xi_1\\
&\lesssim |\xi_3|^{-4}|\xi_3|^{5(3-6b)}|\xi_3| \sim |\xi_3|^{12-30b}.
\end{aligned}
\end{equation}
For $-\frac32 < s \le 0$, the choice $\max(\frac13, \frac{7-2s}{20}) < b < \frac12$ in addition to \eqref{eq:trilinear2-9} ensures
\[\mbox{LHS of } \eqref{eq:trilinear2-8} \lesssim \sup_{|\xi_3| > 1}|\xi_3|^{7-2s-20b} \lesssim 1.\]

\textbf{Case IV} (low $\times$ low $\times$ high $\Rightarrow$ high). $|\xi_3| > 1$, $|\xi_2| \ll |\xi_3| \sim |\xi_3|$. When $|\xi_2| \le 1$, the left-hand side of \eqref{eq:trilinear2-1} is bounded by
	\begin{equation}\label{eq:trilinear2-10}
	\iint\limits_{\ast} \frac{|\xi|^{1-s-5b}\wt{f}_1(\tau_1,\xi_1)\wt{f}_2(\tau_2,\xi_2)\wt{f}_3(\tau_3,\xi_3)\wt{f}_4(\tau,\xi)}{\bra{\tau_1}^{\alpha}\bra{\tau_2}^{\alpha}\bra{\tau_3-\xi_3^5}^{b}}
	\end{equation}
	where
	\[\ast = \{(\tau_1,\tau_2, \tau_3,\tau,\xi_1,\xi_2,\xi_3,\xi) \in \R^8: \xi_1+\xi_2 + \xi_3=\xi,\;  \tau_1+\tau_2 + \tau_3=\tau,\;  |\xi_1|\le |\xi_2| < 1,\; 1 < |\xi_3| \sim  |\xi| \}.\]
Since $\bra{\tau_3-\xi_3^5}^{-b} \lesssim 1$, for $-\frac32 < s \le 0$, the choice $\frac{1-s}{5} < b < \frac12$ in addition to $\alpha > \frac12$ yields
\[\eqref{eq:trilinear2-10} \lesssim \norm{f_1}_{L^2}\norm{f_2}_{L^2}\norm{f_3}_{L^2}\norm{f_4}_{L^2}.\]

When $|\xi_1| \le 1 < |\xi_2|$, the exact same argument used in the \textbf{Case II} for $|\xi_1| \le 1$ can be directly applied to this case, hence we omit the detail.

When $1 < |\xi_1| \le |\xi_2|$, we know from \eqref{eq:tri-resonant function} and \eqref{eq:same size.tri} that
\[|G| \sim |\xi|^4|\xi_1+\xi_2| \ll |\xi|^5 \sim |\tau-\xi^5|.\]
The similar argument used in \textbf{Case I} can be applied to this case, hence we omit the detail.

\textbf{Case V} (low $\times$ high $\times$ high $\Rightarrow$ low). $|\xi_2| > 1$, $|\xi_1|, |\xi| \ll |\xi_2| \sim |\xi_3|$. We may assume that 
\begin{equation}\label{eq:trilinear2-11}
|\tau_2 - \xi_2^5| \le |\tau_3 - \xi_3^5|
\end{equation} 
without loss of generality. We split this case into several cases.

\textbf{Case V-a} $|\xi|, |\xi_1| \le 1$. In this case, we know $|\tau| \le 1$ and $\bra{\tau-\xi^5} \sim 1$. When $|\xi_2 + \xi_3| \le |\xi_3|^{-\frac32}$, we have from \eqref{eq:integral1} and the Cauchy-Schwarz inequality that
\[\begin{aligned}
\mbox{LHS of } \eqref{eq:trilinear2-1} &\lesssim \sup_{|\xi_3| \ge 1}|\xi_3|^{-2s} \Big(\iint\limits_{\substack{|\xi_1| \le 1, \; |\xi_2 + \xi_3| \le |\xi_3|^{-\frac32} \\ \tau_1 , \; \tau_2}} \frac{d\tau_1d\tau_2d\xi_1d\xi_2}{\bra{\tau_1}^{2\alpha}\bra{\tau_2-\xi_2^5}^{2b}\bra{\tau_2-\xi_2^5 + \Sigma_4}^{2b}}\Big)^{\frac12}\prod_{i=1}^4\norm{f_i}_{L^2}\\
&\lesssim \sup_{|\xi_3| \ge 1}|\xi_3|^{-2s-\frac34}\prod_{i=1}^4\norm{f_i}_{L^2} \lesssim \prod_{i=1}^4\norm{f_i}_{L^2},
\end{aligned}\]
whenever $-\frac{3}{8} \le s \le 0$, $\frac14 < b < \frac12$ and $\alpha > \frac12$, where 
\[\Sigma_4 = (\tau_1 - \xi_1^5) + (\tau_3 - \xi_3^5) - G(\xi_1,\xi_2, \xi_3).\]
Otherwise ($|\xi_2 + \xi_3| > |\xi_3|^{-\frac32}$), we know from \eqref{eq:tri-resonant function} that
\[|G| \sim |\xi_3|^4|\xi_2+\xi_3| > |\xi_3|^{\frac52}.\]
If $|\tau_3 - \xi_3^5| \ll |G|$, we know from \eqref{eq:symmetry4} that $|\tau_1| \sim |G| > |\xi_3|^{\frac52}$. Similarly as before, we have from \eqref{eq:integral1} and the Cauchy-Schwarz inequality that
\[\begin{aligned}
\mbox{LHS of } \eqref{eq:trilinear2-1} &\lesssim \sup_{|\xi_3| \ge 1}|\xi_3|^{-2s} \Big(\iint\limits_{\substack{|\xi| \le 1, \; |\xi_2| \sim |\xi_3| \\ |\tau| \le 1, \; \tau_2}} \frac{d\tau d\tau_2d\xi d\xi_2}{\bra{\tau_1}^{2\alpha}\bra{\tau_2-\xi_2^5}^{2b}\bra{\tau_2-\xi_2^5 + \Sigma_4}^{2b}}\Big)^{\frac12}\prod_{i=1}^4\norm{f_i}_{L^2}\\
&\lesssim \sup_{|\xi_3| \ge 1}|\xi_3|^{-2s-\frac{5\alpha}{2}}|\xi_3|^{\frac12}\prod_{i=1}^4\norm{f_i}_{L^2} \lesssim \prod_{i=1}^4\norm{f_i}_{L^2},
\end{aligned}\]
whenever $-\frac{3}{8} \le s \le 0$, $\frac14 < b < \frac12$ and $\alpha > \frac12$. If $|\tau_3-\xi_3^5| \gtrsim |G| > |\xi_3|^{\frac52}$, for $-\frac58 < s$, by choosing $-\frac{4s}{5} < b < \frac12$, we, similarly, have from \eqref{eq:integral1} and the Cauchy-Schwarz inequality that
\[\begin{aligned}
\mbox{LHS of } \eqref{eq:trilinear2-1} &\lesssim \sup_{|\xi_3| \ge 1}|\xi_3|^{-2s} \Big(\iint\limits_{\substack{|\xi|, |\xi_1| \le 1\\ |\tau| \le 1, \; \tau_1}} \frac{d\tau d\tau_1d\xi d\xi_1}{\bra{\tau_1}^{2\alpha}\bra{\tau_2-\xi_2^5}^{2b}\bra{\tau_3-\xi_3^5}^{2b}}\Big)^{\frac12}\prod_{i=1}^4\norm{f_i}_{L^2}\\
&\lesssim \sup_{|\xi_3| \ge 1}|\xi_3|^{-2s-\frac{5b}{2}}\prod_{i=1}^4\norm{f_i}_{L^2} \lesssim \prod_{i=1}^4\norm{f_i}_{L^2},
\end{aligned}\]
whenever $\alpha > \frac12$.

\textbf{Case V-b} $|\xi_1| \le 1< |\xi|$. We use \eqref{eq:tau-xi} for $-\frac52 < -1 <s \le 0$. We know from \eqref{eq:tri-resonant function} and \eqref{eq:same size.tri} that
\begin{equation}\label{eq:trilinear2-12}
|\tau-\xi^5| \sim |\xi|^5 \ll |\xi_3|^4|\xi| \sim |G|.
\end{equation}
If $|\tau_3 - \xi_3^5| \gtrsim |G|$, since $|\xi| \lesssim |\tau_3 - \xi_3^5|^{\frac15}$, for $-\frac23 < s \le 0$, by choosing $\max(\frac14, \frac{3+2s}{10},\frac{3-3s}{10}) < b < \frac12$, we have from \eqref{eq:integral1} and the Cauchy-Schwarz inequality that
\[\begin{aligned}
\mbox{LHS of } \eqref{eq:trilinear2-1} &\lesssim \sup_{\substack{|\xi_3| \ge 1\\ \tau_3 \in \R}}\frac{|\xi_3|^{-2s}}{\bra{\tau_3-\xi_3^5}^b} \Big(\iint\limits_{\substack{|\xi_1| \le 1 < |\xi| \lesssim |\tau_3 - \xi_3^5|^{\frac15}\\ \tau, \; \tau_1}} \frac{|\xi|^{2+2s}d\tau d\tau_1d\xi d\xi_1}{\bra{\tau}^{2b}\bra{\tau_1}^{2\alpha}\bra{\tau + \Sigma_5}^{2b}}\Big)^{\frac12}\prod_{i=1}^4\norm{f_i}_{L^2}\\
&\lesssim \sup_{\substack{|\xi_3| \ge 1 \\ |\tau_3-\xi_3^5| \ge |\xi_3|^4}}|\xi_3|^{-2s}|\tau_3-\xi_3^5|^{-b}|\tau_3-\xi_3^5|^{\frac15(\frac{3+2s}{2})} \prod_{i=1}^4\norm{f_i}_{L^2} \\
&\lesssim \sup_{|\xi_3| \ge 1}|\xi_3|^{-2s - 4b +\frac{6+4s}{5}}\prod_{i=1}^4\norm{f_i}_{L^2} \lesssim \prod_{i=1}^4\norm{f_i}_{L^2},
\end{aligned}\]
where
\[\Sigma_5 = -\xi^5- (\tau_1 - \xi_1^5) - (\tau_3-\xi_3^5) + G(\xi_1,\xi - \xi_1 - \xi_2,\xi_3).\]
Otherwise ($|\tau_3 - \xi_3^5| \ll |G|$), we know from \eqref{eq:symmetry4} under \eqref{eq:trilinear2-11} and \eqref{eq:trilinear2-12} that $|\tau_1 - \xi_1^5| \sim |\tau_1| \sim |G| > |\xi_3|^4$. Then, the left-hand side of \eqref{eq:trilinear2-1} is bounded by
\begin{equation}\label{eq:trilinear2-13}
\sup_{|\xi_3| \ge 1}|\xi_3|^{-2s}\Big(\iint\limits_{\substack{|\xi_1| \ll |\xi_3|, \; |\xi_2| \sim |\xi_3|\\ \tau, \; \tau_2\\|\tau_1| \sim |G| > |\xi_3|^4}} \frac{|\xi|^{2+2s}d\xi d\xi_2 d\tau d\tau_2}{\bra{\tau}^{2b}\bra{\tau_1}^{2\alpha}\bra{\tau_2 - \xi_2^2}^{2b}\bra{\tau - (\tau_2 - \xi_2^5) + \Sigma_6}^{2b}} \Big)^{\frac12}\prod_{i=1}^4\norm{f_i}_{L^2}
\end{equation}
where
\[\Sigma_6 =  -\xi^5 -(\tau_1 - \xi_1^5) + G(\xi_1,\xi_2,\xi-\xi_1-\xi_2).\]
Let $\mu = \Sigma_6$. Since $|\mu| \lesssim |\tau_1| $ and 
\[|\partial_{\xi} (\Sigma_6)| = |5(\xi-\xi_1-\xi_2)^4| \sim |\xi_3|^4,\]
we have for $\frac13 < b < \frac12$ and $\alpha > \frac12$ that
\begin{equation}\label{eq:trilinear2-14}
\begin{aligned}
\iint\limits_{\substack{1 < |\xi| \lesssim  |\xi_3| \\ |\xi_2| \sim  |\xi_3| \\ |\tau_1| > |\xi_3|^4}} \frac{|\tau_1|^{-2\alpha}d\xi \; d\xi_2}{\bra {\Sigma_3}^{6b-2}} &\sim \int\limits_{\substack{|\xi_2| \sim |\xi_3|\\ |\tau_1| > |\xi_3|^4}}\int\limits_{|\mu| \le |\tau_1|} \frac{ |\tau_1|^{-2\alpha}|\xi_3|^{-4} \; d\mu}{\bra{\mu}^{6b-2}}\;d\xi_2\\
&\lesssim |\xi_3|^{-4}|\xi_3|^{4(3-6b-2\alpha)}|\xi_3| \sim |\xi_3|^{9-24b-8\alpha}.
\end{aligned}
\end{equation}
For $-\frac74 \le s \le 0$ and $\alpha > \frac12$, the choice $\max(\frac13 , \frac{7-2s}{24}, \frac{5-4s}{24}) < b  <\frac12$ in addition to \eqref{eq:trilinear2-14} ensures
\[\eqref{eq:trilinear2-13} \lesssim \sup_{|\xi_3| \ge 1}|\xi_3|^{-2s}\max(|\xi_3|^{1+s},1)|\xi_3|^{\frac{9-24b-8\alpha}{2}}\prod_{i=1}^4\norm{f_i}_{L^2} \lesssim \prod_{i=1}^4\norm{f_i}_{L^2}.\]

\textbf{Case V-c} $|\xi| \le 1< |\xi_1|$. We know from \eqref{eq:same size.tri} and \eqref{eq:tri-resonant function} that
\[|\tau| \lesssim 1, \quad |\tau - \xi^5| \lesssim 1 \ll |\xi_3|^4|\xi_1| \sim |G|.\]
If $|\tau_1 - \xi_1^5| \sim |G| \gg |\tau_3-\xi_3^5|$, since $\bra{\tau-\xi^5}^{-2b} \sim 1$ and $\bra{\tau_3 - \xi_3^5}^{-b} \lesssim 1$, the left-hand side of \eqref{eq:trilinear2-1} is bounded by
\begin{equation}\label{eq:trilinear2-15}
\sup_{|\xi_3| \ge 1}|\xi_3|^{-2s}\Big(\iint\limits_{\substack{|\xi| \le 1 < |\xi_1| \ll |\xi_3|\\ |\tau| \lesssim 1, \; |\tau_1-\xi_1^5| \sim |G| \sim  |\xi_3|^4|\xi_1|}} \frac{|\xi_1|^{-2s}d\xi d\xi_1 d\tau d\tau_1}{\bra{\tau-\xi^5}^{2b}\bra{\tau_1-\xi_1^5}^{2b}\bra{\tau - \xi^5 - (\tau_1 - \xi_1^5) - \Sigma_3}^{2b}} \Big)^{\frac12}\prod_{i=1}^4\norm{f_i}_{L^2},
\end{equation}
where $\Sigma_3$ is defined in \eqref{eq:Sigma3}. Note that $|\Sigma_3| \sim |G|$. For $-\frac43 < s \le 0$, we choose $\max(\frac13, \frac{19}{54}-\frac{s}{9}) < b < \frac12$. Using \eqref{eq:integral1} and \eqref{eq:integral3}, and taking the Cauchy-Schwarz inequality with respect to $\xi, \xi_1$, we have
\[\mbox{LHS of} \eqref{eq:trilinear2-15} \lesssim \sup_{|\xi_3| \ge 1}|\xi_3|^{-2s}|\xi_3|^{4(1-3b)}|\xi_3|^{-s+\frac32-3b}\prod_{i=1}^4\norm{f_i}_{L^2} \lesssim \prod_{i=1}^4\norm{f_i}_{L^2}.\]
Otherwise ($|\tau_3 - \xi_3^5| \sim |G|$), let $\mu = \Sigma_3$. Since $|\mu| \lesssim |\tau_3-\xi_3^5| $ and 
\[|\partial_{\xi_1} \Sigma_3| = |-5\xi_1^5 + 5(\xi-\xi_1-\xi_3)^4| \sim |\xi_3|^4,\]
we have for $\frac13 < b < \frac12$ that
\begin{equation}\label{eq:trilinear2-16}
\begin{aligned}
\int_{|\xi| \le 1 < |\xi_1| \lesssim  |\xi_3|} \frac{d\xi \; d\xi_1}{\bra {\Sigma_3}^{6b-2}} &\sim \int_{|\xi|  \le 1}\int_{|\mu| \le |\tau_3 - \xi_3^5|} \frac{ |\xi_3|^{-4} \; d\mu}{\bra{\mu}^{6b-2}}\;d\xi\\
&\lesssim |\xi_3|^{-4}|\tau_3-\xi_3^5|^{3-6b}.
\end{aligned}
\end{equation}
For $-\frac23 \le s \le 0$, the choice $\frac38 < b < \frac12$ in addition to the Cauchy-Schwarz inequality, \eqref{eq:integral1}, \eqref{eq:integral3} and\eqref{eq:trilinear2-16} ensures
\[\begin{aligned}
&\mbox{LHS of } \eqref{eq:trilinear2-1}\\
&\lesssim \sup_{\substack{|\xi_3| \ge 1\\ \tau_3 \in \R}}\frac{|\xi_3|^{-3s}}{\bra{\tau_3-\xi_3^5}^b}\Big(\iint\limits_{\substack{|\xi| \le 1 < |\xi_1| \ll |\xi_3|\\ |\tau| \lesssim 1, \; |\tau_3-\xi_3^5| \sim |G|}} \frac{d\xi d\xi_1 d\tau d\tau_1}{\bra{\tau-\xi^5}^{2b}\bra{\tau_1-\xi_1^5}^{2b}\bra{\tau - \xi^5 - (\tau_1 - \xi_1^5) - \Sigma_3}^{2b}} \Big)^{\frac12}\prod_{i=1}^4\norm{f_i}_{L^2} \\
&\lesssim \sup_{|\xi_3| \ge 1,\; |\tau_3 -\xi_3^5| \ge 1}|\xi_3|^{-3s-2}|\tau_3-\xi_3^5|^{\frac32 - 4b}\prod_{i=1}^4\norm{f_i}_{L^2} \lesssim \prod_{i=1}^4\norm{f_i}_{L^2}.
\end{aligned}\]

\textbf{Case V-d} $1< |\xi_1|, |\xi|$. In this case, we know from \eqref{eq:tri-resonant function} that $|G| \sim |\xi_3|^4|\xi_2+\xi_3|$. We first consider the case when $|\tau - \xi^5| \sim |\xi|^5 \gtrsim |G| \gtrsim |\tau_3 - \xi_3^5|$. If $|\xi| \le |\xi_3|^{\frac45}$, from \eqref{eq:tau-xi} for $-\frac52 < -1 <s \le 0$, it suffices to show
\begin{equation}\label{eq:trilinear2-17}
\sup_{|\xi_3| \ge 1} |\xi_3|^{-3s+\frac45(1+s)} \Big(\iint\limits_{\substack{|\xi| \le |\xi_3|^{\frac45} \\|\xi_1| \ll |\xi_3|\\ |\tau - \xi^5| \gtrsim |G|, \; \tau_2}} \frac{d\xi d\xi_1 d\tau d\tau_1}{\bra{\tau}^{2b}\bra{\tau_1-\xi_1^5}^{2b}\bra{\tau - \xi^5 - (\tau_1 - \xi_1^5) - \Sigma_3}^{2b}} \Big)^{\frac12} \le c,
\end{equation}
where $\Sigma_3$ is defined in \eqref{eq:Sigma3} and whenever $-1 < s \le 0$. Let $\mu = \xi^5 + \Sigma_3$. Since $|\mu| \lesssim |\xi|^5 \le |\xi_3|^4 $ and 
\[|\partial_{\xi_1} (\xi^5 + \Sigma_3)| = |-5\xi_1^5 + 5(\xi-\xi_1-\xi_3)^4| \sim |\xi_3|^4,\]
we have for $\frac13 < b < \frac12$ that
\begin{equation}\label{eq:trilinear2-18}
\begin{aligned}
\int_{|\xi| \le |\xi_3|^{\frac45}, \;|\xi_1| \ll |\xi_3|} \frac{d\xi \; d\xi_1}{\bra {\Sigma_3}^{6b-2}} &\sim \int_{|\xi|  \le  |\xi_3|^{\frac45}}\int_{|\mu| \le |\xi_3|^4} \frac{ |\xi_3|^{-4} \; d\mu}{\bra{\mu}^{6b-2}}\;d\xi \\
&\lesssim |\xi_3|^{-4}|\xi_3|^{4(3-6b)}|\xi_3|^{\frac45}.
\end{aligned}
\end{equation}
For $-\frac{4}{11} < s \le 0$, the choice $\max( \frac13, \frac{26-11s}{60}) < b < \frac12$ in addition to the Cauchy-Schwarz inequality, \eqref{eq:integral1}, \eqref{eq:integral3} and\eqref{eq:trilinear2-18} ensures
\[\mbox{LHS of } \eqref{eq:trilinear2-17} \lesssim \sup_{|\xi_3| \ge 1}|\xi_3|^{-3s+\frac45(1+s)}|\xi_3|^{-2}|\xi_3|^{2(3-6b)}|\xi_3|^{\frac25} \lesssim 1.\]
Otherwise ($|\xi_3|^{\frac45} < |\xi|$), we know $|\xi| \ll |\xi_3| \le |\xi|^{\frac54}$. Then, it suffices to show
\begin{equation}\label{eq:trilinear2-19}
\sup_{|\xi| \ge 1} |\xi|^{1-5b}|\xi|^{-\frac{15s}{4}} \Big(\iint\limits_{\substack{|\xi_2| \le |\xi|^{\frac54} \\|\xi_1| \ll |\xi_3|\\ |\tau - \xi^5| \gtrsim |G|, \;\tau_1,\; \tau_2}} \frac{d\xi_1 d\xi_2 d\tau_1 d\tau_2}{\bra{\tau_1-\xi_1^5}^{2b}\bra{\tau_2-\xi_2^5}^{2b}\bra{\tau_1 - \xi_1^5 + (\tau_2 - \xi_2^5) - \Sigma_1}^{2b}} \Big)^{\frac12} \le c,
\end{equation}
where $\Sigma_1$ is defined in \eqref{eq:Sigma1}. Let $\mu = \Sigma_1$. Since $|\mu| \lesssim |\xi|^5$ and 
\[|\partial_{\xi_1} \Sigma_1| = |-5\xi_1^5 + 5(\xi-\xi_1-\xi_2)^4| \sim |\xi_3|^4,\]
we have for $\frac13 < b < \frac12$ that
\begin{equation}\label{eq:trilinear2-20}
\begin{aligned}
\int_{|\xi_2| \le |\xi|^{\frac54},\; |\xi_1| \ll |\xi_3|} \frac{d\xi_1 \; d\xi_2}{\bra {\Sigma_1}^{6b-2}} &\sim \int_{|\xi_2|  \le  |\xi|^{\frac54}}\int_{|\mu| \le |\xi|^5} \frac{ |\xi_3|^{-4} \; d\mu}{\bra{\mu}^{6b-2}}\;d\xi \\
&\lesssim |\xi|^{-4}|\xi|^{5(3-6b)}|\xi|^{\frac54}.
\end{aligned}
\end{equation}
For $-\frac{23}{30} < s \le 0$, the choice $\max( \frac13, \frac{57-30s}{160}) < b < \frac12$ in addition to the Cauchy-Schwarz inequality, \eqref{eq:integral1}, \eqref{eq:integral3} and\eqref{eq:trilinear2-20} ensures
\[\mbox{LHS of } \eqref{eq:trilinear2-19} \lesssim \sup_{|\xi| \ge 1} |\xi|^{1-5b}|\xi|^{-\frac{15s}{4}} |\xi|^{-2}|\xi|^{\frac52(3-6b)}|\xi|^{\frac58}\lesssim 1.\]

Now we consider the case when $|\tau - \xi^5| \sim |\xi|^5 \ll |\xi_3|^4|\xi_2+\xi_3| \sim |G|$. If $|\tau_3 - \xi_3^5| \ll |G|$, we know from \eqref{eq:symmetry4} and \eqref{eq:trilinear2-11} that $|\tau_1 - \xi_1^5| \sim |G|$. Similarly as before, we can obtain for $\frac38 < b < \frac12$ that
\[\begin{aligned}
\sup_{|\tau_1 - \xi_1^5| \ge 1} |\tau_1 - \xi_1^5|^{-2b}&\Big(\iint_{\substack{1 \le |\xi_1| \ll |\xi_3| \\ \tau, \tau_2 \\ |\tau_1 - \xi_1^5| \sim |G| \gg |\tau_3 - \xi_3^5|}} \frac{d\tau\; d\tau_2 \; d\xi}{\bra{\tau}^{2b}\bra{\tau_1-\xi_1^5}^{2b}\bra{\tau_2-\xi_2^5}^{2b}\bra{\tau - \xi^5 - (\tau_2 - \xi_2^5) - \Sigma_7}^{2b}} \Big)\\
&\lesssim |\xi_3|^{-4}|\tau_1-\xi_1^5|^{3-6b} \lesssim |\xi_3|^{-4},
\end{aligned}\]
where
\[\Sigma_7 = \tau_1 - \xi_1^5 - G(\xi_1,\xi_2,\xi-\xi_1-\xi_2),\]
due to
\[|\partial_{\xi} (-\xi^5 - \Sigma)| = |-5(\xi-\xi_1-\xi_2)^4| \sim |\xi_3|^{4}.\]
This in addition to \eqref{eq:tau-xi} implies
\[\mbox{LHS of } \eqref{eq:trilinear2-1} \lesssim \sup_{|\xi_3| \ge 1} |\xi_3|^{1-2s}|\xi_3|^{-2}\left(\int_{|\xi_1| \ll |\xi_3|} \; d\xi_1\right)^{\frac12}\prod_{i=1}^4\norm{f_i}_{L^2}\prod_{i=1}^4\norm{f_i}_{L^2} \lesssim \prod_{i=1}^4\norm{f_i}_{L^2},\] 
for $-\frac14 \le s \le 0$. Note that the above argument does not depend on the choice of the maximum modulation among $|\tau_i-\xi_i^5|$, $i=1,2,3$, and hence we can have
\[\mbox{LHS of } \eqref{eq:trilinear2-1}  \prod_{i=1}^4\norm{f_i}_{L^2},\] 
for the case when $|\tau_3 - \xi_3^5| \gtrsim |G|$.

\textbf{Case VI} (low $\times$ low $\times$ low $\Rightarrow$ low) $|\xi_3|< 1$. We know from the identity \eqref{eq:symmetry3} that $|\xi| < 1$, which implies $|\tau| \lesssim 1$. Then, the left-hand side of \eqref{eq:trilinear2-1} is equivalent to
	\[\iint\limits_{\substack{\xi_1+\xi_2 + \xi_3=\xi \\ \tau_1+\tau_2 + \tau_3=\tau \\ |\xi_1|,|\xi_2|,|\xi_3|,|\xi| < 1}} \frac{\wt{f}_1(\tau_1,\xi_1)\wt{f}_2(\tau_2,\xi_2)\wt{f}_3(\tau_3,\xi_3)\wt{f}_4(\tau,\xi)}{\bra{\tau_1}^{\alpha}\bra{\tau_2}^{\alpha}\bra{\tau_3}^{\alpha}},\]
which is bounded by 
\[\norm{f_1}_{L^2}\norm{f_2}_{L^2}\norm{f_3}_{L^2}\norm{f_4}_{L^2}\]
by taking the Cauchy-Schwarz inequality, since $\alpha > 1/2$. 

Therefore we complete the proof of Proposition \ref{prop:tri2}.
\end{proof}

\section{Duhamel boundary forcing operator}\label{sec:Duhamel boundary forcing operator}
In this section, we introduce the Duhamel boundary forcing operator, which was introduced by Colliander and Kenig \cite{CK} and further developed by several researchers \cite{Holmerkdv, CC, Cavalcante, CK2018-1}, which helps to construct the solution operator involving the boundary forcing conditions. We, particularly, refer to \cite{CK2018-1} for the fifth-order KdV-type equation.
  
\subsection{Duhamel boundary forcing operator class}
We introduce the Duhamel boundary forcing operator associated to the linear fifth-order equation. Let 
\begin{equation}\label{eq:M}
M = \frac{1}{B(0)\Gamma(4/5)}.
\end{equation} 
For $f\in C_0^{\infty}(\mathbb{R}^+)$, define the boundary forcing operator $\mathcal{L}^0$ of order $0$
\begin{equation}\label{eq:BFO}	
\mathcal{L}^0f(t,x):=M\int_0^te^{(t-t')\partial_x^5}\delta_0(x)\mathcal{I}_{-\frac45}f(t')dt'.
\end{equation}
By the change of variable and \eqref{eq:oscil}, we represent \eqref{eq:BFO} by
%\begin{equation}\label{forcing}
\[	\mathcal{L}^0f(t,x)=M\int_0^t B\left(\frac{x}{(t-t')^{1/5}}\right)\frac{\mathcal{I}_{-\frac45}f(t')}{(t-t')^{1/5}}dt'.\]
%\end{equation}
Moreover, a straightforward calculation gives
\begin{equation}\label{eq:BFO1}
\mathcal{L}^0(\partial_t f)(t,x) = M\delta_0(x)\mathcal{I}_{-\frac45}f(t) + \partial_x^5\mathcal{L}^0f(t,x).
\end{equation}

We state the several lemmas associated to $\mathcal{L}^0f$ defined as in \eqref{eq:BFO}. We refer to \cite{CK2018-1} and references therein for the proofs. 
\begin{lemma}[Continuity and decay property of $\mathcal{L}^0f$ \cite{CK2018-1}]\label{continuity}
	Let $f\in C_{0,c}^{\infty}(\mathbb{R}^+)$.
	\begin{itemize}
\item[(a)] For fixed $ 0 \le t \le 1$, $\partial_x^k \mathcal{L}^0f(t,x)$, $k=0,1,2,3$, is continuous in $x \in \mathbb{R}$ and has the decay property in terms of the spatial variable as follows:
%\begin{equation}\label{eq:decay1}
\[|\partial_x^k \mathcal{L}^0f(t,x)| \lesssim_{N} \norm{f}_{H^{N+k}}\bra{x}^{-N}, \quad N \ge 0.\]
%\end{equation}
\item[(b)] For fixed $ 0 \le t \le 1$, $\partial_x^4\mathcal{L}^0f(t,x)$ is continuous in $x$ for $x\neq 0$ and is discontinuous at $x =0$ satisfying
%\begin{equation}\label{eq:discontinuity}
\[\lim_{x\rightarrow 0^{-}}\partial_x^4\mathcal{L}^0f(t,x)=c_1\mathcal{I}_{-4/5}f(t),\ \lim_{x\rightarrow 0^{+}}\partial_x^4\mathcal{L}^0f(t,x)=c_2\mathcal{I}_{-4/5}f(t)\]
%\end{equation}
for $c_1 \neq c_2$. $\partial_x^4\mathcal{L}^0f(t,x)$ also has the decay property in terms of the spatial variable
%\begin{equation}\label{eq:decay2}
\[|\partial_x^4\mathcal{L}^0f(t,x)| \lesssim_{N} \norm{f}_{H^{N+4}}\bra{x}^{-N}, \quad N \ge 0.\]
%\end{equation}
	\end{itemize}
In particular, we have $\mathcal{L}^0f(t,0) = f(t)$.
\end{lemma}

In the following, we give the generalization of the boundary forcing operator $\mathcal{L}^0f$ and its properties introduced in \cite{CK2018-1}.

Let $\mbox{Re}\ \lambda > 0$ and $g\in C_0^{\infty}(\mathbb{R}^+)$ be given. Define
%\begin{equation}\label{eq:BFOC}
\[\mathcal{L}_{\pm}^{\lambda}g(t,x)=\left[\frac{x_{\mp}^{\lambda-1}}{\Gamma(\lambda)}*\mathcal{L}^0\big(\mathcal{I}_{-\frac{\lambda}{5}}g\big)(t, \cdot)   \right](x),\]
%\end{equation}
where $*$ denotes the convolution operator. Note that $\mathcal{L}_{\pm}^{\lambda}$ is for the right / left half-line problem, respectively. With $\frac{x_{-}^{\lambda-1}}{\Gamma(\lambda)}=\frac{(-x)_+^{\lambda-1}}{\Gamma(\lambda)}$ (in the sense of distribution), we represent each of them by
 \begin{equation}\label{forcing2}
\mathcal{L}_{+}^{\lambda}g(t,x)=\frac{1}{\Gamma(\lambda)}\int_{x}^{\infty}(y-x)^{\lambda-1}\mathcal{L}^0\big(\mathcal{I}_{-\frac{\lambda}{5}}g\big)(t,y)dy.
\end{equation}
 and 
\begin{equation}\label{forcing1}
\mathcal{L}_{-}^{\lambda}g(t,x)=\frac{1}{\Gamma(\lambda)}\int_{-\infty}^x(x-y)^{\lambda-1}\mathcal{L}^0\big(\mathcal{I}_{-\frac{\lambda}{5}}g\big)(t,y)dy.
\end{equation}

For $\mbox{Re} \lambda > -5$, the integration by parts in \eqref{forcing2} and \eqref{forcing1}, the decay property in Lemma \ref{continuity} and \eqref{eq:BFO1} yield
\begin{equation}\label{classe22}
\begin{aligned}
\mathcal{L}_{+}^{\lambda}g(t,x)&=\left[\frac{x_{-}^{(\lambda+5)-1}}{\Gamma(\lambda+5)}*\px^5\mathcal{L}^0\big(\mathcal{I}_{-\frac{\lambda}{5}}g\big)(t, \cdot)   \right](x)\\
&=M\frac{x_{-}^{(\lambda+5)-1}}{\Gamma(\lambda+5)}\mathcal{I}_{-\frac{4}{5}-\frac{\lambda}{5}}g(t)-\int_{x}^{\infty}\frac{(y-x)^{(\lambda+5)-1}}{\Gamma(\lambda+5)}\mathcal{L}^0\big(\partial_t\mathcal{I}_{-\frac{\lambda}{5}}g\big)(t,y)dy
\end{aligned}
\end{equation}
and
%\begin{equation}\label{classe12}
\[\begin{aligned}
\mathcal{L}_{-}^{\lambda}g(t,x)&= \left[\frac{x_{+}^{(\lambda+5)-1}}{\Gamma(\lambda+5)}*\px^5\mathcal{L}^0\big(\mathcal{I}_{-\frac{\lambda}{5}}g\big)(t, \cdot)   \right](x)\\
&=-M\frac{x_{+}^{(\lambda+5)-1}}{\Gamma(\lambda+5)}\mathcal{I}_{-\frac{4}{5}-\frac{\lambda}{5}}g(t)+\int_{-\infty}^x\frac{(x-y)^{(\lambda+5)-1}}{\Gamma(\lambda+5)}\mathcal{L}^0\big(\partial_t\mathcal{I}_{-\frac{\lambda}{5}}g\big)(t,y)dy,
\end{aligned}\]
%\end{equation}
respectively. It, thus, immediately satisfies (in the sense of distributions)
\begin{equation*}
(\partial_t-\partial_x^5)\mathcal{L}_{-}^{\lambda}g(t,x)=M\frac{x_{+}^{\lambda-1}}{\Gamma(\lambda)}\mathcal{I}_{-\frac{4}{5}-\frac{\lambda}{5}}g(t)
\end{equation*}
and
\begin{equation*}
(\partial_t-\partial_x^5)\mathcal{L}_{+}^{\lambda}g(t,x)=M\frac{x_{-}^{\lambda-1}}{\Gamma(\lambda)}\mathcal{I}_{-\frac{4}{5}-\frac{\lambda}{5}}g(t).
\end{equation*}

\begin{lemma}[Spatial continuity and decay properties for $\mathcal{L}_{\pm}^{\lambda}g(t,x)$ \cite{CK2018-1}]\label{holmer1}
	Let $g\in C_0^{\infty}(\mathbb{R}^+)$ and $M$ be as in \eqref{eq:M}. Then, we have
	%\begin{equation}\label{eq:relation}
	\[\mathcal{L}_{\pm}^{-k}g=\partial_x^k\mathcal{L}^{0}\mathcal{I}_{\frac{k}{5}}g, \qquad k=0,1,2,3,4.\]
%	\end{equation}
	Moreover, $\mathcal{L}_{\pm}^{-4}g(t,x)$ is continuous in $x \in \R \setminus \set{0}$ and has a step discontinuity of size $Mg(t)$ at $x=0$. For $\lambda>-4$, $\mathcal{L}_{\pm}^{\lambda}g(t,x)$ is continuous in $x \in\mathbb{R}$. For $-4\leq\lambda\leq 1$ and  $0\leq t\leq 1$, $\mathcal{L}_{\pm }^{\lambda}g(t,x)$ satisfies the following decay bounds:
	\begin{align*}
	&|\mathcal{L}_{-}^{\lambda}g(t,x)|\leq c_{m,\lambda,g}\langle x\rangle^{-m},\; \text{for all}\quad x\leq 0 \quad \text{and} \quad m\geq0,\\ 
	&|\mathcal{L}_{-}^{\lambda}g(t,x)|\leq c_{\lambda,g}\langle x\rangle^{\lambda-1}, \quad \text{for all} \quad  x\geq 0.\\
	&|\mathcal{L}_{+}^{\lambda}g(t,x)|\leq c_{m,\lambda,g}\langle x\rangle^{-m}, \quad \text{for all} \quad x\geq 0 \quad \text{and} \quad m\geq0,\\ \intertext{and}
	&|\mathcal{L}_{+}^{\lambda}g(t,x)|\leq c_{\lambda,g}\langle x\rangle^{\lambda-1},\quad \text{for all}\quad  x\leq 0.
	\end{align*}
\end{lemma}

\begin{lemma}[Values of $\mathcal{L}_{+}^{\lambda}f(t,0)$ and $\mathcal{L}_{-}^{\lambda}f(t,0)$ \cite{CK2018-1}]\label{trace1}
	 For $\mbox{Re}\ \lambda>-4$,
%\begin{equation}\label{lr0}
	\[\mathcal{L}_{+}^{\lambda}f(t,0)=\frac{1}{B(0)\Gamma(4/5)}\frac{\cos\left(\frac{(1+4\lambda)\pi}{10}\right)}{5\sin\left(\frac{(1-\lambda)\pi}{5}\right)}f(t)\]
%	\end{equation}
	and
%\begin{equation}\label{ll0}
\[\mathcal{L}_{-}^{\lambda}f(t,0)= \frac{1}{B(0)\Gamma(4/5)}\frac{\cos\left(\frac{(1-6\lambda)\pi}{10}\right)}{5\sin\left(\frac{(1-\lambda)\pi}{5}\right)}f(t)\]
%\end{equation}
	\end{lemma}

\subsection{Linear version}\label{section4}
We consider the linearized equation of \eqref{eq:5kdv}.
\begin{equation}\label{eq:lin. 5kdv}
\partial_t u -\partial_x^5u =0.
\end{equation}
The unitary group associated to \eqref{eq:lin. 5kdv} as
\begin{equation*}
e^{t\partial_x^5}\phi(x)=\frac{1}{2\pi}\int e^{ix\xi}e^{it\xi^5}\hat{\phi}(\xi)d\xi,
\end{equation*}
allows 
%\begin{equation}\label{linear}
\[\begin{cases}
(\partial_t-\partial_x^5)e^{t\partial_x^5}\phi(x) =0,& (t,x)\in\mathbb{R}\times\mathbb{R},\\
e^{t\partial_x^5}\phi(x)\big|_{t=0}=\phi(x),& x\in\mathbb{R}.
\end{cases}\]
%\end{equation}
Recall $\mathcal{L}_{+}^{\lambda}$ in \eqref{classe22} for the right half-line problem. Let $a_j$ and $b_j$ be constants depending on $\lambda_j$, $j=1,2$, given by
%\begin{equation}\label{eq:entries}
\[a_j = \frac{1}{B(0)\Gamma\left(\frac45\right)}\frac{\cos\left(\frac{(1+4\lambda_j)\pi}{10}\right)}{5\sin\left(\frac{(1-\lambda_j)\pi}{5}\right)}\quad \text{and}\quad b_j = \frac{1}{B(0)\Gamma\left(\frac45\right)}\frac{\cos\left(\frac{(4\lambda_j-3)\pi}{10}\right)}{5\sin\left(\frac{(2-\lambda_j)\pi}{5}\right)}.\]
%\end{equation}
Let us choose $\gamma_1$ and $\gamma_2$ satisfying
\[ \left[\begin{array}{c}
f(t) \\
\mathcal{I}_{\frac15}g(t)  \end{array} \right]=
A
 \left[\begin{array}{c}
\gamma_1(t)\\
\gamma_2(t)   \end{array} \right], \]
where
\[A(\lambda_1,\lambda_2)=
\left[\begin{array}{cc}
a_1 & a_2 \\
	b_1 & b_2 \end{array} \right].\]

We choose an appropriate $\lambda_j$, $j=1,2$, such $A$ is invertible. Then, $u$ defined by
%\begin{equation}\label{eq:solution0}
\[u(t,x)= \mathcal{L}_+^{\lambda_1}\gamma_1(t,x)+\mathcal{L}_+^{\lambda_2}\gamma_2(t,x),\]
%\end{equation}
solves
%\begin{equation}\label{forçante00}
\[\begin{cases}
(\partial_t-\partial_x^5)u= 0,\\
u(0,x) = 0,\\
u(t,0) = f(t), \; \px u(t,0) = g(t).
\end{cases}\]
%\end{equation}

See Section 3 in \cite{CK2018-1} for more details.
 
\subsection{Nonlinear version}\label{section4-1}

The Duhamel inhomogeneous solution operator $\mathcal{D}$
\begin{equation*}
\mathcal{D}w(t,x)=\int_0^te^{(t-t')\partial_x^5}w(t',x)dt'
\end{equation*}
solves
%\begin{equation}\label{nonlinear}
\[\left \{
\begin{array}{l}
(\partial_t-\partial_x^5)\mathcal{D}w(t,x) =w(t,x),\ (t,x)\in\mathbb{R}\times\mathbb{R},\\
\mathcal{D}w(x,0) =0,\ x\in\mathbb{R}.
\end{array}
\right.\]
%\end{equation}

By choosing a suitable  $\gamma_1$ and $\gamma_2$ depending on not only $f$ and $g$, but also $e^{t\px^5}\phi(x)$ and $\mathcal{D}w$,  $u$ defined by
\[u(t,x)= \mathcal{L}_+^{\lambda_1}\gamma_1(t,x)+\mathcal{L}_+^{\lambda_2}\gamma_2(t,x)+ e^{t\px^5}\phi(x) + \mathcal{D}w\]
solves
%\begin{equation}\label{forçante000}
\[\begin{cases}
(\partial_t-\partial_x^5)u= w,\\
u(0,x) = \phi(x),\\
u(t,0) = f(t), \; \px u(t,0) = g(t).
\end{cases}\]
%\end{equation} 

We refer to \cite{CK2018-1} for more details.

\section{Energy estimates}\label{sec:energy}
We are going to give the fundamental energy estimates.
\begin{lemma}\label{grupo}
	Let $s\in\mathbb{R}$ and $0<T \le 1$. If $\phi\in H^s(\mathbb{R})$, then
	\begin{itemize}
		\item[(a)] \emph{(Space traces)} $\|\psi_T(t)e^{t\partial_x^5}\phi(x)\|_{C\big(\mathbb{R}_t;\,H^s(\mathbb{R}_x)\big)}\lesssim \|\phi\|_{H^s(\mathbb{R})}$;
		\item[(b)] \emph{((Derivatives) Time traces)} In particular, for $-\frac92+j \le s \le \frac12 +j$,
		\[
		\|\psi_T(t) \partial_x^{j}e^{t\partial_x^5}\phi(x)\|_{C(\mathbb{R}_x;H^{\frac{s+2-j}{5}}(\mathbb{R}_t))}\lesssim \|\phi\|_{H^s(\mathbb{R})}, \quad j\in\{0,1,2\};
		\]
		\item [(c)] \emph{(Bourgain spaces)} For $b \le \frac12$ and $\alpha > \frac12$, we have
		\[\|\psi_T(t)e^{t\partial_x^5}\phi(x)\|_{X^{s,b} \cap D^{\alpha}}\lesssim T^{\frac12 - \alpha} \|\phi\|_{H^s(\mathbb{R})}.\]
	\end{itemize}
	The implicit constants do not depend on $0 < T \le 1$ but $\psi$.
\end{lemma}

\begin{remark}
In contrast to IVP, the time localization may restrict the regularity range (both upper- and lower-bounds) due to the (derivatives) time trace estimates. Similar phenomenon can be seen in $X^{s,b}$ estimates (see, in particular, Lemma 2.11 in \cite{Tao2006}), while the modulation exponent $b$ is affected to be taken by the cut-off function. See (b) and (c) in Lemma \ref{grupo} for the comparison.
\end{remark}

\begin{proof}
The proofs of (a) and (c) are standard, hence we omit the details and refer to \cite{Tao2006}. 

\textbf{(b).} Let $\phi = \phi_1 + \phi_2$, where $\wh{\phi}_1(\xi) = \chi_{\le 1}(\xi)\wh{\phi}(\xi) $. For $\phi_1$, we observe
\[\ft_t[\psi_T \partial_x^{j}e^{t\partial_x^5}\phi_1](\tau,x) = \int e^{ix\xi} (i\xi)^j \wh{\phi}_1(\xi)\wh{\psi}_T(\tau-\xi^5) \; d\xi,\]
which yields
\begin{equation}\label{eq:h_t norm}
\|\psi_T \partial_x^{j}e^{t\partial_x^5}\phi_1\|_{H^{\frac{s+2-j}{5}}} = \left(\int \bra{\tau}^{\frac{2(s+2-j)}{5}}\left| \int e^{ix\xi} (i\xi)^j \wh{\phi}_1(\xi)\wh{\psi}_T(\tau-\xi^5) \; d\xi \right|^2 \; d\tau  \right)^{\frac12}.
\end{equation}
When $|\tau| \le 1$, it directly follows from $|\wh{\psi}_T(\tau-\xi^5) | \lesssim T$ and the Cauchy-Schwarz inequality that
\[\mbox{RHS of } \eqref{eq:h_t norm} \lesssim T \norm{\phi_1}_{H^s} \lesssim \norm{\phi}_{H^s}.\]
When $|\tau| > 1$, we know $|\wh{\psi}_T(\tau-\xi^5) | \sim |\wh{\psi}_T(\tau) |$, and hence we have
\[\mbox{RHS of } \eqref{eq:h_t norm} \lesssim \left(\int_{|\tau| > 1} |\tau|^{\frac{2(s+2-j)}{5}} |\wh{\psi}_T(\tau)|^2 \; d\tau \right)^{\frac12} \norm{\phi_1}_{H^s}.\]
Note that 
\begin{equation}\label{eq:L2-T}
\int_{|\tau| > 1} |\tau|^{2\sigma} |\wh{\psi}_T(\tau)|^2 \; d\tau = T^{1-2\sigma}\int_{|\tau| > T} |\tau|^{2\sigma} |\wh{\psi}(\tau)|^2 \; d\tau.
\end{equation}
Hence, for $\frac{s+2-j}{5} \le \frac12$, we have 
\[\mbox{RHS of } \eqref{eq:h_t norm} \lesssim \norm{\phi}_{H^s}.\]
For $\phi_2$, observe that
\[
%\begin{aligned}
\partial_x^j e^{t\px^5}\phi_2(x) = \int e^{ix\xi}(i\xi)^je^{it\xi^5} \wh{\phi}_2(\xi) \; d\xi =\int e^{it\eta}e^{ix\eta^{\frac15}}(i\eta^{\frac15})^{j}\wh{\phi}_2(\eta^{\frac15}) \; \frac{d \eta}{5\eta^{\frac45}}.
%\end{aligned}
\]
It implies
\begin{equation}\label{eq:b.2}
\norm{\psi_T \partial_x^j e^{t\px^5}\phi_2}_{H_t^{\frac{s+2-j}{5}}} = \left(\int \bra{\tau}^{\frac{2(s+2-j)}{5}}\left|\int e^{ix\eta^{\frac15}}(i\eta^{\frac15})^j\wh{\phi}_2(\eta^{\frac15}) \wh{\psi}_T(\tau-\eta) \; \frac{d \eta}{5\eta^{\frac45}}\right|^{2} \; d\tau \right)^{\frac12}. 
\end{equation}
It is known from the support of $\phi_2$ that $|\eta| > 1$. For $|\tau| \le 1$, the weight $\bra{\tau}^{\frac{2(s+2-j)}{5}}$ and the integration with respect to $\tau$ are negligible. If $T|\eta| \le 1$, we know $|\wh{\psi}_T(\tau-\eta)| \lesssim T$. Since
\[
\left|\int_{1<|\eta| < 1/T} |\eta|^{\frac j5}\wh{\phi}_2(\eta^{\frac15}) \; \frac{d \eta}{5\eta^{\frac45}} \right| \lesssim \max(1, T^{\frac{s+2-j}{5}-\frac12})\norm{\phi_2}_{H_x^s},
\]\footnote{When $\frac{s+2-j}{5}=\frac12$, the constant depending on $T$ is $\log \frac1T $ instead of $T^{\frac{s+2-j}{5}-\frac12}$, but it does not influence on our analysis.}
for $0 < T \le 1$, we have
\[
\norm{\psi_T \partial_x^j e^{t\px^5}\phi_2}_{H_t^{\frac{s+2-j}{5}}} \lesssim T^{\frac12 + \frac{s+2-j}{5}} \norm{\phi}_{H_x^s} \lesssim \norm{\phi}_{H_x^s},
\]
for $\frac12 + \frac{s+2-j}{5} \ge 0$. If $T|\eta| > 1$, we know $|\wh{\psi}_T(\tau-\eta)| \lesssim |\eta|^{-1}$. Since
\[
\left|\int_{1/T<|\eta| } |\eta|^{\frac j5 -1}\wh{\phi}_2(\eta^{\frac15}) \; \frac{d \eta}{5\eta^{\frac45}} \right| \lesssim \left(\int_{|\eta|>1/T}|\eta|^{-\frac25(s+2-j)-2} \;d\eta \right)^{\frac12} \norm{\phi}_{H_x^s},
\]
we have
\[\mbox{LHS of }\eqref{eq:b.2} \lesssim T^{\frac12 + \frac{s+2-j}{5}} \norm{\phi}_{H_x^s}\lesssim \norm{\phi}_{H_x^s},
\]
for $\frac12 + \frac{s+2-j}{5} \ge 0$. 

We may assume, from now on, that $|\tau|, |\eta| > 1$. For given $\tau$, we further divide the region in $\eta$ into the following:
\begin{equation}\label{eq:cases}
\mathbf{I. }\; |\eta| < \frac12|\tau|, \qquad \mathbf{II. }\; 2|\tau| < |\eta|, \qquad \mathbf{III. }\; \frac12|\tau| \le |\eta| \le 2|\eta|.
\end{equation}
The way to divide the integration region as \eqref{eq:cases} will be used repeatedly.

\textbf{I}. $|\eta| < \frac12|\tau|$. If $T|\tau| \le 1$, we know $|\wh{\psi}_T(\tau-\eta)| \lesssim T$. Since
\begin{equation}\label{eq:b.3}
\left|\int_{|\eta| \le |\tau| } |\eta|^{\frac j5 }\wh{\phi}_2(\eta^{\frac15}) \; \frac{d \eta}{5\eta^{\frac45}} \right|^2 \lesssim |\tau|^{-\frac25(s+2-j)+1} \norm{\phi_2}_{H_x^s}^2,
\end{equation}
we have
\[\mbox{LHS of }\eqref{eq:b.2} \lesssim T\left(\int_{1 < |\tau| \le 1/T} |\tau| \; d\tau\right)^{\frac12} \norm{f}_{H_x^s} \lesssim \norm{\phi}_{H_x^s}.\]
If $T|\tau| > 1$, we know $|\wh{\psi}_T(\tau-\eta)| \lesssim T^{1-k}|\tau|^{-k}$, for any positive $k$. For $k>1$, we have from \eqref{eq:b.3} that
\[
\begin{aligned}
\mbox{LHS of }\eqref{eq:b.2} \lesssim T^{1-k}\left(\int_{1/T < |\tau|} |\tau|^{-2k+1} \; d\tau\right)^{\frac12} \norm{\phi}_{H_x^s} \lesssim \norm{\phi}_{H_x^s}.
\end{aligned}\]
We remark that the smoothness of $\psi$ guarantees not the good bound of $T$, but the integrability, in other words, we only need a large $k$ for
\[\int_{1/T < |\tau|} |\tau|^{-2k+1} \; d\tau < \infty.\]

\textbf{II}. $2|\tau| < |\eta|$. If $T|\eta| < 1$, we know $|\wh{\psi}_T(\tau - \eta) | \lesssim T$. Then, the right-hand side of \eqref{eq:b.2} is bounded by
\begin{equation}\label{eq:b.4}
T\left(\int_{|\tau| \le 1/T} \left|\int_{|\eta| \le 1/T} |\eta|^{\frac{s+2-j}{5}}|\eta|^{\frac j5}\wh{\phi}_2(\eta^{\frac15}) \; \frac{d \eta}{5\eta^{\frac45}}\right|^{2} \; d\tau \right)^{\frac12}.
\end{equation}
Since 
\[
\int \left(|\eta|^{\frac{s+2-j}{5}}|\eta|^{\frac j5}\wh{\phi}_2(\eta^{\frac15})|\eta|^{-\frac45} \right)^2\; d \eta = c\int |\xi|^{2s}|\wh{\phi}_2|^2 \; d\xi,
\]
the Cauchy-Schwarz inequality with respect to $\eta$ and $\tau$ yields $\eqref{eq:b.4} \lesssim \norm{\phi}_{H^s}$. If $T|\eta| > 1$, we know $|\wh{\psi}_T(\tau - \eta) | \lesssim T^{1-k}|\eta|^{k}$, for any positive $k$. On the region $|\tau| < 1/T$, similarly as before, we have
\[
T^{1-k}\left(\int_{|\tau| \le 1/T} \left|\int_{|\eta| > 1/T} |\eta|^{-k}|\eta|^{\frac{s+2-j}{5}}|\eta|^{\frac j5}\wh{\phi}_2(\eta^{\frac15}) \; \frac{d \eta}{5\eta^{\frac45}}\right|^{2} \; d\tau \right)^{\frac12},
\]
which implies from the Cauchy-Schwarz inequality that
\[\mbox{LHS of } \eqref{eq:b.2} \lesssim T^{1-k}T^{-\frac12}\left(\int_{|\eta| > 1/T} |\eta|^{-2k} \; d\eta \right)^{\frac12} \norm{\phi}_{H^s} \lesssim \norm{\phi}_{H^s}.\]
On the region $|\tau| \ge 1/T$, it suffices to control
\[
T^{1-k}\left(\int_{|\tau| > 1/T} \left|\int_{|\tau| < |\eta|} |\eta|^{-k}|\eta|^{\frac{s+2-j}{5}}|\eta|^{\frac j5}\wh{\phi}_2(\eta^{\frac15}) \; \frac{d \eta}{5\eta^{\frac45}}\right|^{2} \; d\tau \right)^{\frac12},
\]
since
\[
\int_{|\tau| < |\eta|} |\eta|^{-2k}\; d\eta \lesssim |\tau|^{-2k+1},
\]
the Cauchy-Schwarz inequality and the change of variable yield
\begin{equation}\label{eq:b.5}
\mbox{LHS of } \eqref{eq:b.2} \lesssim T^{1-k}\left(\int_{|\tau| > 1/T} |\tau|^{-2k+1} \; d\eta \right)^{\frac12} \norm{\phi}_{H^s} \lesssim \norm{\phi}_{H^s}.
\end{equation}

\textbf{III}. $\frac12|\tau| < |\eta| < 2|\tau|$. In this case, we know
\[|\tau|^{\frac{s+2-j}{5}}|\eta|^{\frac{j-4}{5}} \sim |\eta|^{\frac{s-2}{5}}\]
in the integrand of the right-hand side of \eqref{eq:b.2}.
We further assume $|\tau|,|\eta| > 1/T$, otherwise we can use the same way to control \eqref{eq:b.4}. If $\tau \cdot \eta <0$, we know $|\wh{\psi}_T(\tau - \eta) | \lesssim T^{1-k}|\tau|^{k}$. Since
\[\left(\int_{|\eta| \sim |\tau|} 1 \; d\eta \right)^{\frac12} \lesssim |\tau|^{\frac12},\]
we have $\mbox{LHS of } \eqref{eq:b.2}\lesssim \norm{\phi}_{H^s}$ similarly as \eqref{eq:b.5}. If $\tau \cdot \eta > 0$, we further divide the case into $|\tau - \eta| < 1/T$ and $|\tau - \eta| > 1/T$. For the former case, let 
\begin{equation}\label{eq:b.6}
\Phi(\eta) = |\eta|^{\frac{s-2}{5}}\wh{\phi}_2(\eta^{\frac15}).
\end{equation}
We note that $\norm{\Phi}_{L^2} \sim \norm{\phi_2}_{H^s}$. Since $|\wh{\psi}_T(\tau - \eta) | \lesssim T$, we have
\[
\left|\int_{|\eta - \tau| < 1/T} |\eta|^{\frac{s-2}{5}}\wh{\phi}_2(\eta^{\frac15})\wh{\psi}_T(\tau - \eta) \; d\eta \right| \lesssim \left|T\int_{|\eta - \tau| < 1/T} \Phi(\eta) \; d\eta \right| \lesssim M\Phi(\tau),
\]
where $Mf(x)$ is the Hardy-Littlewood maximal function of $f$. Since $\norm{Mf}_{L^p} \lesssim \norm{f}_{L^p}$ for $1< p \le \infty$ (see, in particular, \cite{Stein1993}), we have
\[
\mbox{LHS of } \eqref{eq:b.2} \lesssim \left(\int |M\Phi(\eta)|^2 \; d\tau \right)^{\frac12} \lesssim \norm{\Phi}_{L^2} \lesssim \norm{\phi}_{H^s}.
\]
For the latter case, the integration region in $\eta$ can be reduced to $\tau + 1/T < \eta < 2\tau$ for positive $\tau$ and $\eta$, since the exact same argument can be applied to the other regions.\footnote{Indeed, we only have four regions; $\tau + \frac1T < \eta < 2\tau$ and $\frac12 \tau < \eta < \tau - \frac1T$ for positive $\tau, \eta$, and $\tau + \frac1T < \eta < \frac12\tau$ and $2\tau  < \eta < \tau - \frac1T$ for negative $\tau, \eta$, and the same argument can be applied on each region. \label{fn:4regions}} 
Since $|\wh{\psi}_T(\tau - \eta) | \lesssim T^{1-k}|\tau - \eta|^{-k}$ in this case, the left-hand side of \eqref{eq:b.2} is bounded by
\begin{equation}\label{eq:b.7}
T^{1-k}\left(\int_{|\tau| > 1/T} \left| \int_{\tau + 1/T}^{2\tau} \frac{\Phi(\eta)}{|\tau-\eta|^k} \; d\eta \right|^2 \; d\tau \right)^{\frac12},
\end{equation}
where $\Phi$ is defined as in \eqref{eq:b.6}. Let $\epsilon = (k-1)/2$ for $k > 1$. Then, the change of variable, the Cauchy-Schwarz inequality and the Fubini theorem yield
\[\begin{aligned}
\eqref{eq:b.7} &\lesssim T^{1-k}\left(\int_{|\tau| > 1/T} \left| \int_{1/T}^{\tau} \frac{\Phi(\tau + h)}{|h|^k} \; dh \right|^2 \; d\tau \right)^{\frac12}\\
&\lesssim T^{1-k}T^{\epsilon}\left(\int_{|\tau| > 1/T}\int_{|h| > 1/T} \frac{|\Phi(\tau + h)|^2}{|h|^{2k-1-2\epsilon}} \; dh d\tau \right)^{\frac12}\\
&\lesssim \norm{\Phi}_{L^2} \lesssim \norm{\phi}_{H^s}.
\end{aligned}\]
Therefore, we complete the proof of (b).
\end{proof}

\begin{remark}\label{rem:cut}
The proof of Lemma \ref{grupo} (b) exactly shows the proof of 
\[\norm{\psi_T(t)f(t)}_{H^{\sigma}} \lesssim \norm{f}_{H^{\sigma}},\]
whenever $-\frac12 \le \sigma \le \frac12$. With this, we have a variant of Lemma \ref{cut}
%\begin{equation}\label{eq:cut}
\[\norm{\psi_T(t)f(t)}_{H_0^{\frac{s+2}{5}}} \lesssim \norm{f(t)}_{H_0^{\frac{s+2}{5}}},\]
%\end{equation} 
whenever $-\frac12 < \frac{s+2}{5} < \frac12$.
\end{remark}

\begin{lemma}\label{duhamel}
Let $s\in \mathbb{R}$ and $0 < T \le 1$. For $0 < b < \frac12 < \alpha < 1-b$, there exists  $\theta = \theta(s,j,b,\alpha)$ such that
	\begin{itemize}
		\item[(a)] \emph{(Space traces)} 
		\begin{equation*}
		\|\psi_T(t)\mathcal{D}w(x,t)\|_{C\big(\mathbb{R}_t;\,H^s(\mathbb{R}_x)\big)}\lesssim T^{\theta}\|w\|_{X^{s,-b}};
		\end{equation*}
		\item[(b)] \emph{((Derivatives) Time traces)} for $ -\frac92+ j \le s \le \frac12 + j$, $j=0,1,2$,
		\begin{equation*}
		\begin{aligned}
		\|\psi_T(t) \partial_x^j&\mathcal{D}w(x,t)\|_{C(\mathbb{R}_x;H^{\frac{s+2-j}{5}}(\mathbb{R}_t))} \\
		&\lesssim	\begin{cases} T^{\theta}\|w\|_{X^{s,-b}}, \quad &\mbox{if } \; 0 \le \frac{s+2-j}{5} \le \frac12, \\ T^{\theta}(\|w\|_{X^{s,-b}}+\|w\|_{Y^{s,-b}}), \quad &\mbox{if }\; -\frac12 \le \frac{s+2-j}{5} \le 0\end{cases};
		\end{aligned}
		\end{equation*}
		\item[(c)] \emph{(Bourgain spaces estimates)} 
		\begin{equation*}
		\|\psi_T(t)\mathcal{D}w(x,t)\|_{X^{s,b} \cap D^{\alpha}}\lesssim T^{\theta} \|w\|_{X^{s,-b}}.
		\end{equation*}
	\end{itemize}
\end{lemma}

\begin{remark}\label{rem:Ysb}
In view of the proof of Lemma \ref{duhamel} (b), the intermediate norm $Y^{s,b}$ is needed only for the regularity region $-\frac12 \le \frac{s+2-j}{5} \le 0$ (equivalently, $-\frac92 + j \le s \le j-2$), and thus, the nonlinear estimates in $Y^{s,b}$ norm (Theorems \ref{thm:nonlinear1} (b) and \ref{thm:nonlinear2} (b)) for the negative regularities are enough for our analysis.
\end{remark}

\begin{remark}\label{rem:choice of theta}
In view of the proof of Lemma \ref{duhamel}, under the condition $-\frac12 \le \frac{s+2-j}{5} \le \frac12$, $j=0,1,2$, we can choose $\theta = 1- \alpha -b$ uniformly in $s$ and $j$ for $b < \frac12 < \alpha < \frac34 - b$ such that $T^{\frac32-2\alpha - b}$ can be small enough (by choosing small $T \ll 1$) to close the iteration argument. See Section \ref{sec:thm proof}.
\end{remark}

\begin{proof}[Proof of Lemma \ref{duhamel}]$\;$\\

\textbf{(a).} A direct calculation gives
\begin{equation}\label{eq:duhamel fourier}
\ft[\psi_T\mathcal{D}w](\tau,\xi) = c\int \wt{w}(\tau',\xi) \frac{\wh{\psi}_T(\tau-\tau') - \wh{\psi}_T(\tau - \xi^5)}{i(\tau' - \xi^5)} \; d\tau'.
\end{equation}
Since $\norm{\psi_T \mathcal{D}w}_{C_tH^s} \lesssim \norm{\bra{\xi}^s\ft[\psi_T\mathcal{D}w](\tau,\xi)}_{L_{\xi}^2L_{\tau}^1}$, it suffices to control
\begin{equation}\label{eq:a.1}
\left(\int \bra{\xi}^{2s}\left|\int |\wt{w}(\tau',\xi)| \int \frac{|\wh{\psi}_T(\tau-\tau') - \wh{\psi}_T(\tau - \xi^5)|}{|\tau' - \xi^5|} \; d\tau d\tau' \right|^2 \; d\xi \right)^{\frac12},
\end{equation}
due to \eqref{eq:duhamel fourier}. On the region $|\tau' - \xi^5| \le T^{-1}$, we note from mean value theorem that
\[\frac{|\wh{\psi}_T(\tau-\tau') - \wh{\psi}_T(\tau - \xi^5)|}{|\tau' - \xi^5|} = T^2 |\wh{\psi}'(T(\tau-\xi^5) + \sigma)|,\]
for small $\sigma$ depending on $\tau'$ and $\xi^5$. Since $T\wh{\psi}'(T\tau)$ is $L^1$ integrable with respect to $\tau$, the Cauchy-Schwarz inequality yields
\[
\begin{aligned}
\eqref{eq:a.1} &\lesssim T \left(\int \bra{\xi}^{2s}\left|\int_{|\tau'-\xi^5| \le 1/T} |\wt{w}(\tau',\xi)| d\tau' \right|^2 \; d\xi \right)^{\frac12}\\
&\lesssim T^{\frac12 - b}(1+T^{\frac12 + b}) \norm{w}_{X^{s,-b}}.
\end{aligned}
\]
On the other hand, on the region $|\tau' - \xi^5| > T^{-1}$, we use the $L^1$ integrability of $\wh{\psi}_T$, so that
\[
\begin{aligned}
\eqref{eq:a.1} &\lesssim T \left(\int \bra{\xi}^{2s}\left|\int_{|\tau'-\xi^5| > 1/T} \frac{|\wt{w}(\tau',\xi)|}{|\tau' - \xi^5|} d\tau' \right|^2 \; d\xi \right)^{\frac12}\\
&\lesssim T^{\frac12 - b} \norm{w}_{X^{s,-b}}.
\end{aligned}
\]

\textbf{(b).}  A direct calculate gives
\begin{equation}\label{eq:duhamel fourier(b)}
\psi_T\partial_x^j\mathcal{D}w(t,x) = c\int e^{ix\xi}e^{it\xi^5}(i\xi)^j\psi_T(t) \int \wt{w}(\tau',\xi) \frac{e^{it(\tau'-\xi^5) }-1}{i(\tau' - \xi^5)} \; d\tau'd\xi.
\end{equation}
We denote by $w = w_1 + w_2$, where 
\[\wt{w}_1(\tau,\xi) = \chi_{\le 1/T}(\tau-\xi^5)\wt{w}(\tau,\xi),\]
for a characteristic function $\chi$.

For $w_1$, we use the  Taylor expansion of $e^x$ at $x =0$. Then, we can rewrite \eqref{eq:duhamel fourier(b)} for $w_1$ as
\[
\begin{aligned}
\psi_T\partial_x^j\mathcal{D}w(t,x) &= c\int e^{ix\xi}e^{it\xi^5}(i\xi)^j\psi_T(t) \int \wt{w}_1(\tau',\xi) \frac{e^{it(\tau'-\xi^5) }- 1}{i(\tau' - \xi^5)} \; d\tau'd\xi\\
&=cT\sum_{k=1}^{\infty}\frac{i^{k-1}}{k!}\psi_T^k(t) \int e^{ix\xi}e^{it\xi^5}(i\xi)^j\wh{F}_1^k(\xi)\; d\xi\\
&=cT\sum_{k=1}^{\infty}\frac{i^{k-1}}{k!}\psi_T^k(t)\partial_x^je^{t\partial_x^5}F_1^k(x),
\end{aligned}
\]
where $\psi^k(t) = t^k\psi(t)$ and 
\[\wh{F}_1^k(\xi) = \int \wt{w}_1(\tau,\xi) (T(\tau-\xi^5))^{k-1} \; d\tau.\]
Since
\[
\begin{aligned}
\norm{F_1^k}_{H^s} &= \left(\int \bra{\xi}^{2s} \left| \int \wt{w}_1(\tau,\xi) (T(\tau-\xi^5))^{k-1} \; d\tau \right|^2 \; d\xi \right)^{\frac12}\\
&\lesssim (1+T^{-\frac12-b})\norm{w}_{X^{s,-b}},
\end{aligned}
\]
we have from Lemma \ref{grupo} (b) that
\[
\begin{aligned}
\norm{\psi_T\partial_x^j\mathcal{D}w}_{L_x^{\infty}H_t^{\frac{s+2-j}{5}}} &\lesssim T\sum_{k=1}^{\infty}\frac{1}{k!}\norm{F_1^k}_{H_x^s}\\
&\lesssim T(1+T^{-\frac12-b})\norm{w}_{X^{s,-b}}\sum_{k=1}^{\infty}\frac{1}{k!}\\ 
&\lesssim T^{\frac12-b}\norm{w}_{X^{s,-b}},
\end{aligned}
\]
when $-\frac92 + j \le s \le \frac12 +j$.

For $w_2$, recall \eqref{eq:duhamel fourier(b)}
\[
\begin{aligned}
\psi_T\partial_x^j\mathcal{D}w(t,x) &= c\int e^{ix\xi}e^{it\xi^5}(i\xi)^j\psi_T(t) \int \frac{\wt{w}(\tau',\xi)}{i(\tau' - \xi^5)} \left(e^{it(\tau'-\xi^5) }- 1\right) \; d\tau'd\xi\\
&= I -II.
\end{aligned}
\]
We first consider $II$. Let 
\[\wh{W}(\xi) = \int \frac{\wt{w}_2(\tau,\xi)}{i(\tau-\xi^5)} \; d\tau.\]
Note that 
\[\norm{W}_{H^s} \lesssim T^{\frac12 - b} \norm{w_2}_{X^{s,-b}}.\]
Then, it immediately follows from 
\[II = \psi_T(t) \partial_x^j e^{t\partial_x^5}F(x)\]
and Lemma \ref{grupo} (b) that
\[\norm{\psi_T \partial_x^j e^{t\partial_x^5}F}_{C_xH^{\frac{s+2-j}{5}}} \lesssim \norm{F}_{H^s} \lesssim T^{\frac12 -b}\norm{w}_{X^{s,-b}},\]
when $-\frac92 + j \le s \le \frac12 +j$.

Now it remains to deal with $I$. Taking the Fourier transform to $I$ with respect to $t$ variable, we have
\[
\int e^{ix\xi} (i\xi)^j \int \frac{\wt{w}_2(\tau',\xi)}{i(\tau'-\xi^5)} \wh{\psi}_T(\tau-\tau')\; d\tau'd\xi,
\]
and hence it suffices to control
\begin{equation}\label{eq:b.8}
\left( \int \bra{\tau}^{\frac{2(s+2-j)}{5}} \left| \int e^{ix\xi} (i\xi)^j \int \frac{\wt{w}_2(\tau',\xi)}{i(\tau'-\xi^5)} \wh{\psi}_T(\tau-\tau')\; d\tau'd\xi \right|^2 \; d\tau \right)^{\frac12}.
\end{equation} 
The argument is very similar used in the proof of Lemma \ref{grupo} (b), while the relation among $|\tau|, |\tau'|$ and $|\xi|^5$ should be taken into account carefully. Hence, we only give, here, a short idea on each case. We first split the region in $\tau$ as follows:
\[\mathbf{Case\; I. }\; |\tau| \le 1, \qquad \mathbf{Case\; II. }\; 1 < |\tau| \le \frac1T, \qquad \mathbf{Case\; III. }\; \frac1T < |\tau|.\]

\textbf{Case I.} $|\tau| \le 1$. In this case, the weight $\bra{\tau}^{\frac{2(s+2-j)}{5}}$ and the integration with respect to $\tau$ can be negligible. If $|\xi|^5 \le 1$, the weight $|\xi|^{j}$ and the integration with respect to $\xi$ is negligible as well. Moreover, we know $|\wh{\psi}_T(\tau - \tau')| \lesssim T^{1-k}|\tau'-\xi^5|^{-k}$ for any positive $k$, since $|\tau' - \xi^5| > 1/T$. Then, the Cauchy-Schwarz inequality gives
\[\eqref{eq:b.8} \lesssim T^{\frac32 -b} \norm{w_2}_{X^{s,-b}}.\]
When $1<|\xi|^5$, we further divide the case into $1 < |\xi|^5 < 1/T$ and $1/T < |\xi|^5$. For the former case, we still have $|\wh{\psi}_T(\tau - \tau')| \lesssim T^{1-k}|\tau'-\xi^5|^{-k}$. The Cauchy-Schwarz inequality gives
\[
\begin{aligned}
\eqref{eq:b.8} &\lesssim T^{1-k}\left( \int_{|\xi|^5 < 1/T}|\xi|^{2j-2s} \int_{|\tau'-\xi^5| > 1/T} |\tau' - \xi^5|^{-2k-2+2b} \;d\xi d\tau  \right)^{\frac12}\norm{w_2}_{X^{s,-b}} \\
&\lesssim T^{\frac32 - b}\max(1, T^{-\frac12 + \frac{s+2-j}{5}}) \norm{w_2}_{X^{s,-b}},
\end{aligned}\]
which enables us to obtain $T^{\theta}$ for positive $\theta>0$ when $5(b-1)-2+j < s$ (roughly $-\frac12 \le \frac{s+2-j}{5}$).

For the latter case, we split the region in $\tau'$ similarly as \eqref{eq:cases} (with corresponding variables $|\xi|^5$ and $|\tau'-\xi^5|$). Then, we can apply the same argument to each case. Indeed, for $|\tau'-\xi^5| < \frac12|\xi|^5$, we can control \eqref{eq:b.8} by using $|\wh{\psi}_T(\tau - \tau')| \lesssim T^{1-k}|\xi|^{-5k}$, while we use $|\wh{\psi}_T(\tau - \tau')| \lesssim T^{1-k}|\tau' - \xi^5|^{-k}$ in the case when $2|\xi|^5 < |\tau'-\xi^5|$. Hence, we have for both cases that
\[\eqref{eq:b.8} \lesssim T^{1 -b + \frac{s+2-j}{5}} \norm{w_2}_{X^{s,-b}}.\]

As seen in the proof of Lemma \ref{grupo} (b), the case when $\frac12|\xi|^5 < |\tau'-\xi^5| < 2|\xi|^5$ is more complicated. Since $|\tau| \le 1$, this case is equivalent to the case when $\frac12|\tau-\xi^5| < |\tau'-\xi^5| < 2|\tau-\xi^5|$. If $(\tau'-\xi^5)\cdot(\tau-\xi^5)<0$, by using the facts that $|\wh{\psi}_T(\tau - \tau')| \lesssim T^{1-k}|\xi|^{-5k}$ and
\[ \int_{|\tau'-\xi^5| \sim |\xi|^5} 1 \; d\tau' \lesssim |\xi|^{5},\]
we obtain
\[\eqref{eq:b.8} \lesssim T^{1 -b + \frac{s+2-j}{5}} \norm{w_2}_{X^{s,-b}}.\]
Otherwise ($(\tau'-\xi^5)\cdot(\tau-\xi^5) > 0$), we use the Hardy-Littlewood maximal function of $|\tau'-\xi^5|^{-b}\wt{w}_2(\tau',\xi)$ for $|\tau - \tau'| < 1/T$, and the smoothness of $\psi$ ($|\wh{\psi}_T(\tau - \tau')| \lesssim |\tau-\tau'|^{-1}$) for $1/T < |\tau - \tau'| < |\tau-\xi^5|$, so that 
\[\eqref{eq:b.8} \lesssim T^{\frac12 -b + \frac{s+2-j}{5}} \norm{w_2}_{X^{s,-b}},\]
which imposes the regularity restriction $0 \le \frac{s+2-j}{5}$. In order to cover $-\frac12 \le \frac{s+2-j}{5} < 0$ regime, we use $Y^{s,b}$ space for the case when $\frac12|\xi|^5 < |\tau'-\xi^5| < 2|\xi|^5$. It suffice to consider
\begin{equation}\label{eq:b.8-1}
\left( \int_{|\tau| \le 1}  \left| \int_{|\xi|^5\le 1/T} |\xi|^j \int_{|\tau' - \xi^5| > 1/T} \frac{\wt{w}_2(\tau',\xi)}{|\tau'-\xi^5|} \wh{\psi}_T(\tau-\tau')\; d\tau'd\xi \right|^2 \; d\tau \right)^{\frac12}.
\end{equation}

We may assume $|\tau'| > 2$, otherwise, we use $\bra{\tau'}^{\frac s5} \sim 1$ and $|\wh{\psi}_T(\tau-\tau')| \lesssim T$ to obtain 
\[\eqref{eq:b.8-1} \lesssim T^{\frac{19}{10}-b-\frac j5} \norm{w_2}_{Y^{s,-b}}.\]
Since $\bra{\tau'} \sim \bra{\tau-\tau'}$, we have from \eqref{eq:L2-T} that 
\[\begin{aligned}
\eqref{eq:b.8-1} &\lesssim \left| \int_{|\xi|^5\le 1/T} |\xi|^{j-5+5b} \int_{|\tau'|> 2} \bra{\tau'}^{\frac s5}\bra{-b}\wt{w}_2(\tau',\xi)\bra{\tau-\tau'}^{-\frac s5}\wh{\psi}_T(\tau-\tau')\; d\tau'd\xi \right|\\
&\lesssim T^{\frac{9}{10}-b-\frac j5}\left( \int_{|\tau| \ge 1} |\tau|^{-\frac{2s}{5}}|\wh{\psi}_T(\tau)|^2 \; d\tau \right)^{\frac12} \norm{w_2}_{Y^{s,-b}}\\
&\lesssim T^{\frac12-b+\frac{2-j}{5}}T^{\frac12 +\frac{s}{5}} \norm{w_2}_{Y^{s,-b}}.
\end{aligned}\]
Thus we cover $-\frac12 \le \frac{s+2-j}{5} < 0$.

\textbf{Case II.} $1 < |\tau| \le \frac1T$. For $|\xi|^5 < 1/T$, we can apply the same argument in \textbf{Case I}, since, roughly speaking, the spare bound $T^{\frac12 + \frac{s+2-j}{5}}$  obtained in \textbf{Case I} controls
\[\left(\int_{|\tau| < 1/T} |\tau|^{\frac{2(s+2-j)}{5}}\; d\tau\right)^{\frac12}. \]
Moreover, the case when $1/T < |\xi|^5$, and $|\tau'-\xi^5| < \frac12|\xi|^5$ or $2|\xi|^5 < |\tau'-\xi^5|$, can be dealt with similarly. Hence, we have for these cases that
\[\eqref{eq:b.8} \lesssim T^{\frac12 -b} \norm{w_2}_{X^{s,-b}}.\]
For the rest case, in view of the proof, we can see that $L^2$ integral with respect to $\tau$ is performed for $w_2$. Moreover, the weight $|\tau|^{\frac{s+2-j}{5}} \lesssim T^{-\frac{s+2-j}{5}}$ (when $0 \le \frac{s+2-j}{5}$), which is killed by the spare bound $T^{\frac{s+2-j}{5}}$. Hence we have for the rest case that
\[\eqref{eq:b.8} \lesssim T^{\frac12 -b} \norm{w_2}_{X^{s,-b}}.\]
Otherwise (when $-\frac12 \le \frac{s+2-j}{5} < 0$), we have similarly as before that
\[\eqref{eq:b.8-1} \lesssim T^{1-b+\frac{s+2-j}{5}} \norm{w_2}_{Y^{s,-b}}.\]

\textbf{Case III.} $1/T < |\tau|$. This case is much more complicated. When $|\xi|^5 < 1/T$, $\xi^5$ is negligible compared with $\tau$ and $\tau'$, and hence \eqref{eq:b.8} is reduced to\footnote{When $\frac{s+2-j}{5}=\frac12$, the constant depending on $T$ is $-\log T $ instead of $T^{\frac{s+2-j}{5}-\frac12}$, but it does not influence on our analysis.}
\begin{equation}\label{eq:b.9}
\max\left(1, T^{\frac{s+2-j}{5}-\frac12} \right)\left( \int \bra{\tau}^{\frac{2(s+2-j)}{5}} \left|  \int \frac{\wt{w}_2^{\ast}(\tau')}{|\tau'|} \wh{\psi}_T(\tau-\tau')\; d\tau' \right|^2 \; d\tau \right)^{\frac12},
\end{equation}
where $\wt{w}_2^{\ast}(\tau') = \norm{\bra{\cdot}\wt{w}_2(\cdot,\tau')}_{L^2}$. Then, the following cases can be treated via the similar way:
\begin{itemize}
\item[III.a] $|\tau'| < \frac12 |\tau|$, in this case we use $|\wh{\psi}_T(\tau-\tau')| \lesssim T^{1-k}|\tau|^{-k}$,
\item[III.b] $2|\tau| < |\tau'|$, in this case we use $|\wh{\psi}_T(\tau-\tau')| \lesssim T^{1-k}|\tau'|^{-k}$,
\item[III.c] $\frac12 |\tau| < |\tau'| < 2 |\tau|$ and $\tau \cdot \tau' < 0$, in this case we use $|\wh{\psi}_T(\tau-\tau')| \lesssim T^{1-k}|\tau|^{-k}$.
\end{itemize}
Indeed, we roughly have
\[\left( \int \bra{\tau}^{\frac{2(s+2-j)}{5}} \left|  \int \frac{\wt{w}_2^{\ast}(\tau')}{|\tau'|} \wh{\psi}_T(\tau-\tau')\; d\tau' \right|^2 \; d\tau \right)^{\frac12} \lesssim T^{1-b-\frac{s+2-j}{5}}\norm{w_2}_{X^{s,-b}},\]
which, together with \eqref{eq:b.9}, implies
\begin{equation}\label{eq:b.9-1}
\eqref{eq:b.8} \lesssim T^{\frac12-b}\norm{w_2}_{X^{s,-b}}.
\end{equation}
On the other hand, when $(s+2-j)/5 > 1/2$ ($s> 0$), we have from \eqref{eq:b.9} that
\[\left( \int \bra{\tau}^{\frac{2(s+2-j)}{5}} \left|  \int \frac{\wt{w}_2^{\ast}(\tau')}{|\tau'|} \wh{\psi}_T(\tau-\tau')\; d\tau' \right|^2 \; d\tau \right)^{\frac12},\]
which does not guarantee \eqref{eq:b.9-1}. However, it is possible to obtain
\[\eqref{eq:b.8} \lesssim T^{\theta} \norm{w_2}_{Y^{s,-b}},\]
for positive $\theta > 0$. Precisely, for III.b and III.c , we have
\[\begin{aligned}
\eqref{eq:b.8} &\lesssim \left( \int_{1/T<|\tau|} \bra{\tau}^{\frac{2(2-j)}{5}} \left| \int_{|\xi|^5 \le 1/T} |\xi|^j \int \frac{\wt{w}_2(\tau',\xi)}{i(\tau'-\xi^5)} \wh{\psi}_T(\tau-\tau')\; d\tau'd\xi \right|^2 \; d\tau \right)^{\frac12}\\
&\lesssim T^{\frac{2-j}{5}-\frac12} \left( \int \bra{\tau}^{\frac{2(2-j)}{5}} \left|  \int \frac{\bra{\tau'}^{\frac{s}{5}}\wt{w}_2(\tau')}{|\tau'|} \wh{\psi}_T(\tau-\tau')\; d\tau' \right|^2 \; d\tau \right)^{\frac12}\\
&\lesssim T^{\frac12-b}\norm{w_2}_{Y^{s,-b}}.
\end{aligned}\] 
For III.a, it follows from
\[\left( \int \bra{\tau}^{\frac{2(s+2-j)}{5}} \left|  \int \bra{\tau'}^{-\frac s5}\frac{\bra{\tau'}^{\frac s5}\wt{w}_2(\tau')}{|\tau'|} \wh{\psi}_T(\tau-\tau')\; d\tau' \right|^2 \; d\tau \right)^{\frac12} \lesssim T^{1-b-\frac{2-j}{5}}\norm{w_2}_{Y^{s,-b}},\]
thanks to
\[\int_{1/T< |\tau'|} |\tau'|^{-2-\frac{2s}{5}+2b} \; d\tau' \lesssim T^{1+\frac{2s}{5}-2b}.\]

For the case when $\tau \cdot \tau' > 0$, we, similarly, split the case into $|\tau - \tau'| < 1/T$ and $|\tau-\tau'|>1/T$. Then, by using the Hardy-Littlewood maximal function of $|\tau'|^{-b}\wt{w}_2^{\ast}(\tau')$ for $|\tau - \tau'| < 1/T$, and the smoothness of $\psi$ ($|\wh{\psi}_T(\tau - \tau')| \lesssim |\tau-\tau'|^{-1}$) for $1/T < |\tau - \tau'| < |\tau|$ similarly as before, we have for the rest case that
\[\eqref{eq:b.8} \lesssim T^{\frac12 -b} \norm{w_2}_{X^{s,-b}}.\]
By the same reason, we have 
\[\eqref{eq:b.8} \lesssim T^{\frac12 -b} \norm{w_2}_{Y^{s,-b}},\]
when $(s+2-j)/5 > 1/2$.

Now we consider the case when $|\xi|^5 > 1/T$. For given $\tau, \xi$, we further divide the case into $|\tau' - \xi^5| \le \frac12 |\tau - \xi^5|$, $2|\tau - \xi^5| \le |\tau' - \xi^5|$ and $\frac12|\tau - \xi^5| < |\tau' - \xi^5| < 2|\tau - \xi^5|$.

For the case when $|\tau' - \xi^5| \le \frac12 |\tau - \xi^5|$, we know $|\tau - \xi^5| > 1/T$ and $|\wh{\psi}_T(\tau - \tau')| \lesssim T^{1-k}|\tau-\xi^5|^{-k}$. Moreover, the region of $\xi$ can be expressed as $\cup_{j=1}^{4}\mathcal{A}_j$, where
\[\mathcal{A}_1 = \left\{\xi : |\xi|^5 > \frac1T, \; 2|\tau| < |\xi|^5 \right\},\] 
\[\mathcal{A}_2 = \left\{\xi : |\xi|^5 > \frac1T, \; |\xi|^5 < \frac12|\tau| \right\},\]
\[\mathcal{A}_3 = \left\{\xi : |\xi|^5 > \frac1T, \; \frac12|\tau| \le |\xi|^5 \le 2|\tau|, \; \tau \cdot \xi^5 < 0 \right\}\]
and 
\[\mathcal{A}_4 = \left\{\xi : |\xi|^5 > \frac1T, \; \frac12|\tau| \le |\xi|^5 \le 2|\tau|, \; \tau \cdot \xi^5 > 0 \right\}.\]
On $\mathcal{A}_1$, we have $|\tau|^{\frac{s+2-j}{5}} \lesssim |\xi|^{s+2-j}$\footnote{This property restricts the regularity condition as $\frac{s+2-j}{5} > 0$. However, in the case when $\frac{s+2-j}{5} > 0$, since $|\tau|^{\frac{s+2-j}{5}} \lesssim T^{-\frac{s+2-j}{5}}$, the same argument yields
\[\eqref{eq:b.8} \lesssim T^{\frac12 - b}\norm{w_2}_{X^{s,-b}}.\]}, $|\tau-\xi^5| \sim |\xi|^5$ and $|\wh{\psi}_T(\tau-\tau')| \lesssim T^{1-k}|\xi|^{-5k}$ for $k > 1$. Let
\[\wt{W}^{\ast}(\tau',\xi) = \bra{\tau'-\xi^5}^{-b}\bra{\xi}^s\wt{w}_2(\tau',\xi).\]
Note that $\norm{W^{\ast}}_{L_{x,t}^2} = \norm{w_2}_{X^{s,-b}}$. Then, we have
\begin{equation}\label{eq:b.10}
\begin{aligned}
\eqref{eq:b.8} &\lesssim T^{1-k}\left( \int_{|\tau| > 1/T}  \left| \int_{\mathcal{A}_1} |\xi|^{2-5k} \int_{1/T < |\tau'-\xi^5| } |\tau'-\xi^5|^{-1+b}\wt{W}^{\ast}(\tau',\xi)\; d\tau'd\xi \right|^2 \; d\tau \right)^{\frac12}\\
&\lesssim T^{1-k}T^{\frac12-b} \left( \int_{|\tau| > 1/T}  |\tau|^{1-2k} \; d\tau \right)^{\frac12} \norm{W^{\ast}}_{L^2}\\
&\lesssim T^{\frac12 - b}\norm{w_2}_{X^{s,-b}}.
\end{aligned}
\end{equation}
On $\mathcal{A}_2$, we have $|\tau-\xi^5| \sim |\tau|$ and $|\wh{\psi}_T(\tau-\tau')| \lesssim T^{1-k}|\tau|^{-k}$ for $k > 1$. Then, similarly as \eqref{eq:b.10}, we have
\[\eqref{eq:b.8} \lesssim T^{\frac12 - b}\norm{w_2}_{X^{s,-b}}.\]
On $\mathcal{A}_3$, since $|\tau-\xi^5| \sim |\tau| \sim |\xi|^5$, we have 
\[\eqref{eq:b.8} \lesssim T^{\frac12 - b}\norm{w_2}_{X^{s,-b}},\] 
similarly as on $\mathcal{A}_1$ or $\mathcal{A}_2$. On $\mathcal{A}_4$, we have $|\tau|^{\frac{s+2-j}{5}} \sim |\xi|^{s+2-j}$ and $|\wh{\psi}_T(\tau-\tau')| \lesssim T^{1-k}|\tau-\xi^5|^{-k}$ for $k > 1$. Moreover, it is enough to consider the region $\tau + \frac1T < \xi^5 < 2\tau$ due to the footnote \ref{fn:4regions} in the proof of Lemma \ref{grupo} (b). Then, we have
\begin{align}
\eqref{eq:b.8} &\lesssim T^{1-k}\left( \int_{\tau > 1/T}  \left| \int_{\tau + 1/T}^{2\tau} \xi^{2}|\tau-\xi^5|^{-k} \int_{1/T < |\tau'-\xi^5| } |\tau'-\xi^5|^{-1+b}\wt{W}^{\ast}(\tau',\xi)\; d\tau'd\xi \right|^2 \; d\tau \right)^{\frac12} \nonumber \\
&\lesssim T^{1-k}T^{\frac12-b} \left( \int_{|\tau| > 1/T}  \left| \int_{\tau + 1/T}^{2\tau} \xi^{2}|\tau-\xi^5|^{-k}\wt{W}^{\ast\ast}(\xi)\; d\xi \right|^2 \; d\tau \right)^{\frac12}, \label{eq:b.11}
\end{align}
where $\wt{W}^{\ast\ast}(\xi) = \norm{\wt{W}^{\ast}(\cdot, \xi)}_{L^2}$. Let $h =  \xi^5 - \tau$. Then the change of variables, the Cauchy-Schwarz inequality and the Fubini theorem yields
\[\begin{aligned}
\eqref{eq:b.11} &\lesssim T^{1-k}T^{\frac12-b} \left( \int_{|\tau| > 1/T}  \left| \int_{1/T}^{\tau} |h|^{-k}\wt{W}^{\ast\ast}((\tau + h)^{\frac15})(\tau + h)^{-\frac25}\; dh \right|^2 \; d\tau \right)^{\frac12}\\
&\lesssim ^{\frac32-k-b + \epsilon}\left( \int_{|\tau| > 1/T} \int_{1/T}^{\tau} |h|^{-2k+1+2\epsilon}|\wt{W}^{\ast\ast}((\tau + h)^{\frac15})|^2(\tau + h)^{-\frac45}\; dh  \; d\tau \right)^{\frac12}\\
&\lesssim T^{\frac12 -b}\norm{w_2}_{X^{s,-b}},
\end{aligned}\]
for small $0<\epsilon \ll 1$, which implies 
\[\eqref{eq:b.8} \lesssim T^{\frac12 - b}\norm{w_2}_{X^{s,-b}}.\]

For the case when $2|\tau - \xi^5| \le |\tau' - \xi^5|$, the region of $\xi$ can be further divided by 
\[\mathcal{B}_1 = \left\{\xi : |\xi|^5 > \frac1T, \; |\tau - \xi^5| < \frac1T \right\}\]
and 
\[\mathcal{B}_2 = \left\{\xi : |\xi|^5 > \frac1T, \; |\tau - \xi^5|  \ge \frac1T  \right\}.\]
On $\mathcal{B}_1$, we know $|\tau|^{\frac{s+2-j}{5}} \sim |\xi|^{s+2-j}$. Since $|\wh{\psi}_T(\tau-\tau')| \lesssim T$ and 
\[\int_{1/T < |\tau'-\xi^5| } |\tau'-\xi^5|^{-1+b}\wt{W}^{\ast}(\tau',\xi)\; d\tau' \lesssim T^{1-2b}\wt{W}^{\ast\ast}(\xi),\]
we have from the change of variable ($\eta = \xi^5$) that
\[
\begin{aligned}
\eqref{eq:b.8} &\lesssim T^{\frac12-b} \left( \int_{|\tau| > 1/T}  \left| T\int_{|\eta - \tau| < 1/T} \wt{W}^{\ast\ast}(\eta^{\frac15})\eta^{-\frac25}\; d\eta \right|^2 \; d\tau \right)^{\frac12}\\
&\lesssim T^{\frac12-b} \left( \int_{|\tau| > 1/T}  |M\wt{W}^{\ast\ast\ast}(\tau)|^2 \; d\tau \right)^{\frac12},
\end{aligned}
\]
where $\wt{W}^{\ast\ast\ast}(\eta) = \wt{W}^{\ast\ast}(\eta^{\frac15})\eta^{-\frac25}$. Note that $\norm{\wt{W}^{\ast\ast\ast}}_{L^2} = c\norm{w_2}_{X^{s,-b}}$. Therefore, we have
\[\eqref{eq:b.8} \lesssim T^{\frac12 - b}\norm{w_2}_{X^{s,-b}}.\]
On $\mathcal{B}_2$, by dividing the region of $\xi$ as $\mathcal{A}_j$, $j=1,2,3,4$, we have similarly 
\[\eqref{eq:b.8} \lesssim T^{\frac12 - b}\norm{w_2}_{X^{s,-b}}.\]

For the rest case ($\frac12|\tau - \xi^5| < |\tau' - \xi^5| < 2|\tau - \xi^5|$), we further divide the region of $\tau'$ as $ \mathcal{C}_1\cup \mathcal{C}_2$, where
\[\mathcal{C}_1 = \left\{\tau' : |\tau'| > 1/T, \;  \frac12|\tau - \xi^5| < |\tau' - \xi^5| < 2|\tau - \xi^5|, \; (\tau'-\xi^5)\cdot(\tau-\xi^5)<0\right\}\]
and
\[\mathcal{C}_2 = \left\{\tau' : |\tau'| > 1/T, \;  \frac12|\tau - \xi^5| < |\tau' - \xi^5| < 2|\tau - \xi^5|, \; (\tau'-\xi^5)\cdot(\tau-\xi^5)>0\right\}.\]
On $\mathcal{C}_1$, since
\[|\wh{\psi}_T(\tau - \tau')| \lesssim T^{1-k}|\tau-\xi^5|^{-k} \sim T^{1-k}|\tau-\xi^5|^{-k},\]
for $k \ge 0$, by dividing the region of $\xi$ as $\mathcal{A}_j$, $j=1,2,3,4$, we have similarly 
\[\eqref{eq:b.8} \lesssim T^{\frac12 - b}\norm{w_2}_{X^{s,-b}}.\]
On the other hand, we further split the set $\mathcal{C}_2$ by
\[\mathcal{C}_{21} = \left\{\tau' : |\tau'| > 1/T, \;  \frac12|\tau - \xi^5| < |\tau' - \xi^5| < 2|\tau - \xi^5|, \; (\tau'-\xi^5)\cdot(\tau-\xi^5)>0,\; |\tau - \tau'| < \frac1T \right\}\] 
and
\[\mathcal{C}_{22} = \left\{\tau' : |\tau'| > 1/T, \;  \frac12|\tau - \xi^5| < |\tau' - \xi^5| < 2|\tau - \xi^5|, \; (\tau'-\xi^5)\cdot(\tau-\xi^5)>0 ,\; |\tau - \tau'| > \frac1T\right\}.\] 
On $\mathcal{C}_{21}$, \eqref{eq:b.8} is reduced by
\begin{equation}\label{eq:b.12}
\left( \int_{|\tau| > 1/T} |\tau|^{\frac{2(s+2-j)}{5}} \left| \int_{|\xi|^5>1/T} |\xi|^{j-s}|\tau-\xi^5|^{-1+b}M\wt{W}^{\ast}(\tau,\xi)d\xi \right|^2 \; d\tau \right)^{\frac12}.
\end{equation}
Then, by dividing the region of $\xi$ in \eqref{eq:b.12} as $\mathcal{A}_{j}$, $j=1,2,3,4$, we have similarly 
\[\eqref{eq:b.8} \lesssim T^{\frac12 - b}\norm{w_2}_{X^{s,-b}}.\]
On $\mathcal{C}_{22}$, we know $|\wh{\psi}_T(\tau-\tau')| \lesssim T^{1-k}|\tau-\tau'|^{-k}$ for $k \ge 0$. Then, \eqref{eq:b.8} is reduced by
\begin{equation}\label{eq:b.13}
\begin{aligned}
T&^{1-k+\epsilon}\Big( \int_{|\tau| > 1/T} |\tau|^{\frac{2(s+2-j)}{5}} \\
& \times \Big| \int_{|\xi|^5>1/T} |\xi|^{j-s}|\tau-\xi^5|^{-1+b}\Big(\int_{\frac1T}^{\tau-\xi^5}|h|^{-2k+1+2\epsilon}|\wt{W}^{\ast}(\tau+h,\xi)|^2 \; dh\Big)^{\frac12} d\xi \Big|^2 \; d\tau \Big)^{\frac12},
\end{aligned}
\end{equation}
for small $0< \epsilon \ll 1$. Then, for $k \gg 1$ large enough, by dividing the region of $\xi$ in \eqref{eq:b.13} as $\mathcal{A}_{j}$, $j=1,2,3,4$, we have similarly 
\[\eqref{eq:b.8} \lesssim T^{\frac12 - b}\norm{w_2}_{X^{s,-b}}.\]
Therefore, we have for $-\frac92 + j\le s \le \frac12 +j$ that
\[\|\psi_T(t) \partial_x^j\mathcal{D}w(x,t)\|_{C(\mathbb{R}_x;H^{\frac{s+2-j}{5}}(\mathbb{R}_t))} \lesssim T^{\theta}\left( \|w\|_{X^{s,-b}} + \|w\|_{X^{Y,-b}} \right),\]
for some $\theta = \theta(s,j,b) > 0$.

\textbf{(c).} Since
\[\ft[\psi_T \mathcal{D}w](\tau,\xi) = \int \wt{w}(\tau',\xi)\frac{\wh{\psi}_T(\tau - \tau') - \wh{\psi}_T(\tau - \xi^5)}{i(\tau'-\xi^5)} \; d\tau',\]
it suffices to show
\begin{equation}\label{eq:c.1}
\left(\int_{|\xi|\le 1}\int\bra{\tau}^{2\alpha} \left| \int \wt{w}(\tau',\xi)\frac{\wh{\psi}_T(\tau - \tau') - \wh{\psi}_T(\tau - \xi^5)}{i(\tau'-\xi^5)} \; d\tau'\right|^2 \; d\tau d\xi\right)^{\frac12} \lesssim T^{\theta}\norm{w}_{X^{s,-b}}
\end{equation}
and
\begin{equation}\label{eq:c.2}
\begin{aligned}
\Big(\int_{|\xi| > 1}|\xi|^{2s}\int\bra{\tau-\xi^5}^{2b} \Big| \int \wt{w}(\tau',\xi)\frac{\wh{\psi}_T(\tau - \tau') - \wh{\psi}_T(\tau - \xi^5)}{i(\tau'-\xi^5)} \; d\tau'\Big|^2& \; d\tau d\xi\Big)^{\frac12}\\
 &\lesssim T^{\theta}\norm{w}_{X^{s,-b}},
\end{aligned}
\end{equation}
for some $\theta = \theta(\alpha,b) > 0$. For $T|\tau'-\xi^5| \le 1$, we use the mean value theorem in order to deal with
\[\frac{|\wh{\psi}_T(\tau - \tau') - \wh{\psi}_T(\tau - \xi^5)|}{|\tau'-\xi^5|} \lesssim T^2|\wt{\psi}'(T(\tau-\xi^5) + \delta)|,\]
for some $|\delta| \le 1$, in the left-hand side of \eqref{eq:c.1} and \eqref{eq:c.2}. Then, since
\[\left(\int \bra{\tau}^{2\sigma} T^4|\wh{\psi}'(T\tau)|^2 \; d\tau \right)^{\frac12} \lesssim T^{\frac32-\sigma}\norm{\bra{\tau}^{\sigma}\wh{\psi}'(\tau)}_{L^2} \]
and
\[\Big(\int\bra{\xi}^{2s}\Big| \int_{|\tau'-\xi^5| \le 1/T} \wt{w}(\tau',\xi) \; d\tau'\Big|^2\;d\xi\Big)^{\frac12} \lesssim T^{-\frac12 - b}\norm{w}_{X^{s,-b}},\]
we have
\[\begin{aligned}
\Big(\int_{|\xi|\le 1}\int\bra{\tau}^{2\alpha} \Big| \int_{|\tau'-\xi^5| \le 1/T} \wt{w}(\tau',\xi)\frac{\wh{\psi}_T(\tau - \tau') - \wh{\psi}_T(\tau - \xi^5)}{i(\tau'-\xi^5)} \; d\tau'\Big|^2& \; d\tau d\xi\Big)^{\frac12} \\
&\lesssim T^{1-\alpha-b}\norm{w}_{X^{s,-b}}
\end{aligned}\]
and
\[\begin{aligned}
\Big(\int_{|\xi| > 1}|\xi|^{2s}\int\bra{\tau-\xi^5}^{2b} \Big| \int_{|\tau'-\xi^5| \le 1/T} \wt{w}(\tau',\xi)\frac{\wh{\psi}_T(\tau - \tau') - \wh{\psi}_T(\tau - \xi^5)}{i(\tau'-\xi^5)} \; d\tau'\Big|^2& \; d\tau d\xi\Big)^{\frac12}\\
 &\lesssim T^{1-2b}\norm{w}_{X^{s,-b}}.
\end{aligned}\]

Otherwise ($T|\tau'-\xi^5| > 1$), since
\[\left(\int \bra{\tau}^{2\sigma} |\wh{\psi}_T(\tau)|^2 \; d\tau \right)^{\frac12} \lesssim T^{\frac12-\sigma}\norm{\psi}_{H^{\sigma}} \]
and
\[\Big(\int_{|\xi| > 1}|\xi|^{2s}\Big| \int_{|\tau'-\xi^5| > 1/T} \frac{\wt{w}(\tau',\xi)}{i(\tau'-\xi^5)} \; d\tau'\Big|^2\;d\xi\Big)^{\frac12} \lesssim T^{\frac12 - b}\norm{w}_{X^{s,-b}},\]
it suffices to show
\begin{equation}\label{eq:c.3}
\left(\int_{|\xi|\le 1}\int\bra{\tau}^{2\alpha} \left| \int_{|\tau'\xi^5|>1/T} \frac{\wt{w}(\tau',\xi)}{i(\tau'-\xi^5)}\wh{\psi}_T(\tau - \tau') \; d\tau'\right|^2 \; d\tau d\xi\right)^{\frac12} \lesssim T^{\theta}\norm{w}_{X^{s,-b}}
\end{equation}
and
\begin{equation}\label{eq:c.4}
\Big(\int_{|\xi| > 1}|\xi|^{2s}\int\bra{\tau-\xi^5}^{2b} \Big| \int_{|\tau'\xi^5|>1/T} \frac{\wt{w}(\tau',\xi)}{i(\tau'-\xi^5)}\wh{\psi}_T(\tau - \tau') \; d\tau'\Big|^2 \; d\tau d\xi\Big)^{\frac12}\lesssim T^{\theta}\norm{w}_{X^{s,-b}},
\end{equation}
for some $\theta > 0$. It follows the similar way used in the proof of (b). In fact, the proofs of \eqref{eq:c.3} and \eqref{eq:c.4}  are much simpler and easier than the proof of (b), since $L^2$ integral with respect to $\xi$ is negligible and hence it is enough to consider the relation between $\tau - \xi^5$ and $\tau'-\xi^5$. Thus, we omit the details and we have
\[\norm{\psi_T \mathcal{D}w}_{D^{\alpha}} \lesssim T^{1-\alpha - b} \norm{w}_{X^{s,-b}}\]
and  
\[\norm{\psi_T \mathcal{D}w}_{X^{s,b}} \lesssim T^{1-2b} \norm{w}_{X^{s,-b}}.\]
\end{proof}

\begin{lemma}\label{edbf}
Let $-\frac52 < s < \frac12$\footnote{The restriction of regularity makes the range of $\lambda$ below non-empty.}.
\begin{itemize}
\item[(a)] \emph{(Space traces)}
For $\max(s-\frac{9}{2}, -4) <\lambda< \min(s+\frac{1}{2}, \frac12)$, we have
		$$\|\psi_T(t)\mathcal{L}_{\pm}^{\lambda}f(t,x)\|_{C\big(\mathbb{R}_t;\,H^s(\mathbb{R}_x)\big)}\leq c \|f\|_{H_0^\frac{s+2}{5}(\mathbb{R}^+)};$$
\item[(b)] \emph{((Derivatives) Time traces)} For $-4+j<\lambda<1+j$, $j=0,1,2$, we have 
%		\begin{equation}\label{eq:(b)0}
\[		\|\psi_T(t)\partial_x^j\mathcal{L}_{\pm}^{\lambda}f(t,x)\|_{C\big(\mathbb{R}_x;\,H_0^{\frac{s+2-j}{5}}(\mathbb{R}_t^+)\big)}\leq c \|f\|_{H_0^\frac{s+2}{5}(\mathbb{R}^+)};\]
%		\end{equation}
\item[(c)] \emph{(Bourgain spaces)} 
For $b < \frac12 < \alpha < 1-b$ and  $\max(s-2, -\frac{13}{2}) <\lambda< \min(s+\frac{1}{2}, \frac12)$, we have
		\[\|\psi_T(t)\mathcal{L}_{\pm}^{\lambda}f(t,x)\|_{X^{s,b}\cap D^{\alpha}}\leq c \|f\|_{H_0^\frac{s+2}{5}(\mathbb{R}^+)}.\]
\end{itemize}

\end{lemma}

\begin{proof}
From the fact $\psi_T(t) = \psi_2(t)\psi_T(t)$ for $0 < T \le 1$ and the definition of $\mathcal{L}_{\pm}^{\lambda}$, it suffices to consider  
\[\psi_2(t)\mathcal{L}_{\pm}^{\lambda}(\psi_Tf(t,x))\]
instead of 
\[\psi_T(t)\mathcal{L}_{\pm}^{\lambda}f(t,x).\]
Then, by Lemma 4.3 in \cite{CK2018-1} and Remark \ref{rem:cut} (a variant of Lemma \ref{cut}), we have Lemma \ref{edbf}. We omit the details.
\end{proof}

\section{Proof of Theorems \ref{theorem1} -- \ref{theorem22}}\label{sec:thm proof}
The proofs of Theorems \ref{theorem1} -- \ref{theorem22} are based on the argument in \cite{CK2018-1}, while the scaling argument does not hold here as mentioned in Section \ref{sec:intro}. Hence, we only provide a sketch of the proof of Theorem \ref{theorem1}.

We fix $-\frac52<s<\frac12$. Recall from \cite{CK2018-1}
\[a_j = \frac{1}{5B(0)\Gamma\left(\frac45\right)}\frac{\cos\left(\frac{(1+4\lambda_j)\pi}{10}\right)}{\sin\left(\frac{(1-\lambda_j)\pi}{5}\right)}\quad \text{and}\quad b_j = \frac{1}{5B(0)\Gamma\left(\frac45\right)}\frac{\cos\left(\frac{(4\lambda_j-3)\pi}{10}\right)}{\sin\left(\frac{(2-\lambda_j)\pi}{5}\right)}\]
and define a matrix
\[A(\lambda_1,\lambda_2)=
\left[\begin{array}{cc}
a_1 & a_2 \\
	b_1 & b_2 \end{array} \right].\]\
We note that when $-\frac52<s<\frac12$, the parameters $\lambda_1$ and $\lambda_2$ satisfying
\begin{equation}\label{eq:lambda condition}
\max(s-2,\ -3)<\lambda_j<\min\left(\frac12,s+\frac12\right),\; j=1,2,
\end{equation}
and
\begin{equation}\label{eq:lambda condition1}
\lambda_1 - \lambda_2 \neq 5n, \qquad n \in \Z,
\end{equation}
facilitate that  Lemma \ref{edbf} holds and $A$ is invertible.

We fix $\lambda_j$, $j=1,2$, satisfying \eqref{eq:lambda condition} and \eqref{eq:lambda condition1}. We bring the solution operator on $[0,T]$ from \cite{CK2018-1} as follows:
\begin{equation}\label{eq:solution map}
\Lambda u(t,x)= \psi_T(t)\mathcal{L}_+^{\lambda_1}\gamma_1(t,x)+\psi_T(t)\mathcal{L}_+^{\lambda_2}\gamma_2(t,x)+\psi_T(t)F(t,x),
\end{equation}
where 
\[\left[\begin{array}{c}
\gamma_1(t)\\
\gamma_2(t)   \end{array} \right]=A^{-1}\left[\begin{array}{c}
f(t)-F(t,0) \\
\mathcal{I}_{\frac15}g(t)-  \mathcal{I}_{\frac15} \partial_xF(t,0)  \end{array} \right],\]
and $F(t,x)=e^{it\partial_x^5}\underline{u}_0-\mathcal{D}((1-\partial_x^2)^{\frac12}\partial_x(u^2))(t,x)$. 

For given initial and boundary data $u_0, f$ and $g$, we fix $0<T<1$ such that
\begin{equation}\label{eq:T}
4C^2T^{\frac32-2\alpha-b}(\|u_0\|_{H^s(\R^+)}+\|f\|_{H^{\frac{s+2}{5}}(\R^+)}+\|g\|_{H^{\frac{s+1}{5}(\R^+)}}) < \frac12,
\end{equation}
where $C$ is the maximum constant among other implicit constants appeared in all estimates in Sections \ref{sec:energy} and \ref{sec:nonlinear}. Note that $\frac12 < \frac34 -\frac b2 < 1-b$ holds for $b < \frac12$, and hence it is possible to choose a small $T>0$ satisfying \eqref{eq:T}, since $\frac32-2\alpha-b > 0$ when $b < \frac12 < \alpha < \frac34 -\frac b2$.

Recall the $Z_1^{s,\alpha,b}$-norm defined in \eqref{eq:solution space}. All estimates obtained in Sections \ref{sec:energy} and \ref{sec:nonlinear} yield
\[\|\Lambda u\|_{Z_1^{s,\alpha,b}}\leq CT^{\frac12 -\alpha}(\|u_0\|_{H^s(\R^+)}+\|f\|_{H^{\frac{s+2}{5}}(\R^+)}+\|g\|_{H^{\frac{s+1}{5}(\R^+)}} ) + CT^{1-\alpha-b}\norm{u}_{Z_1^{s,\alpha,b}}^2.\]
Similarly,
\[\|\Lambda u_1 - \Lambda u_2\|_{Z_1^{s,\alpha,b}}\leq CT^{1-\alpha-b}(\norm{u_1}_{Z_1^{s,\alpha,b}}+\norm{u_2}_{Z_1^{s,\alpha,b}})\norm{u_1-u_2}_{Z_1^{s,\alpha,b}},\]
for $u_1(0,x) = u_2(0,x)$. These immediately imply that $\Lambda$ is a contraction map on 
\[\set{u \in Z_1^{s,\alpha,b} : \norm{u}_{Z_1^{s,\alpha,b}} < 2CT^{\frac12 -\alpha}(\|u_0\|_{H^s(\R^+)}+\|f\|_{H^{\frac{s+2}{5}}(\R^+)}+\|g\|_{H^{\frac{s+1}{5}(\R^+)}} ) },\] and it completes the proof.

\end{document}